\documentclass[preprint,12pt]{elsarticle}

\usepackage{epsf,amsmath,amsfonts,amsthm,amssymb}
\usepackage{epsfig,mathtools}
\usepackage{color,graphicx}% Include figure files
\usepackage{dcolumn}% Align table columns on decimal point
\usepackage{bm,enumerate}% bold math
\usepackage{epsfig,latexsym}
\usepackage{lineno,newtxtext,nicefrac,bigints}
\usepackage{dashrule,hyperref}
\usepackage{mathptmx}      % use Times fonts if available on your TeX system
\newtheorem{theorem}{Theorem}[section]

\usepackage[skip=0pt]{caption}
\usepackage[skip=0pt]{subcaption}

\newcommand{\bseq}{\begin{subequations}}
\newcommand{\eseq}{\end{subequations}}
\newcommand{\beq}{\begin{equation}}
\newcommand{\eeq}{\end{equation}}
\newcommand{\bef}{\begin{figure}}
\newcommand{\eef}{\end{figure}}
\journal{Journal of Computational Physics}

\begin{document}

\begin{frontmatter}



\title{Implicit-explicit time integration method for fractional advection-reaction-diffusion equations}


\author[inst1]{D. Ghosh}

\affiliation[inst1]{organization={Department of Mathematics},%Department and Organization
            addressline={IIIT Delhi},
            postcode={110020},
            country = {India}}

\author[inst1]{T. Chauhan}
\author[inst1]{S. Sircar \corref{cor1}}
\cortext[cor1]{Corresponding Author: sarthok@iiitd.ac.in}

\begin{abstract}
We propose a novel family of asymptotically stable, implicit-explicit, adaptive, time integration method (denoted with the $\theta$-method) for the solution of the fractional advection-diffusion-reaction (FADR) equations. This family of time integration method generalized the computationally explicit $L_1$-method adopted by Brunner (J. Comput. Phys. {\bf 229} 6613-6622 (2010)) as well as the fully implicit method proposed by Jannelli (Comm. Nonlin. Sci. Num. Sim., {\bf 105}, 106073 (2022)). The spectral analysis of the method (involving the group velocity and the phase speed) indicates a region of favorable dispersion for a limited range of Peclet number. The numerical inversion of the coefficient matrix is avoided by exploiting the sparse structure of the matrix in the iterative solver for the Poisson equation. The accuracy and the efficacy of the method is benchmarked using (a) the two-dimensional (2D) fractional diffusion equation, originally proposed by Brunner, and (b) the incompressible, subdiffusive dynamics of a planar viscoelastic channel flow of the Rouse chain melts (FADR equation with fractional time-derivative of order $\alpha=\nicefrac{1}{2}$) and the Zimm chain solution ($\alpha=\nicefrac{2}{3}$). Numerical simulations of the viscoelastic channel flow effectively capture the non-homogeneous regions of high viscosity at low fluid inertia (or the so-called `spatiotemporal macrostructures'), experimentally observed in the flow-instability transition of subdiffusive flows.
\end{abstract}

\begin{keyword}
Caputo derivative \sep IMEX time integration \sep Upwind difference scheme \sep Rouse polymer melt \sep Zimm chain solution
% %% MSC codes here, in the form: \MSC code \sep code
% %% or \MSC[2008] code \sep code (2000 is the default)
\end{keyword}

\end{frontmatter}


%%%%%%%%%%%%%%%%%%%%%%%%%%%%%%%%%%%%%%%%%%%%%%%%%%%%
\section{Introduction} \label{sec:intro}
Fractional partial differential equations (FPDE) have emerged as a powerful tool for modeling the multiphysics and multiscale processes in numerical simulations ranging from physics~\cite{Goychuk2017,Goychuk2020,Goychuk2021}, biology~\cite{Lai2009} to quantitative finance~\cite{Coffey2004}. For example, some of the most significant and profoundly published experimental results in random flow environments, such as the cytosol and the plasma membrane of biological cells~\cite{Rubenstein2003}, crowded complex fluids and polymer solutions~\cite{Levine2001}, dense colloidal suspensions~\cite{Kremer1990} and single-file diffusion in colloidal systems~\cite{Kou2004}, are better rationalized within the fractional calculus framework. The growing number of applications of fractional derivatives in various fields of science and engineering indicate that there is a significant demand for better numerical algorithms. However, such algorithms are predominantly designed for one-dimensional (1D) problems~\cite{Murio2008}, due to the severe memory restrictions imposed by these derivatives, a challenge which we alleviate in Section~\ref{subsec:AdaptiveTime}.

The individual physics or scale components in FPDE have very different properties that are reflected in their discretization, for example, in Fractional Advection Diffusion Reaction (FADR) systems~\cite{Jannelli2022}, the discrete advection has a relatively slow (or `nonstiff') dynamics while the diffusion is a fast (or `stiff') evolving process. Implicit-Explicit (IMEX) integrators have been proposed as an attractive alternative (compared with the fully explicit or fully implicit time integration methods) where one combines the explicit (implicit) integration for the slow (fast) scale~\cite{Crouzeix1980}. In an IMEX scheme the system of equations assume the form~\cite{Ascher1995},
\beq
\frac{\partial^\alpha \Phi}{\partial t^\alpha} = {\bf f}(\Phi) + \eta {\bf g}(\Phi),
\label{eqn:IMEX}
\eeq
where the superscript, $\alpha$, denotes order of the fractional derivative and $\eta$ is a nonnegative parameter. In Equation~\eqref{eqn:IMEX}, the terms collected in ${\bf f}(\Phi)$ are on a slow time scale. Because they are (possibly) nonlinear, the implementation of a fully implicit integration scheme faces performance challenges, either from a poor performance of iterative solvers or from a complex Jacobian matrix associated with the problem. Therefore, it makes sense to solve this term explicitly. The terms in ${\bf g}(\Phi)$, however, are on a fast time scale, and their explicit solution may require excessively small time steps in order to maintain numerical stability. Being usually linear in nature, they can be solved implicitly without further complications. The stability of such methods is still bounded by the Courant-Friedrichs-Lewy (CFL) condition, but because the fast terms are treated implicitly, these conditions are less strict when compared to a fully explicit scheme of similar formal order of accuracy. A class of such an asymptotically stable IMEX method is introduced in Section~\ref{sec:method}.

We have limited our focus in this work, on the development, analysis and applicability of a novel family of spatiotemporal discretization method for the numerical solution of the 1D and 2D FADR systems. In the knowledge of the present authors, a detailed analysis of the numerical methods for FADR equation, in the finite difference framework, is absent. Such analysis is available for the fraction diffusion equation~\cite{Brunner2010} and the advection-diffusion equation~\cite{Khaled2005}. Section~\ref{sec:method} describes this method. Using the 1D linear FADR equation, the time-asymptotic stability analysis and the spectral analysis for the method is outlined in Section~\ref{sec:anal}. The numerical method is benchmarked using the test cases for the 2D fraction diffusion equation, proposed by Brunner~\cite{Brunner2010}, in Section~\ref{sec:MV}. The numerical results of the subdiffusive dynamics of the viscoelastic channel flow is delineated in section~\ref{sec:NS}. Section~\ref{sec:conclude} concludes with a brief discussion of the implication of these results in future studies.

%%%%%%%%%%%%%%%%%%%%%%%%%%%%%%%%%%%%%%%%%%%%%%%%%%%%
\section{Methodology} \label{sec:method}
Let $\Gamma$ be a bounded domain in $\mathbb{R}^2$ with sufficiently smooth boundary $\partial \Gamma$. We present the numerical method for an initial-boundary value problem for the time-dependent FADR equation with fractional time-derivative of order $\alpha \in (0,\,1)$, as follows,
\begin{align}
&\frac{\partial^\alpha u({\bf x}, t)}{\partial t^\alpha} + K_1({\bf x}, t) \nabla \cdot u({\bf x}, t) = K_2({\bf x}, t) \nabla^2 u({\bf x}, t) + f({\bf x}, t), \quad {\bf x} \in \Gamma, \quad t \in (0, T), \nonumber \\
& u({\bf x}, 0) = u_0({\bf x}), \quad {\bf x} \in \Gamma, \nonumber \\
& \partial u({\bf x}, 0) = g_0({\bf x}), \quad {\bf x} \in \partial \Gamma, \quad t \in (0, T),
\label{eqn:method1}
\end{align}
where $\frac{\partial^\alpha u({\bf x}, t)}{\partial t^\alpha}$ denotes the Caputo fractional derivative of order $\alpha$~\cite{Podlubny1999} with respect to $t$ defined by
\beq
\frac{\partial^\alpha u({\bf x}, t)}{\partial t^\alpha} = \frac{1}{\Gamma(1-\alpha)} \int^t_0 \frac{d t'}{(t-t')^\alpha} \frac{\partial u({\bf x}, t)}{\partial t'}, \quad 0 < \alpha < 1,
\label{eqn:method2}
\eeq
and the operators $\nabla(\cdot)$ and $\nabla^2(\cdot)$ in equation~\eqref{eqn:method1}, are the (integer order) gradient and the Laplacian operator in $\mathbb{R}^2$. The FADR equation is related with the non-Markovian continuous-time random walk and is a model for anomalous diffusional flow-fields such as polymer melts~\cite{Chauhan2021}, flows of liquid crystals~\cite{Sircar2010} as well as biological flows including mucus~\cite{Sircar2016} and cartilage~\cite{Sircar2015}. In the next three sections, we propose the numerical method for the spatiotemporal discretization of equation~\eqref{eqn:method1}.

%%%%%%%%%%%%%%%%%%%%%%%%%%%%%%%%%%%%%%%%%%%%%%%%%%%%
\subsection{Time integration} \label{subsec:Time}
The Caputo fractional time derivative in equations~(\ref{eqn:method1}, \ref{eqn:method2}) is discretized using the difference approximation, discussed in~\cite{Podlubny1999}. Suppose the time interval [0, $T$] is discretized uniformly into $n$ subintervals, define $t_k = k \Delta t, \,\, k = 0, 1,\ldots,n$, where $\Delta t  = T/n$ is the time-step. Let $u({\bf x},\,t_k)$ be the exact value of a function $u({\bf x}, t)$ at time step $t_k$. Then, the fractional derivative can be approximated as follows,
\begin{align}
%^cD^\alpha_t u({\bf x}, t) 
\frac{\partial^\alpha u({\bf x}, t)}{\partial t^\alpha} &\approx \frac{1}{\Gamma(1-\alpha)} \sum^k_{j=0} \frac{u({\bf x},\,t_{j+1})-u({\bf x},\,t_j)}{\Delta t} \int^{(j+1)\Delta t}_{j\Delta t} \frac{d t'}{(t_{k+1}-t')^\alpha} \nonumber \\
&=\frac{1}{\Gamma(1-\alpha)} \sum^k_{j=0} \frac{u({\bf x},\,t_{j+1})-u({\bf x},\,t_j)}{\Delta t} \int^{(k-j)\Delta t}_{(k-j+1)\Delta t} t'^{-\alpha} d t' \nonumber \\
&=\frac{1}{\Gamma(1-\alpha)} \sum^k_{j=0} \frac{u({\bf x},\,t_{k+1-j})-u({\bf x},\,t_{k-j})}{\Delta t} \int^{(j+1)\Delta t}_{j\Delta t} t'^{-\alpha} d t' \nonumber \\
&=\frac{(\Delta t)^{1-\alpha}}{\Gamma(2-\alpha)} \sum^k_{j=0} \frac{u({\bf x},\,t_{k+1-j})-u({\bf x},\,t_{k-j})}{\Delta t} r^\alpha_j,
\end{align}
where the weight, $r^\alpha_j=\left[ (j+1)^{1-\alpha} - j^{1-\alpha} \right]$. Following an earlier work by the authors which utilized a standard integer-order advection-reaction-diffusion (ADR) equations~\cite{Sircar2020}, we retain `IMEX method philosophy' by explicitly discretizing the advection and the reaction term (i.~e., `$K_1\nabla \cdot u({\bf x}, t)$' term and `$f({\bf x}, t)$' term, respectively in equation~\eqref{eqn:method1}), while the diffusive term (i.~e., `$K_2({\bf x}, t) \nabla^2 u({\bf x}, t)$' term in equation~\eqref{eqn:method1}) is treated semi-implicitly, as follows,
\beq
K_2({\bf x}, t) \nabla^2 u({\bf x}, t) \approx \theta K_2({\bf x}, t) \nabla^2 u({\bf x}, t)|_{t_k} + (1-\theta) K_2({\bf x}, t) \nabla^2 u({\bf x}, t)|_{t_{k-1}}.
\label{eqn:method3}
\eeq
This family of time integration method (referred to as the $\theta-$method in subsequent discussion), generalizes the computationally explicit $L1$-method by Brunner~\cite{Brunner2010} as well as the fully implicit method recently proposed by Jannelli~\cite{Jannelli2022}.

%%%%%%%%%%%%%%%%%%%%%%%%%%%%%%%%%%%%%%%%%%%%%%%%%%%%
\subsection{Adaptive time stepping} \label{subsec:AdaptiveTime}
When the simulation time is long, the size of `memory' in the fractional derivative approximation becomes enormously large. However, according to the `fading memory property'~\cite{Diethelm2006}, for long times, the solution of the FADR systems change more slowly than the standard integer-order ADR processes. Hence, it makes sense to employ a large time step at longer times. Let $\widetilde{u(\cdot,\, t_k)}$ be the numerical approximation for $u(\cdot,\, t_k)$. To detect this change, we define a measure between the numerical solutions of two consecutive time steps by,
\beq
\Delta u_{t_k} = \frac{\| \widetilde{u(\cdot,\, t_k)} - \widetilde{u(\cdot,\, t_{k-1})} \|_{l^2}}{|| \widetilde{u(\cdot,\, t_{k-1}) ||_{l^2}}}, \quad k = 1, \ldots, n.
\eeq
For some user-defined relaxation parameter $\delta$, if $\Delta u_{t_k} < \delta$, then the time spacing is geometrically increased (i.~e., $\Delta t \rightarrow 2\Delta t$), upto some prefixed value $\Delta t_{\text{max}}$.
%%%%%%%%%%%%%%%%%%%%%%%%%%%%%%%%%%%%%%%%%%%%%%%%%%%%
\subsection{Spatial approximation} \label{subsec:Space}
The gradients and Laplacian in equation~\eqref{eqn:method1} are spatially approximated using the upwind difference scheme (UDS) and the second order central difference scheme (CDS), respectively. Recall that the CDS approximation of the convective terms in equation~\eqref{eqn:method1} does not model the flow-physics accurately~\cite{Sircar2020}. Furthermore, a first order upwind scheme is too diffusive, thereby necessitating the use of higher order upwind schemes. For example,
\beq
K_1({\bf x}, t) \frac{\partial u}{\partial x} \approx {K_1}_{ij} \left( \frac{u^k_{i+1,j}-u^k_{i-1,j}}{2\Delta x} \right) + q(K^+_1 u^-_x + K^-_1 u^+_x),
\eeq
where
\begin{align}
K^-_1 = \text{min}({K_1}_{ij}, 0) \quad &K^+_1 = \text{max}({K_1}_{ij}, 0), \nonumber\\
u^-_x = \frac{u^k_{i-2,j} - 3u^k_{i-1,j} + 3u^k_{i,j} - u^k_{i+1,j}}{3 \Delta x} \quad & u^+_x = \frac{u^k_{i-1,j} - 3u^k_{i,j} + 3u^k_{i+1,j} - u^k_{i+2,j}}{3 \Delta x},
\label{eqn:UD3}
\end{align}
where the parameter $q=0.5$ represents the third-order accurate upwind formula (UD3) and which is utilized for the interior stencil points. The use of ghost points are avoided by setting $q = 0$ for grid points immediately adjacent to the boundary, leading to a second order accurate method (UD2) at these points. Since the focus of the present work is on the development of the time integration method, we retain same the spatial approximation, outlined above, in all the subsequent test cases. 

%%%%%%%%%%%%%%%%%%%%%%%%%%%%%%%%%%%%%%%%%%%%%%%%%%%%
\section{Asymptotic stability and spectral analysis: linear 1D FADR equation} \label{sec:anal}
The linear 1D FADR equation~\eqref{eqn:method1} (with $K_1=c, K_2=\gamma, f(x, t) = \lambda u(x, t)$ for $x \in (-\infty, \infty)$, where $c, \gamma$ and $\lambda$ are constants specifying the advection speed, coefficient of diffusion and coefficient of reaction, respectively) serves as a model for FPDE replicating multi-scale processes. 

%%%%%%%%%%%%%%%%%%%%%%%%%%%%%%%%%%%%%%%%%%%%%%%%%%%%
\subsection{Asymptotic stability analysis} \label{subsec: ASA}
First, we show that the solution of the linear 1D FADR equation, coupled with periodic boundary conditions and discretized using the numerical method outlined in Section~\ref{subsec:Time}-\ref{subsec:Space}, is asymptotically stable for a limited range of CFL and Peclet numbers. The discretization of the linear 1D FADR equation, using the $\theta$-method, is as follows,
\begin{align}
&\frac{(\Delta t)^{-\alpha}}{\Gamma{(2-\alpha)}}\left\{ \sum_{j=1}^n \left(j^{(1-\alpha)}-(j-1)^{(1-\alpha)}\right) \left(u_i^{n-j+1}-u_i^{n-j}\right)\right\} + c \left(\frac{u_{i+1}^{n-1}-u_{i-1}^{n-1}}{2\Delta x}\right) + \nonumber\\
& qc \left(\frac{u_{i-2}^{n-1}-3u_{i-1}^{n-1}+3u_{i}^{n-1}-u_{i+1}^{n-1}}{3\Delta x}\right)
\nonumber \\
& = \gamma \theta \frac{u_{i+1}^{n}-2u_{i}^{n}+u_{i-1}^{n}}{(\Delta x)^2} + \gamma(1-\theta)\frac{u_{i+1}^{n-1}-2u_{i}^{n-1} + u_{i-1}^{n-1}}{(\Delta x)^2} + \lambda u_{i}^{n-1},
\label{eqn:step2}
 \end{align}
where, $n / i$, denote the temporal / spatial index and we set the parameter, $q=0.5$. Introducing the following non-dimensional parameters: the CFL number, $N_c$, Peclet number, $Pe$ and Damk\"{o}hler number, $Da$, where
\beq
N_c = \frac{c (\Delta t)^\alpha}{\Delta x}, \quad Pe = \frac{\gamma (\Delta t)^\alpha}{(\Delta x)^2}, \quad Da = \frac{\lambda (\Delta x)}{c},
\eeq
and rearranging the terms in equation~\eqref{eqn:step2}, we arrive at the following equation,
\begin{align}
&\left(1\!+\! 2 Pe \theta \Gamma{(2\!-\! \alpha)}\right) u_i^n  \! -\!   Pe \theta \Gamma{(2\!-\!\alpha)}u_{i+1}^n \! -\!
Pe\theta\Gamma{(2\!-\! \alpha)}u_{i-1}^n   =   \left(1-qN_c\Gamma{(2\!-\! \alpha)}
\right. \nonumber\\
& \left. -2 Pe(1\!-\!\theta)\Gamma{(2\!-\!\alpha)} + Da N_c\Gamma{(2-\alpha)}\right)u_i^{n-1}  +  \bigg(-N_c\frac{\Gamma{(2-\alpha)}}{2}+q N_c\frac{\Gamma{(2-\alpha)}}{3} 
 \nonumber \\
& + Pe(1\!-\!\theta)\Gamma{(2\!-\!\alpha)}\bigg)u_{i+1}^{n-1} \!+\! \bigg(N_c\frac{\Gamma{(2\!-\!\alpha)}}{2} + q N_c\Gamma{(2\!-\!\alpha)} \! +\! Pe (1-\theta)\Gamma{(2-\alpha)}\!\! \bigg)
\nonumber \\
&u_{i-1}^{n-1}- q N_c\frac{\Gamma{(2-\alpha)}}{3}u_{i-2}^{n-1} + \sum_{j=2}^n \left(j^{(1-\alpha)}-(j-1)^{(1-\alpha)}\right) \left(u_i^{n-j+1}  -  u_i^{n-j}\right).
\label{eqn:step3}
\end{align}
If $\widetilde{u_i^n}$ be the (numerically) approximate solution of equation~\eqref{eqn:step3}, then the round-off error, $\varepsilon_i^n = u_i^n - \widetilde{u_i^n}\quad (i = 0, \ldots, M)$, identically satisfies the same equation. Due to periodic boundary conditions, we have,
\beq
\varepsilon_0^n=\varepsilon_M^n \quad n = 1,2,......,T.
     \label{eqn:step4}
\eeq
Assuming the round-off error has the following form,
\begin{align}
    \varepsilon_i^n = \xi_n e^{Ik(ih)},
    \label{eqn:step5}
\end{align}
where $\xi_n = |\varepsilon_i^n|, I=\sqrt{-1}, h=\Delta x$ and $k = \frac{2\pi l}{L}$ (l: index, and L: Spatial domain length). We substitute the above relation in the equation for the round-off error, to arrive at the following form,
\beq
\mu_1 \xi_n = \mu_2 \xi_{n-1} +\sum_{j=2}^n r^\alpha_j \left(\xi_{n-j+1}-\xi_{n-j}\right),
 \label{eqn:step6}
\eeq
where
\begin{align}
&\mu_1 = 1+2Pe\theta\Gamma(2-\alpha)(1-\cos(kh)), \nonumber \\
&\mu_2 = 1  -  \left(  N_cq \left(1  -  \frac{2}{3}\cos(kh)\right) + 2 Pe(1  -  \theta)(1  -  \cos(kh))  -  DaN_c  +   IN_c\sin(kh)  
\right. \nonumber \\
&\left. \qquad -  \frac{N_c q}{3} \left(2e^{-Ikh}  -  e^{-2Ikh}\right)   \right)\Gamma(2  -  \alpha), \nonumber \\
&r^\alpha_j = j^{(1-\alpha)}-(j-1)^{(1-\alpha)}.
\label{eqn:step7}
\end{align}
Next we prove the result that the finite difference scheme~(\ref{eqn:step6}-\ref{eqn:step7}) is asymptotically stable, in the form of the following theorem.
\begin{theorem}
The approximate solution to the 1D linear FADR equation using the $\theta$-method~\eqref{eqn:step6} with $\alpha \in (0,\,1)$, on the finite domain $x \in [-L,\, L]$ with periodic boundary conditions is asymptotically stable for all $t \ge 0$.
\label{thm:IMEX}
\end{theorem}
\begin{proof}
It suffices for us to show that the time-amplification factor, $\xi_n$ (equation~\eqref{eqn:step6}), obey the inequality,
\beq
\xi_n \le \xi_{n-1} \le \xi_{n-2} \le \ldots \le \xi_1 \le \xi_0.
\label{eqn:step8}
\eeq
\begin{itemize}
\item For $n=1$ in equation~(\ref{eqn:step6},\ref{eqn:step7}), we observe the $\mu_1 \ge 1$ while $\mu_2 \le 1$ for sufficiently small $N_c$, and $Pe$, thereby leading us towards the conclusion that $\xi_1 \le \xi_0$.

\item Using induction hypothesis, we assume that $\xi_{n-1} \le \xi_{n-2} \le \ldots \le \xi_1 \le \xi_0$. Next we show that $\xi_n \le \xi_{n-1}$.

\item Since $\mu_1 \ge 1$ and $\mu_2 \le 1$, equation~\eqref{eqn:step6} can be replaced with the following inequality,
\beq
\xi_n \le \xi_{n-1} +\sum_{j=2}^n r^\alpha_j \left(\xi_{n-j+1}-\xi_{n-j}\right) \le \xi_{n-1}.
\label{eqn:step9}
\eeq
The second inequality in equation~\eqref{eqn:step9} is because,
\begin{enumerate}[i)]
\item $r^\alpha_j$ is positive and 
\item $r^\alpha_j > r^\alpha_{j+1}$, and
\item $\xi_{n-j+1}-\xi_{n-j} \le 0$ (via the induction hypothesis),
\end{enumerate}
which completes the proof.
\end{itemize}
\end{proof}
We emphasize that while the $\theta$-method is not unconditionally stable, it is time-asymptotically stable for a restricted range of $N_c$ and $Pe$.

%%%%%%%%%%%%%%%%%%%%%%%%%%%%%%%%%%%%%%%%%%%%%%%%%%%%
\subsection{Spectral analysis} \label{subsec:SA}
Although the $\theta$-method is asymptotically stable, the presence of dispersion errors (through negative group velocity and large phase speed errors) would invalidate the long time integration results~\cite{Sircar2006}. Hence we couple the result in Section~\ref{subsec: ASA} along with the toolset of spectral analysis to deduce the relevant range of the parameters, $N_c, Pe$ and $Da$, for an accurate representation of the numerical solution of the 1D FADR equation.

Using the spectral (Fourier) representation of the approximate solution of equation~\eqref{eqn:step3}, we have $\widetilde{u_i^n} = \xi'_n e^{I(ikh)}\,\,\,(I=\sqrt{-1},\, k=\frac{2\pi l}{L})$. We define the numerical time-amplification factor, $G_{num} = \frac{\widetilde{u_i^n}}{\widetilde{u_i^{n-1}}} = \frac{\xi'_n}{\xi'_{n-1}}$. Dividing equation~\eqref{eqn:step3} with $\widetilde{u_i^{n-1}}$, we arrive at the equation governing $G_{num}$,
\beq
C_0 G_{num}^n + C_1 G_{num}^{n-1}+ \sum_{j=2}^n r_j^\alpha \left(G_{num}^{n-j+1}-G_{num}^{n-j}\right) = 0,
\label{eqn:step10}
\eeq
where the coefficients,
\begin{align}
C_0 = &1+2Pe\theta \Gamma{(2-\alpha)}(1-\cos(kh)) \nonumber\\
C_1 = &-1+qN_c\Gamma(2-\alpha)\left(1-\frac{4}{3}\cos(kh)+\frac{1}{3}\cos(2kh)\right)\!+\!2Pe(1-\theta)\Gamma(2-\alpha)(1-
\nonumber \\
&\cos(kh)) \!-\! Da N_c\Gamma(2\!-\! \alpha) \!+\!IN_c\Gamma(2-\alpha)\left(\!\sin(kh)+\frac{2}{3}q\sin(kh)-\frac{q}{3}\sin(2kh)\right).
\label{eqn:step11}
\end{align}
We remark that the order `n' of the polynomial~\eqref{eqn:step10} is fixed at $n=75$ in subsequent analysis, since our numerical studies have shown that an increase in the polynomial order by one, shifts the contours of the spectral variables by less that $0.001\%$.

\begin{figure}[htbp]
\includegraphics[width=0.33\linewidth, height=0.3\linewidth]{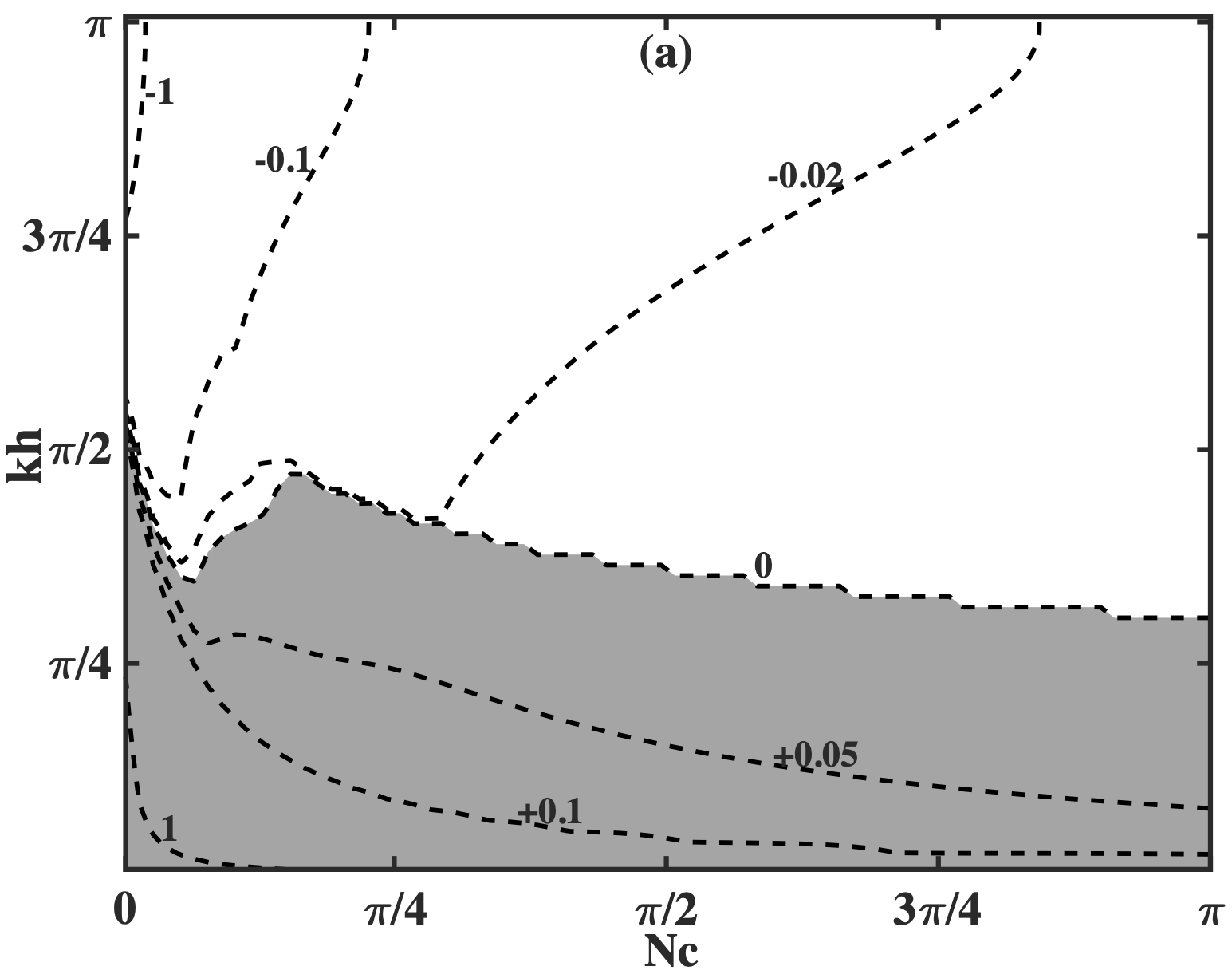}
\includegraphics[width=0.33\linewidth, height=0.3\linewidth]{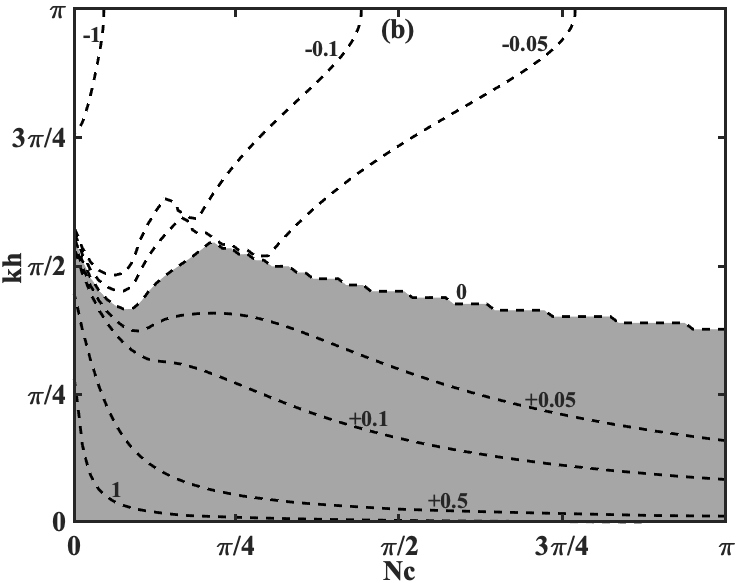}
\includegraphics[width=0.33\linewidth, height=0.3\linewidth]{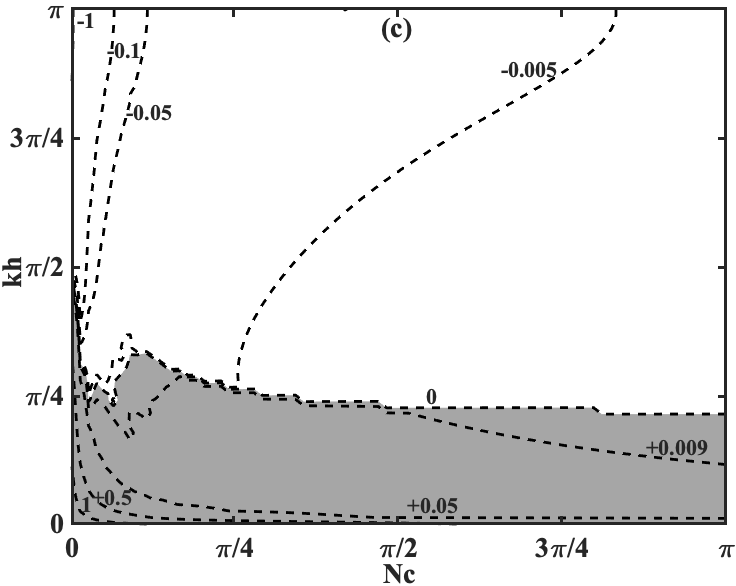}
\vskip 1pt
\includegraphics[width=0.33\linewidth, height=0.3\linewidth]{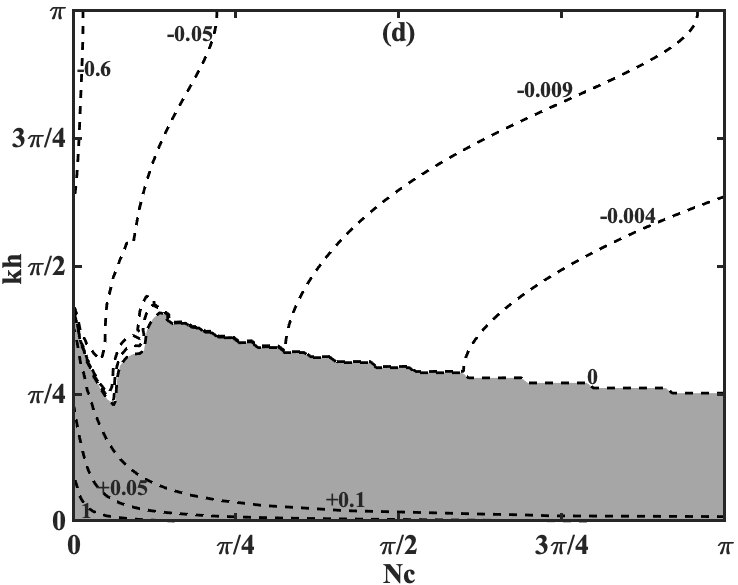}
\includegraphics[width=0.33\linewidth, height=0.3\linewidth]{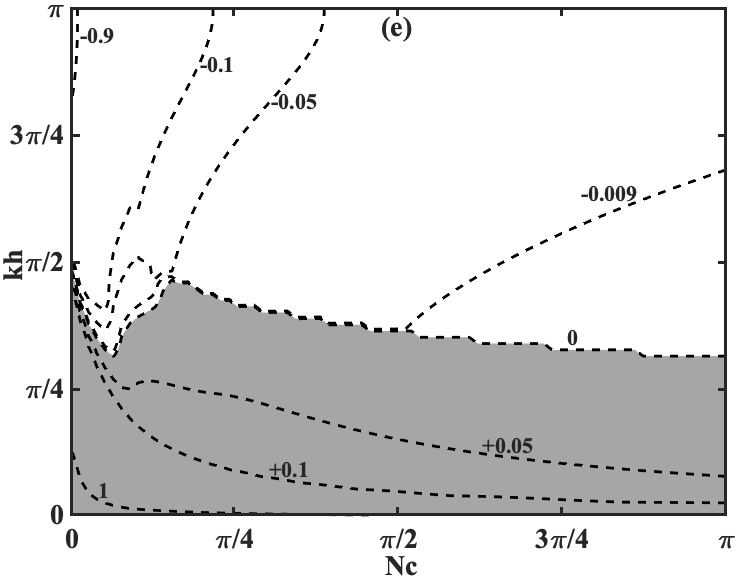}
\includegraphics[width=0.33\linewidth, height=0.3\linewidth]{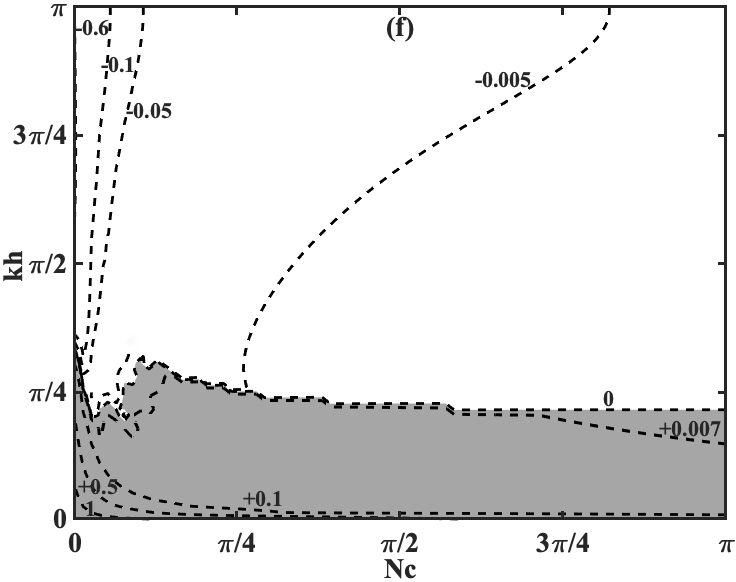}
\vskip 1pt
\includegraphics[width=0.33\linewidth, height=0.3\linewidth]{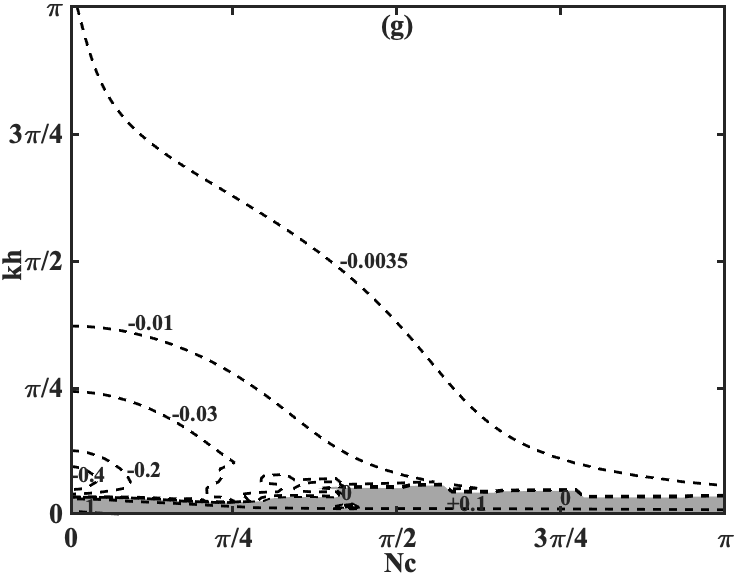}
\includegraphics[width=0.33\linewidth, height=0.3\linewidth]{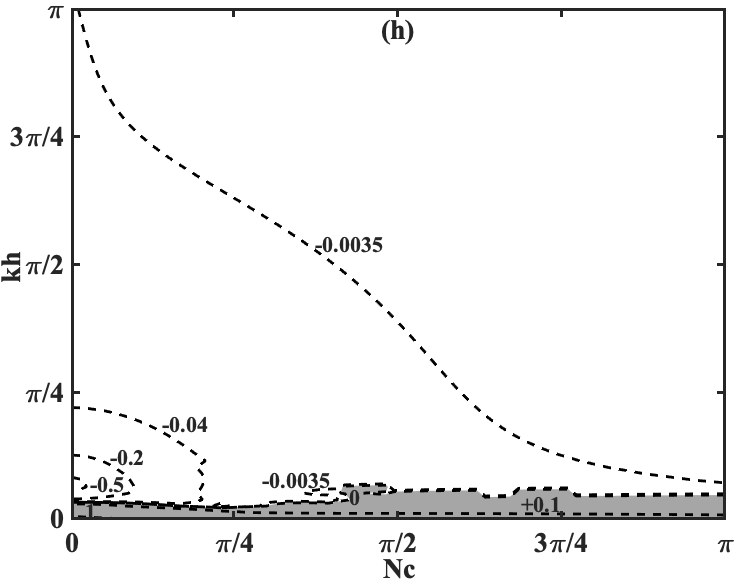}
\includegraphics[width=0.33\linewidth, height=0.3\linewidth]{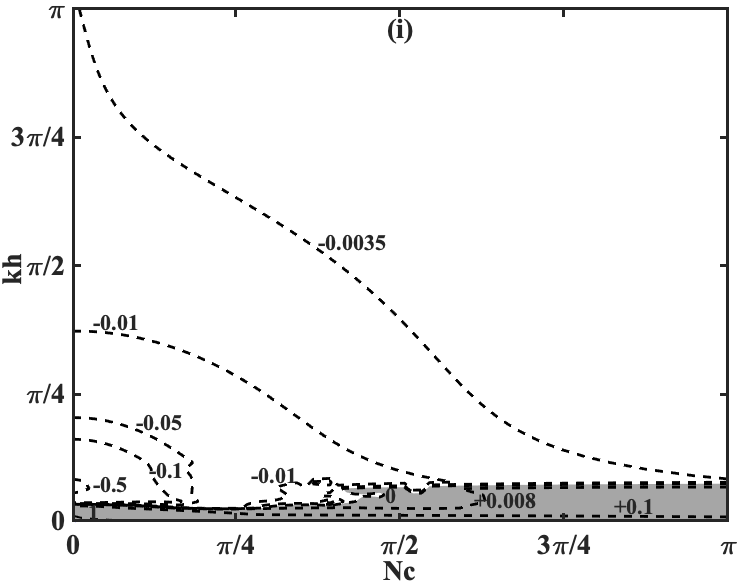}
\caption{Group velocity ratio contours, $V_g$, at $\theta=0.5$, $\alpha=0.9$ and (a) $Pe=0.001, Da=-0.01$, (b) $Pe=0.001, Da=0.0$, (c) $Pe=0.001, Da=0.01$, (d) $Pe=0.01, Da=-0.01$, (e) $Pe=0.01, Da=0.0$, (f) $Pe=0.01, Da=0.01$, (g) $Pe=1.0, Da=-0.01$, (h) $Pe=1.0, Da=0.0$ and (i) $Pe=1.0, Da=0.01$.}
\label{fig:Fig1}
\end{figure}

Next, transforming the exact solution of the linear 1D FADR equation, using Fourier-Laplace transform as $u(x, t) = \int \int \hat{U}(k, \omega) e^{{\it I}(kx - \omega t)}$, we arrive at the exact dispersion relation,
\beq
\omega_{exact} = \omega = I(\lambda-\gamma k^2-c I k)^{\frac{1}{\alpha}}.
\label{eqn:step12}
\eeq
Similarly, we have a corresponding numerical dispersion relation for the approximate equation~\eqref{eqn:step3},
\beq
\omega_{num} = I(\lambda^N-\gamma^N k^2-c^N I k)^{\frac{1}{\alpha}},
\label{eqn:step12a}
\eeq
where the superscript `N' denotes the corresponding numerical values of the parameters. Since
\beq
G_{num} = \frac{u(k, t+\Delta t)}{u(k, t)} = \frac{e^{I(kx-\omega_{num} (t+\Delta t))}}{e^{I(kx-\omega_{num} t)}} = e^{-I\omega_{num}\Delta t} = e^{I \beta},
\label{eqn:step13}
\eeq
and the numerical phase speed is given by, $c_{num} = \frac{\omega_{num}}{k}$, we find that,
\beq
c_{num} = -\frac{\beta}{k \Delta t} = -\frac{1}{k \Delta t}\tan^{-1} \left( \frac{(G_{num})_{Imag}}{(G_{num})_{Re}} \right),
\label{eqn:step14}
\eeq
where the subscript `Imag' / `Re' denote the imaginary / real values, respectively. Using the expression for the exact phase speed, $c_{exact} = \frac{\omega_{exact}}{k}$ and the non-dimensional parameters described in equation~\eqref{eqn:step3}, we find the ratio of the phase speeds,
\beq
\frac{c_{num}}{c_{exact}} = \frac{-I\beta}{(DaN_c-(kh)^2Pe-I(kh)N_c)^{\frac{1}{\alpha}}},
\label{eqn:step15}
\eeq

Finally, the expression for the numerical and the exact group velocities are given by,
\begin{align}
& {V_g}_{num} = \left[\frac{\partial}{\partial k} (\omega_{num})\right]_{Re} = \left[\frac{\partial}{\partial k} \left(\frac{\beta}{\Delta t}\right)\right]_{Re} = \left[\frac{h}{\Delta t}\frac{d\beta}{d(kh)}\right]_{Re}\nonumber\\
&{V_g}_{exact}\! =\!\! \left[\frac{\partial}{\partial k} (\omega_{exact})\right]_{Re} \!\!\! = \! \!\left[\frac{\partial}{\partial k} \left(I(\lambda\!-\!\gamma k^2 \!-\! cI k)^{\frac{1}{\alpha}}\right)\right]_{Re} \!\!\!=\! \!\left[\frac{1}{\alpha} (c\!-\!2k\gamma I)(-I\omega)^{1\!-\!\alpha}\right]_{Re}.
\label{eqn:step16}
\end{align}

Hence,
\beq
\frac{(V_g)_{num}}{(V_g)_{exact}} = \left( \frac{\alpha}{\splitfrac{N_c\left(r^{1-\alpha}(\Delta t)^{1-\alpha}\cos{(1-\alpha)\phi}\right)}{+2 kh \left(r^{1-\alpha}(\Delta t)^{1-\alpha}\sin{(1-\alpha)\phi}\right)Pe}} \right) \left(\frac{d\beta}{d(kh)}\right)\bigg|_{Re},
\label{eqn:step17}
\eeq
where $\omega = r e^{I (\phi+\frac{\pi}{2})}$.
\begin{figure}[htbp]
\includegraphics[width=0.33\linewidth, height=0.3\linewidth]{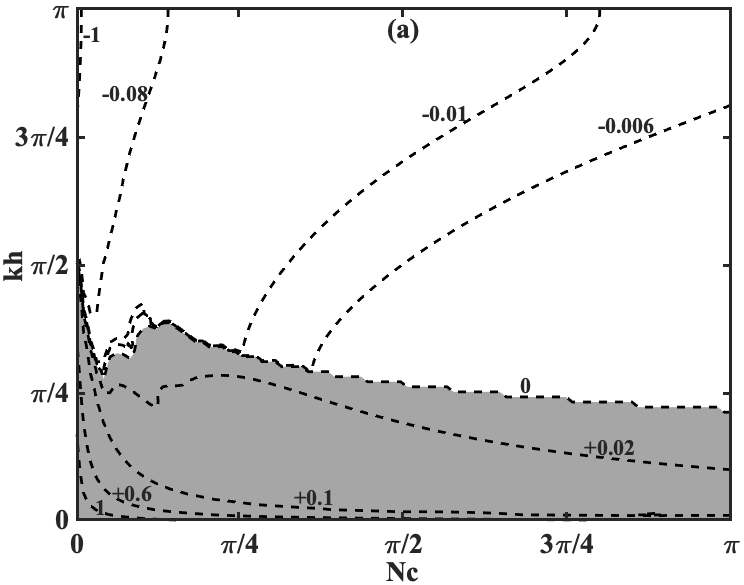}
\includegraphics[width=0.33\linewidth, height=0.3\linewidth]{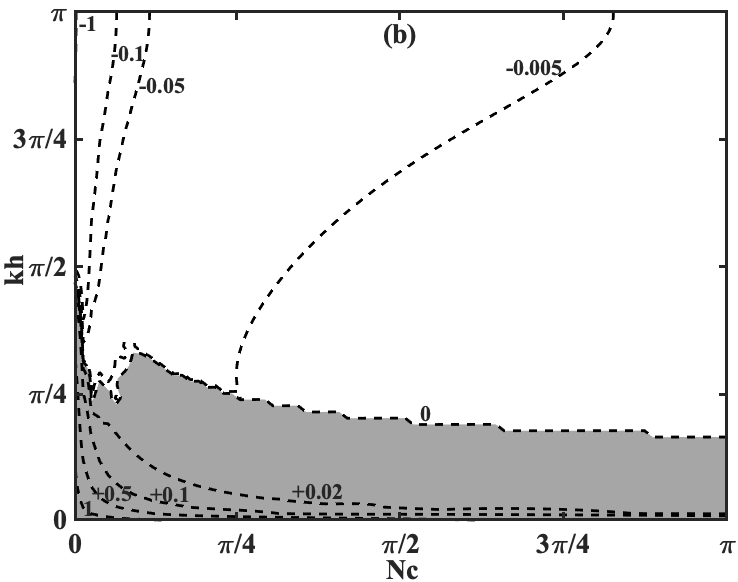}
\includegraphics[width=0.33\linewidth, height=0.3\linewidth]{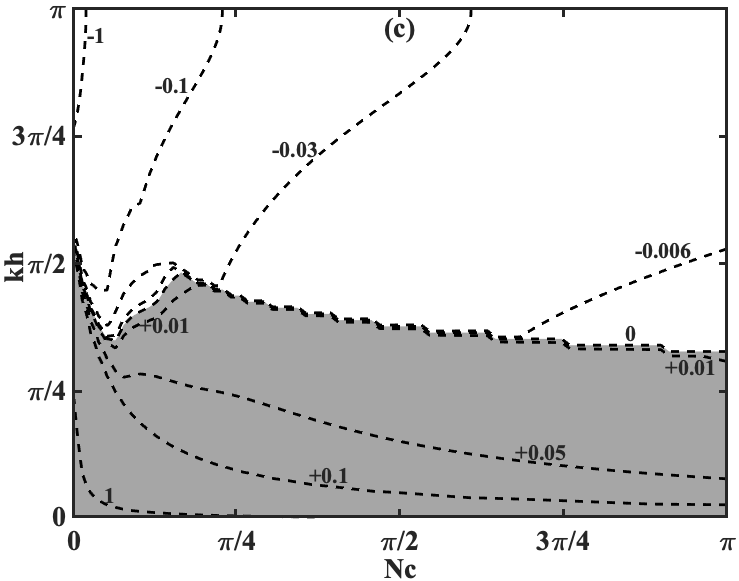}
\vskip 1pt
\includegraphics[width=0.33\linewidth, height=0.3\linewidth]{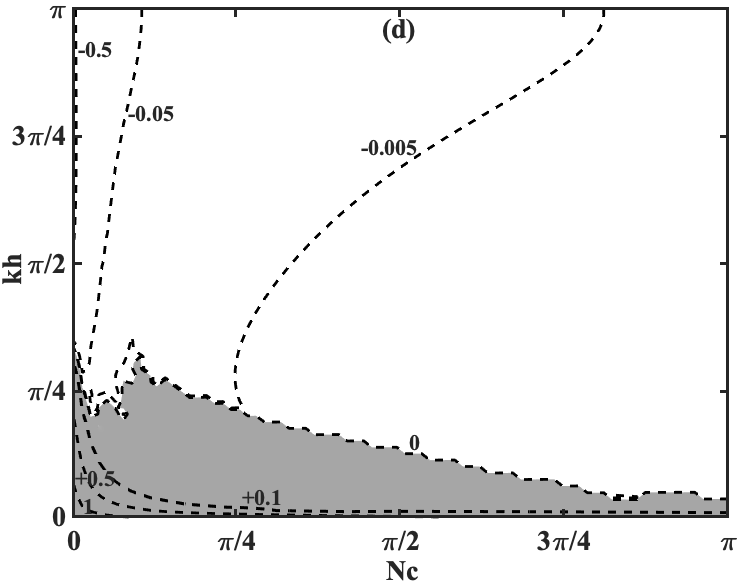}
\includegraphics[width=0.33\linewidth, height=0.3\linewidth]{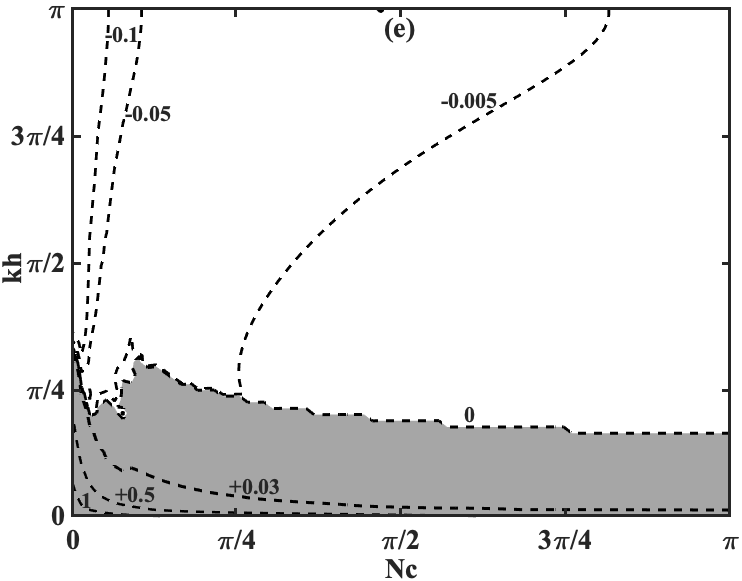}
\includegraphics[width=0.33\linewidth, height=0.3\linewidth]{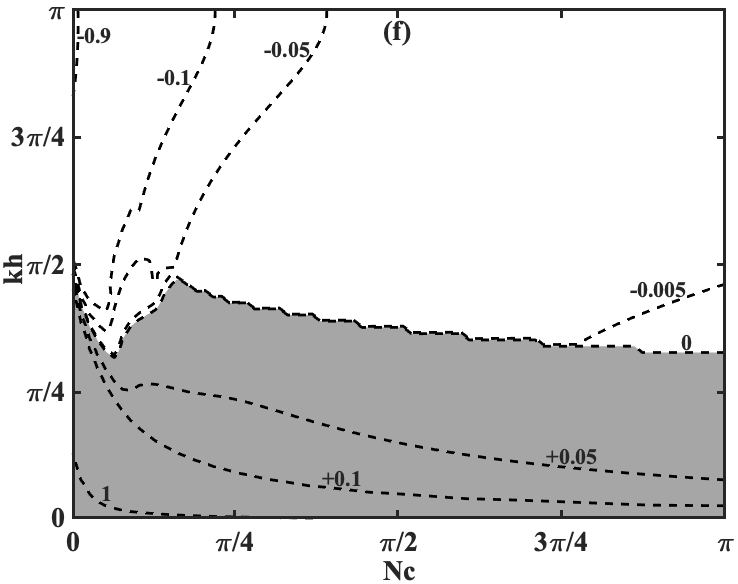}
\vskip 1pt
\includegraphics[width=0.33\linewidth, height=0.3\linewidth]{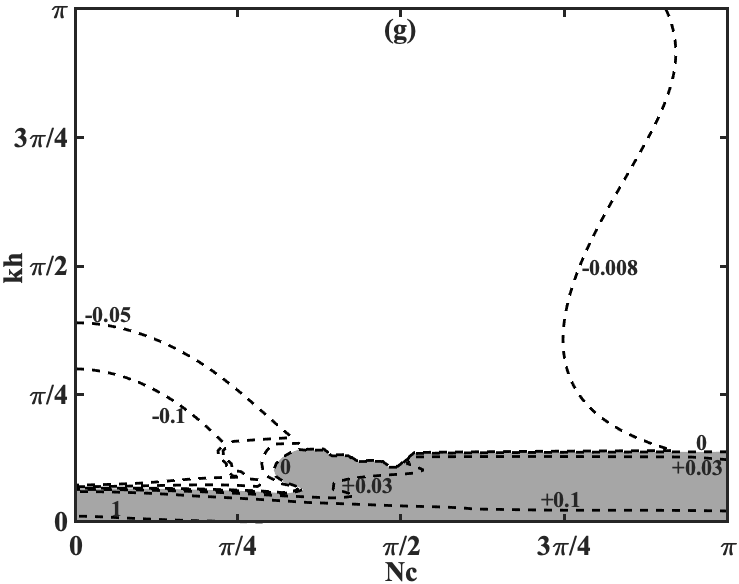}
\includegraphics[width=0.33\linewidth, height=0.3\linewidth]{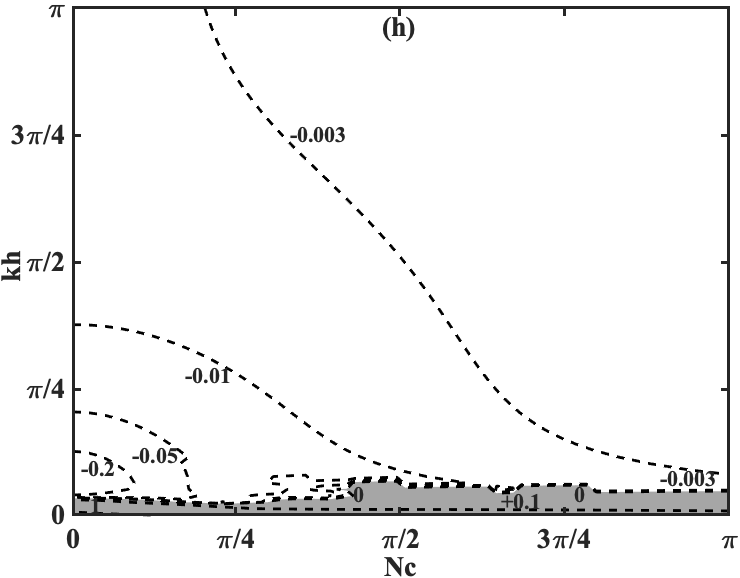}
\includegraphics[width=0.33\linewidth, height=0.3\linewidth]{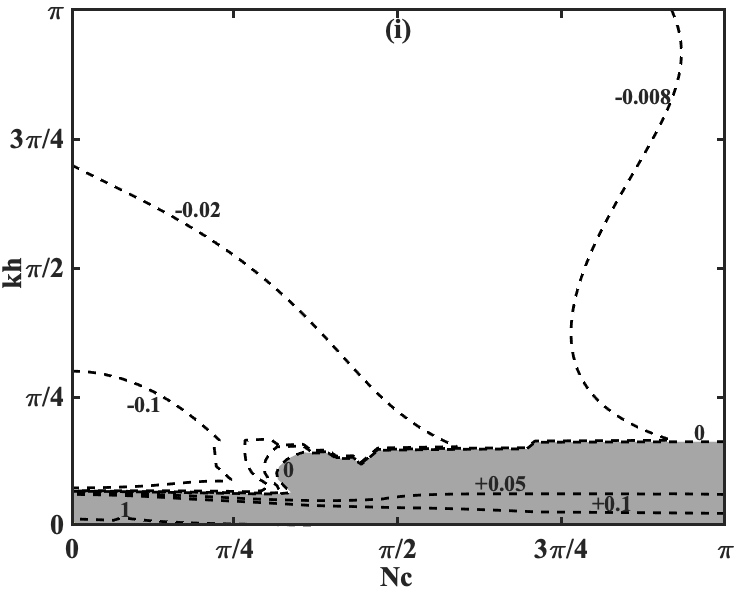}
\caption{Group velocity ratio contours, $V_g$, at $\theta=1.0$, $\alpha=0.9$ and (a) $Pe=0.001, Da=-0.01$, (b) $Pe=0.001, Da=0.0$, (c) $Pe=0.001, Da=0.01$, (d) $Pe=0.01, Da=-0.01$, (e) $Pe=0.01, Da=0.0$, (f) $Pe=0.01, Da=0.01$, (g) $Pe=1.0, Da=-0.01$, (h) $Pe=1.0, Da=0.0$ and (i) $Pe=1.0, Da=0.01$.}
\label{fig:Fig2}
\end{figure}

The contour plots for the group velocity ratio, $V_g = \left[\frac{V_{g, num}}{V_{g, exact}}\right]_{Re}$, for Peclet numbers, $Pe = 0.001, 0.01$ and $1.0$, and Damk\"{o}hler numbers, $Da = -0.01, 0.0$ and $0.01$, are presented for two values of $\theta$, namely, $\theta=0.5$ (in figure~\ref{fig:Fig1}) and $\theta=1.0$ (in figure~\ref{fig:Fig2}) in the $(N_c, kh)$-plane. The corresponding contour plots for the absolute phase speed error, $\Delta c = |1 - \frac{c_{num}}{c_{exact}}|$ are shown in figure~\ref{fig:Fig3} and figure~\ref{fig:Fig4} for $\theta=0.5$ and $\theta=1.0$, respectively. Regions of favorable spectral properties ($V_g > 0$ and $\Delta c \le 0.1$) are highlighted with grey color.
\begin{figure}[htbp]
\includegraphics[width=0.33\linewidth, height=0.3\linewidth]{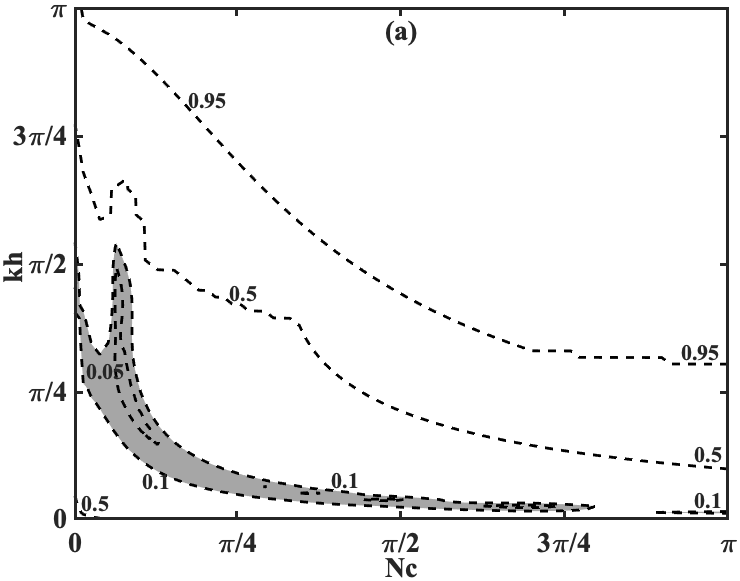}
\includegraphics[width=0.33\linewidth, height=0.3\linewidth]{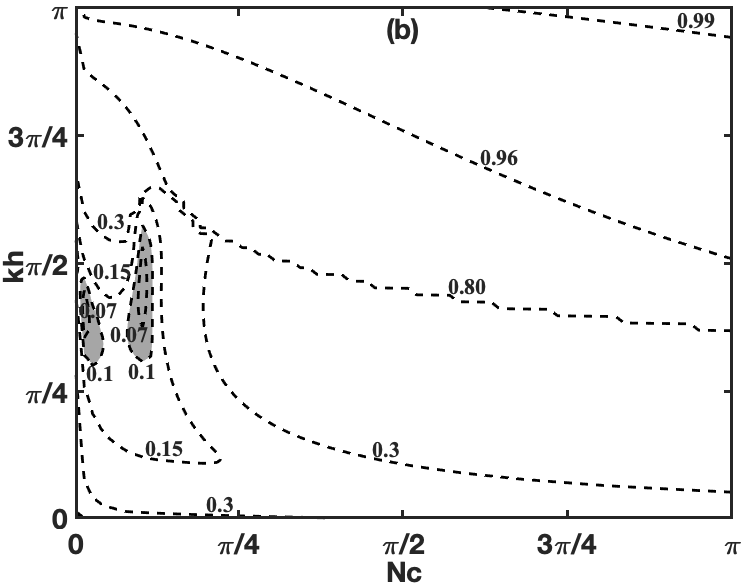}
\includegraphics[width=0.33\linewidth, height=0.3\linewidth]{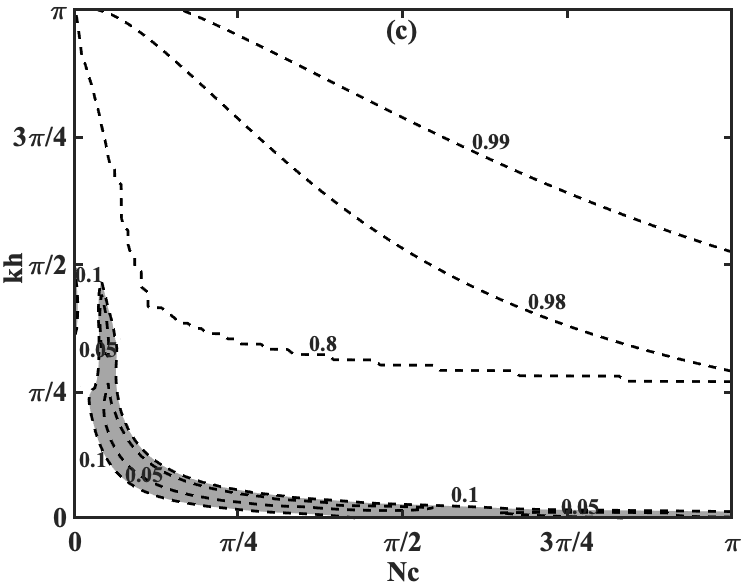}
\vskip 1pt
\includegraphics[width=0.33\linewidth, height=0.3\linewidth]{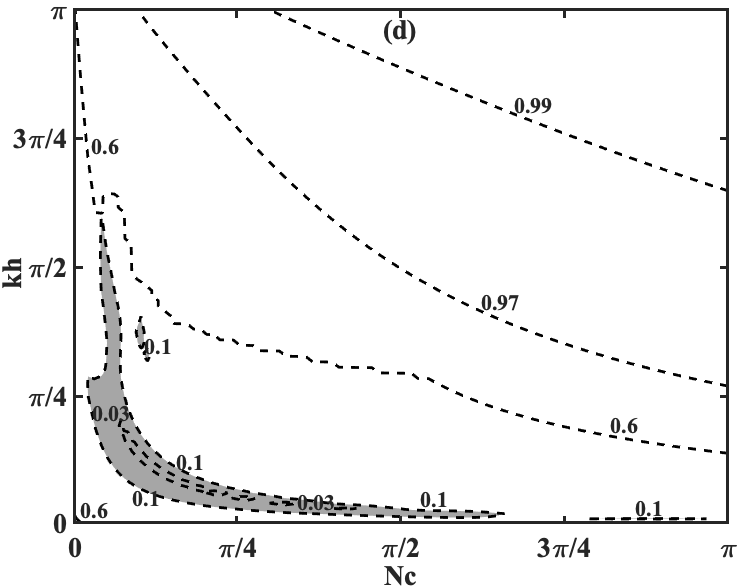}
\includegraphics[width=0.33\linewidth, height=0.3\linewidth]{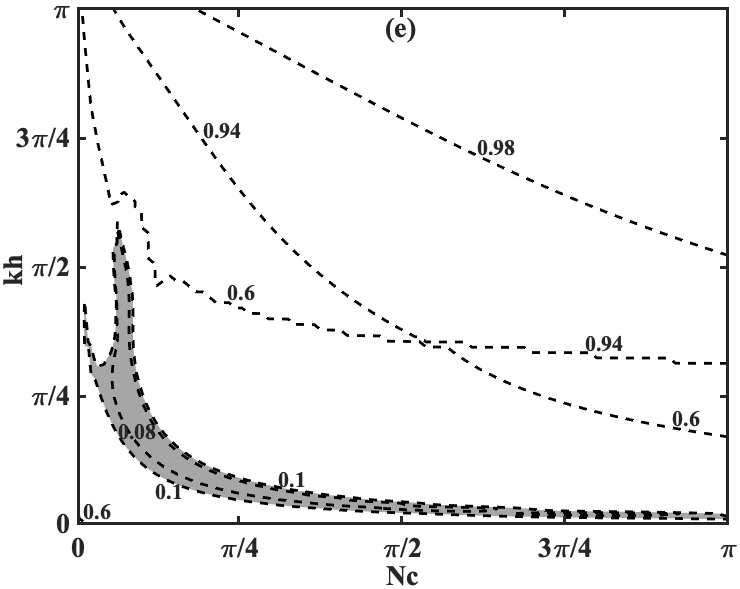}
\includegraphics[width=0.33\linewidth, height=0.3\linewidth]{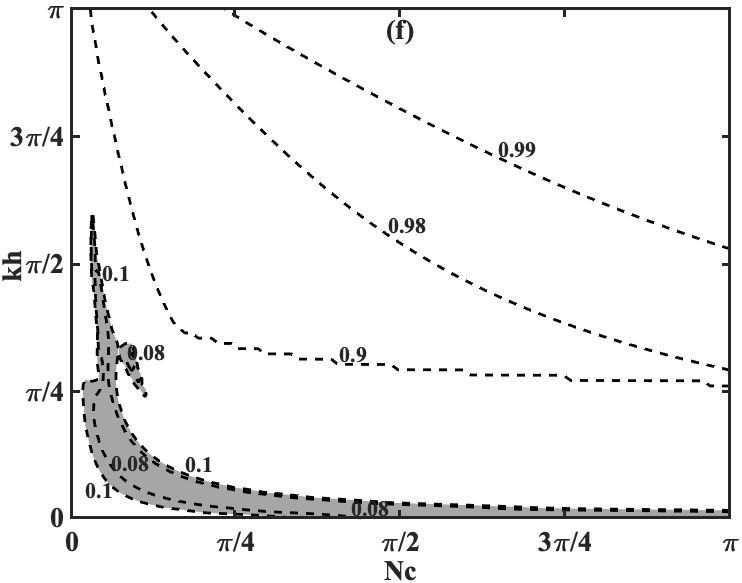}
\vskip 1pt
\includegraphics[width=0.33\linewidth, height=0.3\linewidth]{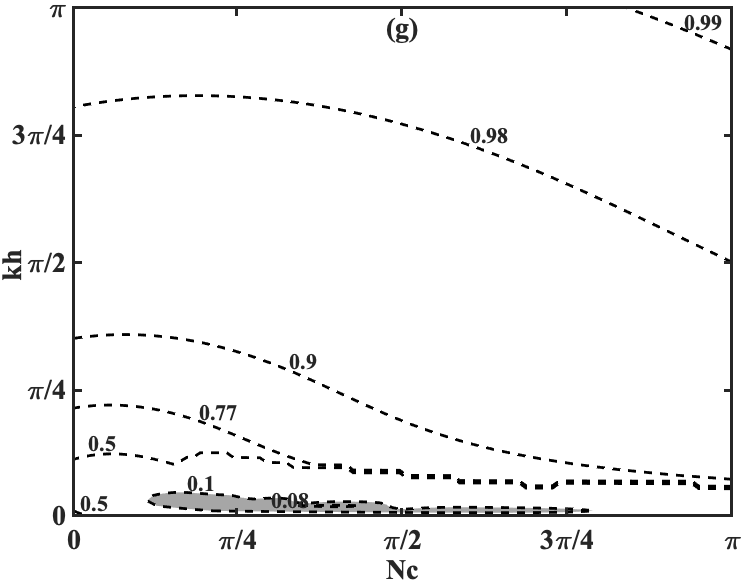}
\includegraphics[width=0.33\linewidth, height=0.3\linewidth]{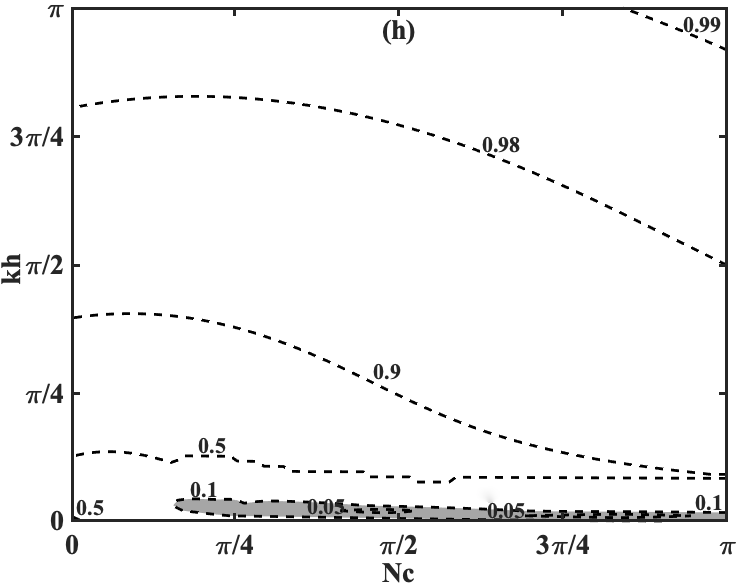}
\includegraphics[width=0.33\linewidth, height=0.3\linewidth]{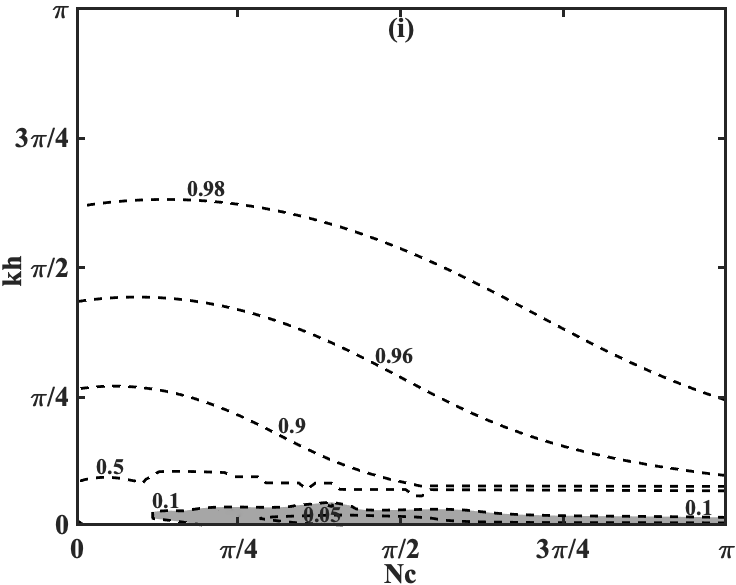}
\caption{Absolute phase speed error contours, $\Delta c$ at $\theta=0.5$, $\alpha=0.9$ and (a) $Pe=0.001, Da=-0.01$, (b) $Pe=0.001, Da=0.0$, (c) $Pe=0.001, Da=0.01$, (d) $Pe=0.01, Da=-0.01$, (e) $Pe=0.01, Da=0.0$, (f) $Pe=0.01, Da=0.01$, (g) $Pe=1.0, Da=-0.01$, (h) $Pe=1.0, Da=0.0$ and (i) $Pe=1.0, Da=0.01$.}
\label{fig:Fig3}
\end{figure}

We outline the spectral properties, namely, the group velocity ratio and the absolute phase speed error, for a specific case of the fractional order, $\alpha = 0.9$. In general, we find that in the limit of vanishingly small values of $kh$, the numerical method has favorable spectral properties (i.~e., $V_g > 0$ and $\Delta c \le 0.1$ in the limit, $kh \rightarrow 0$). The region of positive group velocities ($V_g > 0$) indicate a region where the numerical solution travels in the correct direction and the numerical instabilities in the form of q-waves are avoided~\cite{Sircar2020}. Figure~\ref{fig:Fig1} and~\ref{fig:Fig2} indicate that this favorable region is restricted to smaller values of $kh$ with larger values of $Pe$ as well as with smaller values of $\theta$. While the former observation can be attributed to the fact that a stiff diffusive term (i.~e., larger `$K_2({\bf x}, t) \nabla^2 u({\bf x}, t)$' term in equation~\eqref{eqn:method1} can be correctly approximated with smaller grid-size, $h$; the latter observation can be explained through a numerical stabilization due to the implicit treatment of fast time scale term (in this case, the diffusive term). Both of these observations are in congruence with the corresponding analysis of the integer order ADR equations~\cite{Sircar2020}. Similarly, the absolute phase error contours highlight spectrally favorable region ($\Delta c \le 0.1$, equivalently the phase errors are restricted to 10\% of the exact phase speed, see figures~\ref{fig:Fig3}, \ref{fig:Fig4}) at lower values of $Pe$ and at lower values of $\theta$, indicating the fact that the phase errors in this numerical method can be reduced by introducing finer grids~\cite{Sengupta2012}.

We summarize our discussion by indicating two sources of error which are particularly perplexing in the Direct Numerical Simulations (DNS) of FADR equations. The first source of error is the existence of $q-$waves for those set of numerical parameters for which the spatiotemporal discretization is asymptotically stable. In such a situation, the $q-$waves do not attenuate and have to be eliminated by deploying an explicit filter~\cite{Visbal2002}. Another aspect of spectral error is related to the Gibbs' phenomenon which occurs as a consequence of sharp discontinuity in the solution and which causes fictitious oscillations, a problem which can be remedied using high accuracy dispersion relation preserving schemes (which has atleast shown promise in integer order partial differential equations)~\cite{Sircar2006}. 
\begin{figure}[htbp]
\includegraphics[width=0.33\linewidth, height=0.3\linewidth]{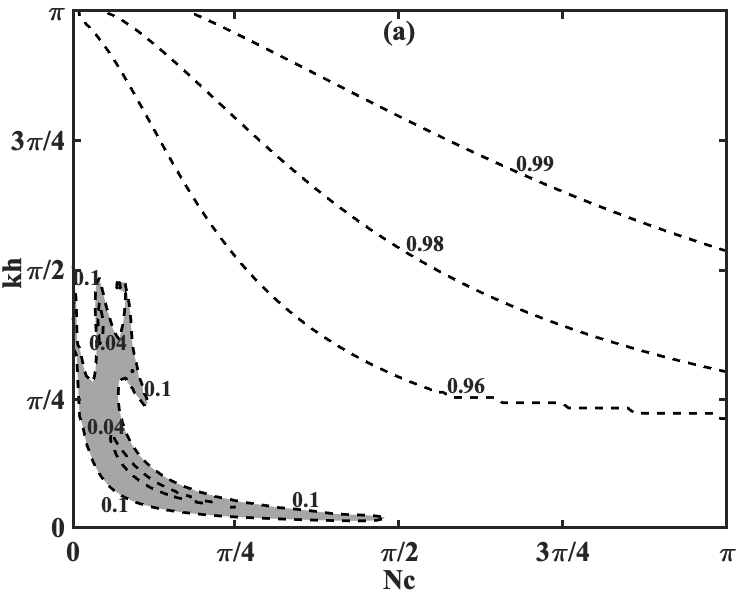}
\includegraphics[width=0.33\linewidth, height=0.3\linewidth]{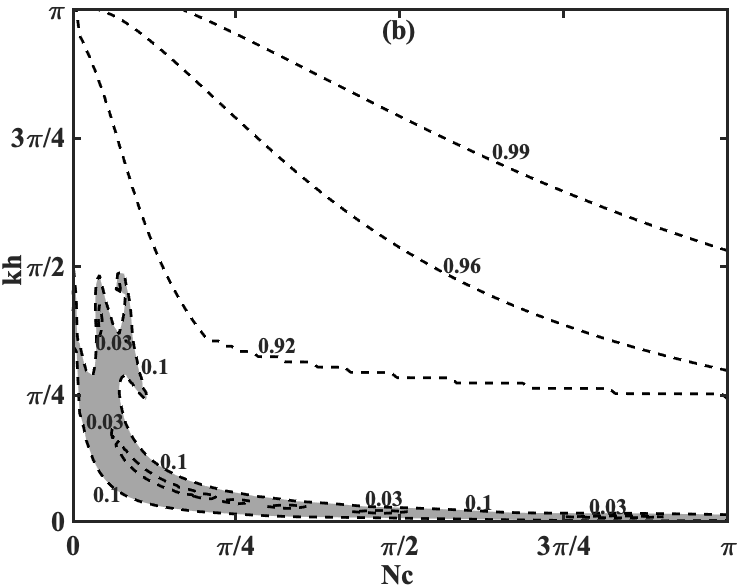}
\includegraphics[width=0.33\linewidth, height=0.3\linewidth]{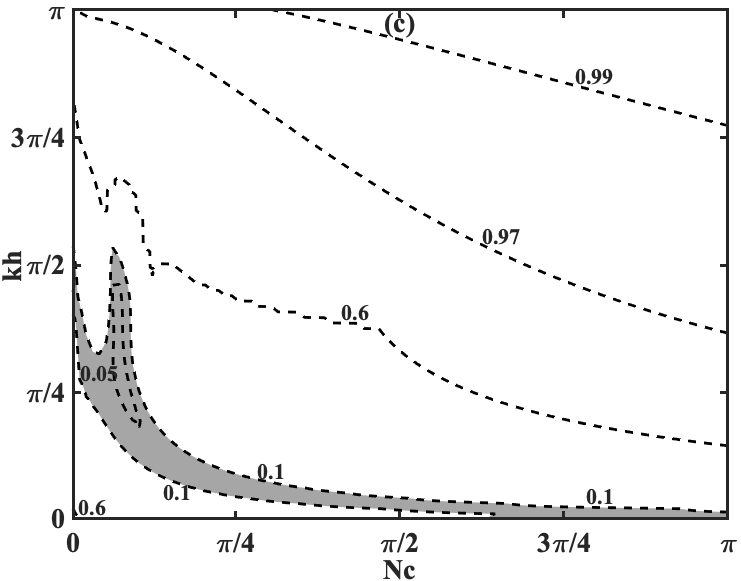}
\vskip 1pt
\includegraphics[width=0.33\linewidth, height=0.3\linewidth]{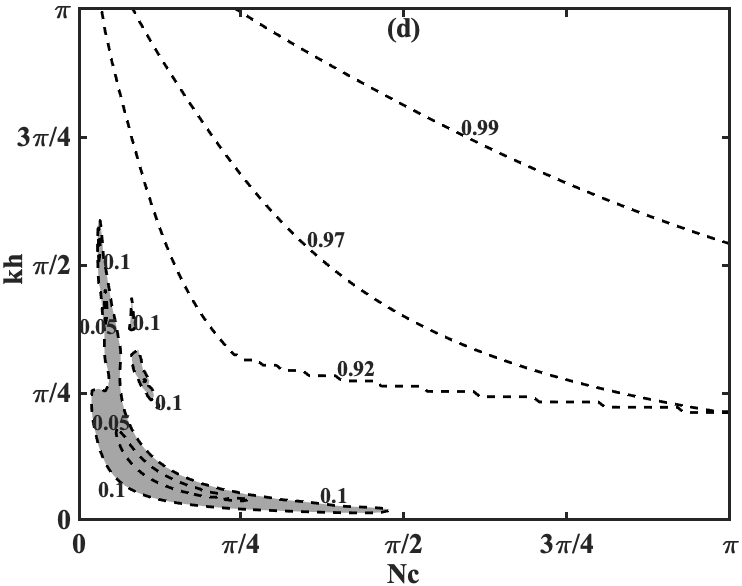}
\includegraphics[width=0.33\linewidth, height=0.3\linewidth]{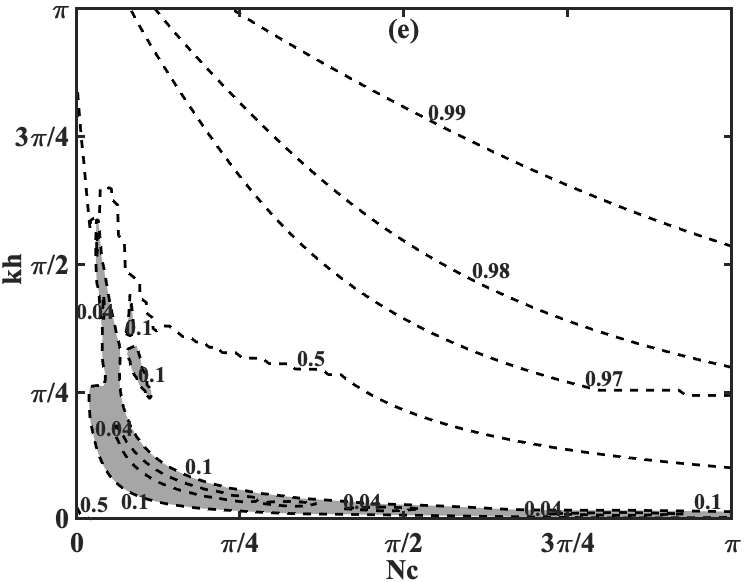}
\includegraphics[width=0.33\linewidth, height=0.3\linewidth]{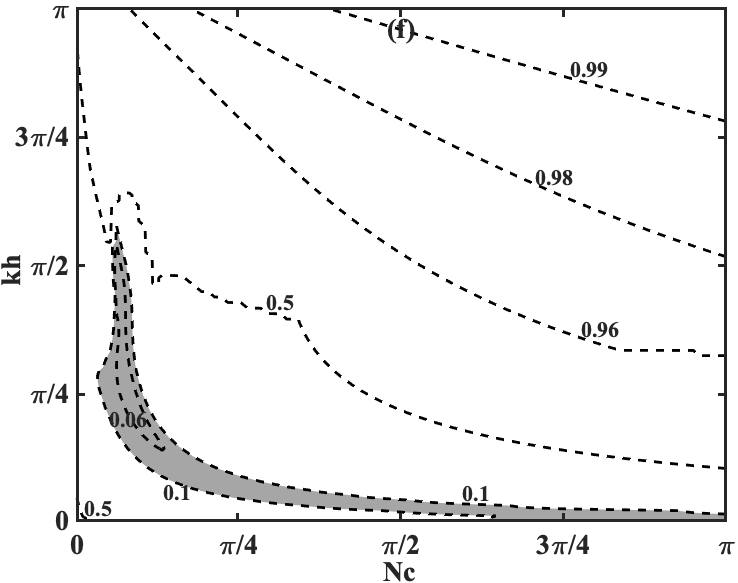}
\vskip 1pt
\includegraphics[width=0.33\linewidth, height=0.3\linewidth]{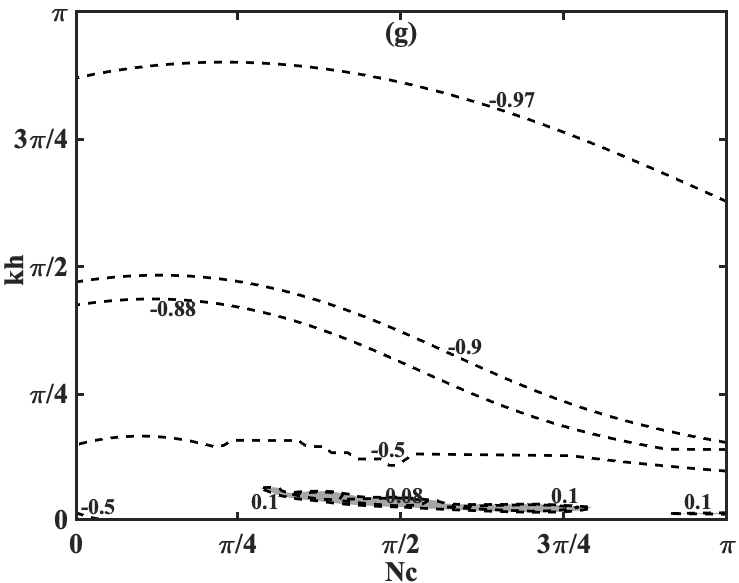}
\includegraphics[width=0.33\linewidth, height=0.3\linewidth]{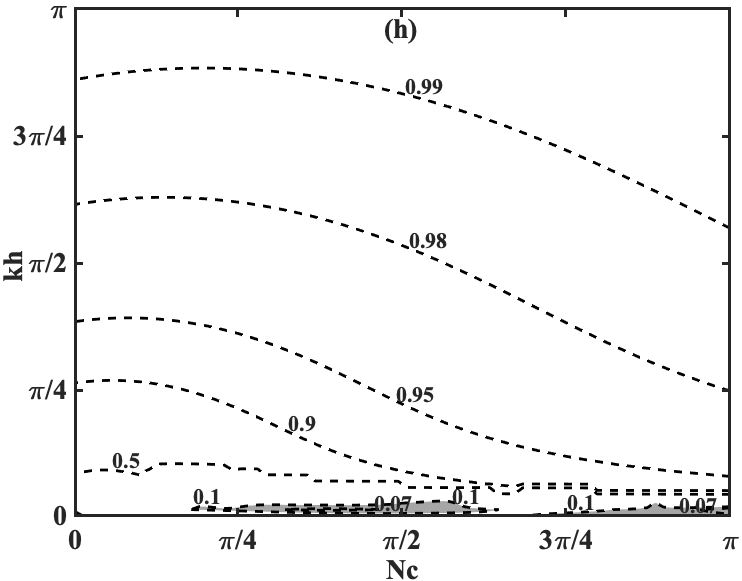}
\includegraphics[width=0.33\linewidth, height=0.3\linewidth]{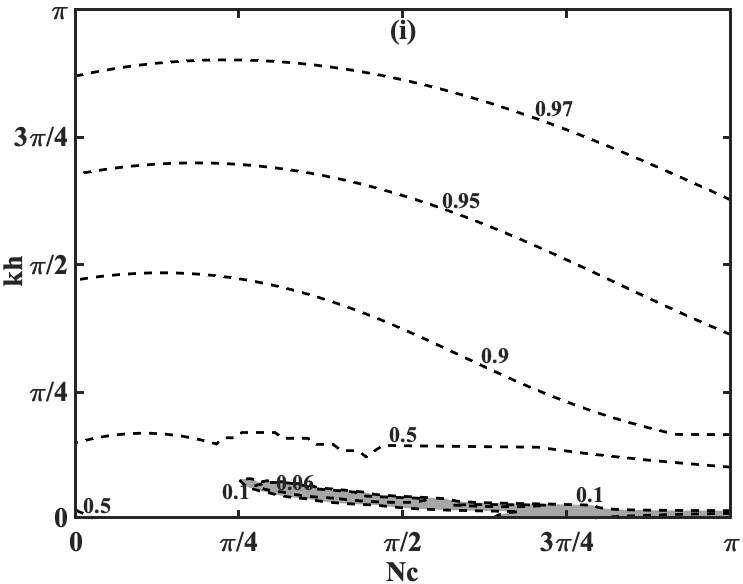}
\caption{Absolute phase speed error contours, $\Delta c$ at $\theta=1.0$, $\alpha=0.9$ and (a) $Pe=0.001, Da=-0.01$, (b) $Pe=0.001, Da=0.0$, (c) $Pe=0.001, Da=0.01$, (d) $Pe=0.01, Da=-0.01$, (e) $Pe=0.01, Da=0.0$, (f) $Pe=0.01, Da=0.01$, (g) $Pe=1.0, Da=-0.01$, (h) $Pe=1.0, Da=0.0$ and (i) $Pe=1.0, Da=0.01$.}
\label{fig:Fig4}
\end{figure}

%%%%%%%%%%%%%%%%%%%%%%%%%%%%%%%%%%%%%%%%%%%%%%%%%%%%
\section{Method validation: 2D fractional diffusion equation} \label{sec:MV}
Using the spectrally relevant parameter values discussed in section~\ref{subsec:SA}, the $\theta$-method is verified by solving the 2D fractional diffusion equation ($K_1 = f = 0$ and $K_2 = 1$ in equation~\eqref{eqn:method1}), originally proposed by Brunner~\cite{Brunner2010}, for the case of (a) zero Dirichlet boundary conditions and (b) zero Neumann conditions, over the square domain, $(x, y) \in \Gamma = [-1,\,1]^2$. The initial condition for the two test cases can be outlined as follows,
\begin{enumerate}[i)]
\item {\bf Dirichlet case:} $u_0(x, y) = \cos\left( \frac{\pi}{2} x \right) \cos\left( \frac{\pi}{2} y \right)$,
\item {\bf Neumann case:} $u_0(x, y) = \sin\left( \frac{\pi}{2} x \right) \sin \left( \frac{\pi}{2} y \right)$,
\end{enumerate}
and the exact solution for both cases is given by
\beq
u_{Exact}(x, y, t) = E_\alpha \left( -\frac{1}{4} \pi^2 t^\alpha \right) u_0(x, y), 
\label{eqn:MV1}
\eeq
where $E_\alpha (z) = \sum^\infty_{k=0} \frac{z^k}{\Gamma(\alpha k + 1)}, \quad \alpha > 0$, is the one-parameter Mittag-Leffler function. 
\begin{figure}[htbp]
\includegraphics[width=0.495\linewidth, height=0.43\linewidth]{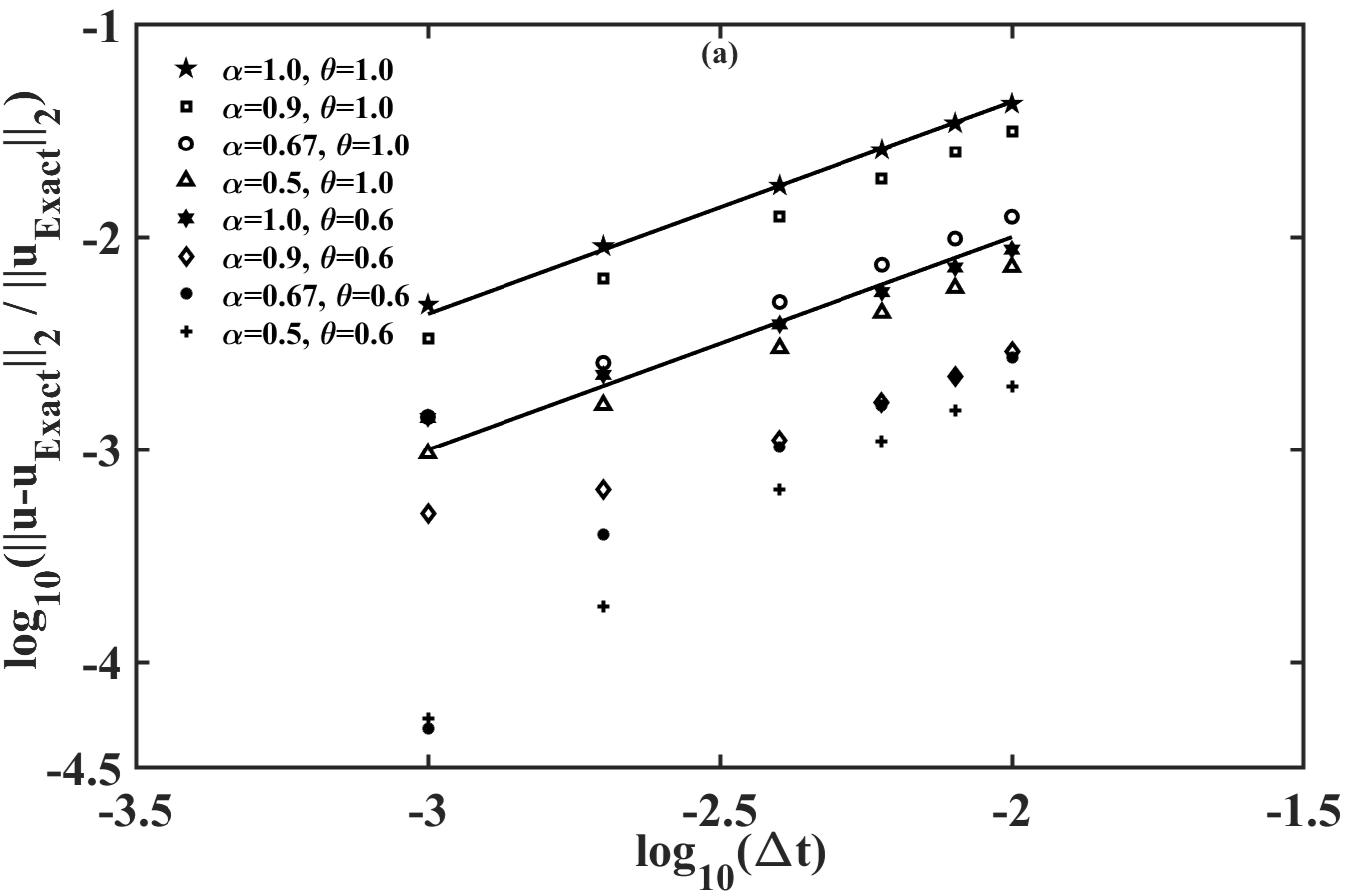}
\includegraphics[width=0.495\linewidth, height=0.43\linewidth]{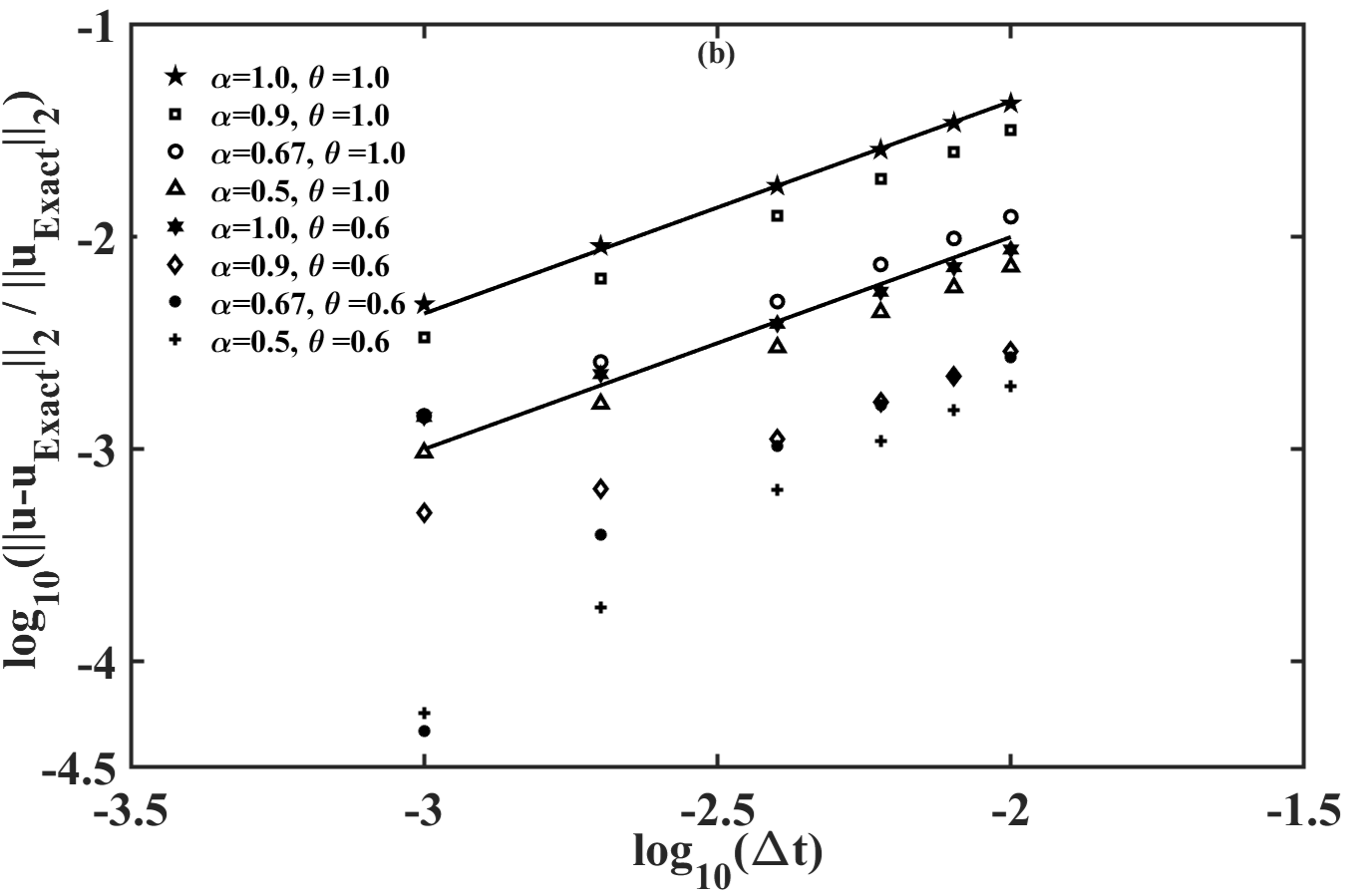}
\caption{Relative error for the solution of the 2D fractional diffusion equation at simulation time $T=0.35$, at $\alpha=1.0, \theta=1.0$ ($\star$); $\alpha=0.9, \theta=1.0$ ($\square$); $\alpha=0.67, \theta=1.0$ ($\circ$); $\alpha=0.5, \theta=1.0$ ($\triangle$); $\alpha=1.0, \theta=0.6$ ($\ast$); $\alpha=0.9, \theta=0.6$ ($\diamond$); $\alpha=0.67, \theta=0.6$ ($\bullet$); $\alpha=0.5, \theta=0.6$ ($+$); and for (a) Dirichlet boundary conditions, and (b) Neumann boundary conditions.}
\label{fig:Fig5}
\end{figure}

The domain, $\Gamma \cup \partial \Gamma$, is discretized using $51 \times 51$ points and the relative error in $l_2-norm$, $\frac{\| \widetilde{u} - u_{Exact}\|_2}{\| u_{Exact} \|_2}$, versus discretization time, $\Delta t$, is shown in figure~\ref{fig:Fig5}. The total computation time is fixed at $T=0.35$, for all cases. We note that slope of the error curves at $\alpha=1.0$ is one (highlighted with solid curves in figure~\ref{fig:Fig5}a,b). This observation can be explained since at this value of $\alpha$, the $\theta$-method reduces to the standard, integer order, Euler method, which is first order accurate. Further, note that the slope of the error curves for $\alpha < 1.0$ is sub-linear, indicating that for the fractional-order diffusion equation, the $\theta$-method has sub-linear order of accuracy.

%%%%%%%%%%%%%%%%%%%%%%%%%%%%%%%%%%%%%%%%%%%%%%%%%%%%
\section{Numerical simulation: 2D fractional viscoelastic channel flow} \label{sec:NS}
Next, we describe the incompressible, subdiffusive dynamics of a planar (2D) viscoelastic channel flow for polymer melts and present the numerical solution using the $\theta$-method. In an earlier work~\cite{Chauhan2021}, the authors have derived of the mathematical model (which is a nonlinear 2D coupled FADR equation) and the linear stability of the incompressible, subdiffusive dynamics of such flows. Using the following scales for non-dimensionalizing the governing equations: the height of the channel $H$ for {\color{black}length}, the timescale $T$ corresponding to maximum base flow velocity (i.~e., $T = (H / \mathcal{U}_0)^{1/\alpha}$) for time and $\rho \mathcal{U}^2_0$ for pressure and stresses (where $\rho, \mathcal{U}_0$ is the density and the velocity scale, respectively), we summarize the model in streamfunction-vorticity formulation as follows,
\bseq \label{eqn:FullSystem}
\begin{align}
&Re \left[ \frac{\partial^\alpha \Omega}{\partial t^\alpha} + {\bf v} \cdot \nabla \Omega \right] = \nu \nabla^2 \Omega + (1-\nu) \nabla \times \nabla \cdot ({\bf A} +  {\bf F}), \label{eqn:Momentum} \\
& \nabla^2 \Psi = -\Omega, \label{eqn:Poisson} \\
&\frac{\partial^\alpha {\bf A}}{\partial t^\alpha} + {\bf v}\cdot \nabla {\bf A} - (\nabla {\bf v})^T {\bf A} -{\bf A}\nabla {\bf v} = \frac{{\bf D} - {\bf A}}{We}, \label{eqn:ExtraStress}
\end{align}
\eseq
where the variables $t, \Psi, {\bf v} = (u, v) = (\frac{\partial \Psi}{\partial y}, -\frac{\partial \Psi}{\partial x}), {\bf A}, \Omega = \nabla \times {\bf v}$ denote time, streamfunction, velocity, elastic stress tensor and vorticity, respectively. ${\bf F} = \frac{\mu}{1 - \nu} ({\bf E}:{\bf E}^T)$ is the finger tensor and ${\bf E} = e^{t(\nabla {\bf v})^{\frac{1}{\alpha}}}$ is the deformation gradient tensor (the operator $(\, : \,)$ represents the tensor inner product). The term `$\left(\nabla \times \nabla \cdot {\bf F}\right)$' represents the external body force. The parameters $\eta_s, \eta_p, \eta_0 = \eta_s + \eta_p$, $\nu = \eta_s / \eta_0$ and $\mu$ are the solvent viscosity, the polymeric contribution to the shear viscosity, the total viscosity, the viscous contribution to the total viscosity of the fluid and strength of the body force, respectively. The dimensionless groups characterizing inertia and elasticity are Reynolds number, $Re = \frac{\rho \mathcal{U}_0 H}{\eta_0}$, and Weissenberg number, $We = \frac{\lambda^\alpha \mathcal{U}_0}{H}$, respectively. The parameter, $\lambda$, is the polymer relaxation time. Note that equation~\eqref{eqn:ExtraStress} represents fractional version of the regular Oldroyd-B model for viscoelastic fluids~\cite{Sircar2019}, derived in a recently reported work~\cite{Chauhan2021}. Two specific cases of the fractional derivative are considered, namely, monomer diffusion in Rouse chain melts ($\alpha = \nicefrac{1}{2}$)~\cite{Rouse1953}, and in Zimm chain solution ($\alpha = \nicefrac{2}{3}$)~\cite{Zimm1956}. 

%%%%%%%%%%%%%%%%%%%%%%%%%%%%%%%%%%%%%%%%%%%%%%%%%%%%
\subsection{Initial and boundary conditions}\label{subsec:IBC}
\paragraph{Initial conditions:} A rectilinear coordinate system is used with $x, y$ denoting the channel flow direction and the transverse direction, respectively. The origin of this coordinate system is chosen at the left end of the lower wall of the channel. Initially, we assume that the flow is a plane (2D) Poiseuille flow with its variation entirely in the transverse direction, namely,
\beq
{\bf U}_0 = (y - y^2) {\bf e_x},
\label{eqn:IF}
\eeq
where ${\bf e_x}$ is the unit vector along x-direction. The base state elastic stress tensor, ${\bf A_0} = [{A_0}_{i j}]$, is chosen as,
\begin{align}
{A_0}_{11} = 0, \nonumber \\
{A_0}_{12} = {A_0}_{21} = 1 - 2y, \nonumber \\
{A_0}_{22} = 2 We (1 - 2y)^2,
\label{eqn:IS}
\end{align}

\paragraph{Boundary conditions:} In order to imitate an infinitely long channel, periodic boundary conditions are assumed at the flow inlet and outlet. No-slip (i.~e., $u = v = 0$) and zero tangential conditions (i.~e., $\frac{\partial u}{\partial x} = \frac{\partial v}{\partial x} = 0$) are imposed on the lower wall ($y = 0$) and the upper wall ($y = 1.0$) of the channel, respectively. Further, incompressibility constraint provides an additional condition on the walls: $\frac{\partial v}{\partial y} = 0$. Since the flow is parallel to the channel walls, the walls may be treated as streamline. Thus, the streamfunction value, $\Psi$, on the wall is set as a constant. That constant (which may be different on the lower and the upper wall) is found from the no-slip condition. Zero tangential condition imply that all tangential derivatives of streamfunction vanish on the wall. Thus, the boundary condition for vorticity is found from the Poisson equation~\eqref{eqn:Poisson},
\beq
\frac{\partial^2 \Psi}{\partial y^2}|_{wall} = -\Omega_{wall}.
\label{eqn:OmegaBC}
\eeq
Finally, the boundary conditions for the elastic stress tensor is constructed from equation~\eqref{eqn:ExtraStress}, coupled with the no-slip and zero tangential conditions, as follows,
\begin{align}
& A_{11} = 0, \nonumber \\
& A_{12} + We \frac{\partial^\alpha A_{12}}{\partial t^\alpha} - \frac{\partial u}{\partial y} = 0, \nonumber \\
& A_{22} + We \left( \frac{\partial^\alpha A_{22}}{\partial t^\alpha} - 2 \frac{\partial u}{\partial y} A_{12}\right)= 0.
\label{eqn:ExtraStressBC}
\end{align}

%%%%%%%%%%%%%%%%%%%%%%%%%%%%%%%%%%%%%%%%%%%%%%%%%%%%
\subsection{Algorithmic details}\label{subsec:AD}
The size of the domain is chosen to be $(x, y) \in [0,\,5] \times [0,\,1]$. The domain is discretized using $76 \times 51$ points, such that the discrete points are equally spaced at $\Delta x = \frac{5}{75}$ and $\Delta y = \frac{1}{50}$ (excluding the boundary points, where the Dirichlet boundary conditions are imposed) in the $x-$ and the $y-$ directions, respectively. The variable in the $\theta$-method is fixed at $\theta=1.0$. The minimum and the maximum time-step are chosen as $\Delta t_{min}=10^{-3}$ and $\Delta t_{max}=1.6 \times 10^{-2}$, respectively. The time-step in the simulation is increased adaptively using the technique outlined in Section~\ref{subsec:AdaptiveTime}. The Poisson equation~\eqref{eqn:Poisson} is iteratively solved using the Gauss-Siedal iteration technique. Assuming $(N+1)$ (or $(M+1)$) points in the $x$ (or $y$) direction, the size of the coefficient matrix is $N(M-1) \times N(M-1)$. Hence, due to size constraints, the numerical inversion of this coefficient matrix imposes severe memory restrictions. Instead, we note that the coefficient matrix has the following block diagonal structure with $(M-1)$ blocks,
\beq
\begin{pmatrix}
    D_1 & D_2 & 0 & \cdots & 0 \\
    D_2 & \ddots & \ddots & & \vdots \\
    0 & \ddots & \ddots & \ddots& 0\\
    \vdots &  & \ddots & \ddots & D_2\\
    0 & \cdots& 0 & D_2 & D_1 
  \end{pmatrix}
\label{eqn:AD1}
\eeq
where $D_1, D_2$ are $N \times N$ tridiagonal and diagonal matrices, respectively, such that $D_1 = \text{Tridiag} (r_1, R, r_1)$ and $D_2 = \text{Diag}~(r_2)$ and $r_1 = (\Delta x)^{-2}, r_2 = (\Delta y)^{-2}$ and $R = -2(r_1 + r_2)$. Since the exact location of the non-zero entries are known, the invocation of the entire coefficient matrix in the Guass-Siedal iteration is avoided. The solution is updated by utilizing the non-zero entries of each row on a case-by-case basis. The row diagonal dominance of the matrix~\eqref{eqn:AD1} insures that the iteration converges in finite number of steps. Similarly, the vorticity equation~\eqref{eqn:Momentum} involves another Laplacian term and its discretization involves a coefficient matrix identical to the form~\eqref{eqn:AD1} (with $r_1 = -\nu \theta (dx)^{-2}, r_2 = -\nu \theta (dy)^{-2}, R = -2(r_1 + r_2) + \frac{Re (\Delta t)^{-\alpha}}{\Gamma(2-\alpha)}$). Thus, the Laplacian term in equation~\eqref{eqn:Momentum} is treated identically as above.

%%%%%%%%%%%%%%%%%%%%%%%%%%%%%%%%%%%%%%%%%%%%%%%%%%%%
\subsection{Numerical results}\label{subsec:AD}
The Zimm's model~\cite{Zimm1956} predicts the (`shear rate and polymer concentration independent') viscosity of the polymer solution by calculating the hydrodynamic interaction of flexible polymers (an idea which was originally proposed by Kirkwood~\cite{Kirkwood1954}) by approximating the chains using a bead-spring setup. In contrast, the Rouse model~\cite{Rouse1953} predicts that the viscoelastic properties of the polymer chain via a generalized Maxwell model, where the elasticity is governed by a single relaxation time, which is independent of the number of Maxwell elements (or the so-called `submolecules'). Since we do not wish to study the effects of elasticity, the Weissenberg number is fixed at $We=10.0$. The value of $\mu$ is fixed at $\mu=10^{-2}$ such that the body force term in equation~\eqref{eqn:Momentum} behaves as a perturbative term for the system which is initially at steady state~(\ref{eqn:IF}, \ref{eqn:IS}). Two different values of $\nu$ are considered: $\nu=0.3$ (elastic stress dominated case) and $\nu=0.6$ (viscous stress dominated case).

\paragraph{Zimm's model:} The vorticity contours for the Zimm's model are represented via figures~\ref{fig:Fig7} and~\ref{fig:Fig9}, for the elastic stress dominated case and the viscous stress dominated case, respectively. Observe that the `flow structures' are larger in size and magnitude for the elastic stress dominated case (i.~e., compare the range of the contour values in both figures). Also, observe that the structures appearing in the Zimm's model are significantly smaller in magnitude than the Rouse model (figures~\ref{fig:Fig6} and~\ref{fig:Fig8}). Physically, the formation of these `spatiotemporal macrostructures' are associated with the entanglement of the polymer chains at microscale~\cite{Rubenstein2003}, leading to localized, non-homogeneous regions with higher viscosity. Flows with parameter values, $\nu < 0.5$ (or the elastic stress dominated case), are those associated with higher concentration of polymers per unit volume. Experiments~\cite{Fogelson2015} have shown that non-Newtonian fluids with a larger polymer concentration, have a greater tendency (for the polymer strands) to agglomerate, the so-called `over-crowding effect'~\cite{Doi1996}.
\begin{figure}[htbp]
\includegraphics[width=0.49\linewidth, height=0.3\linewidth]{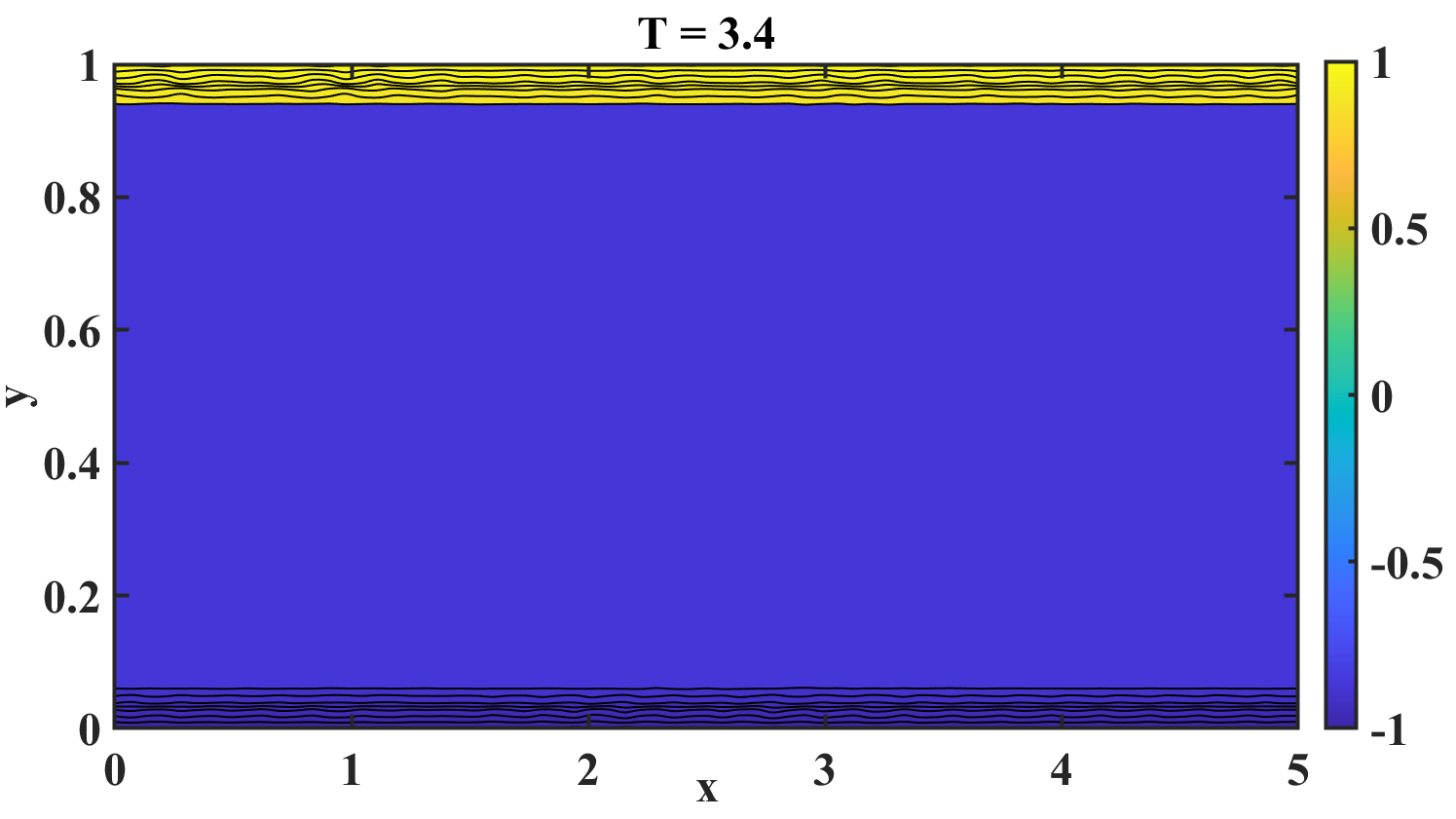}
\includegraphics[width=0.49\linewidth, height=0.3\linewidth]{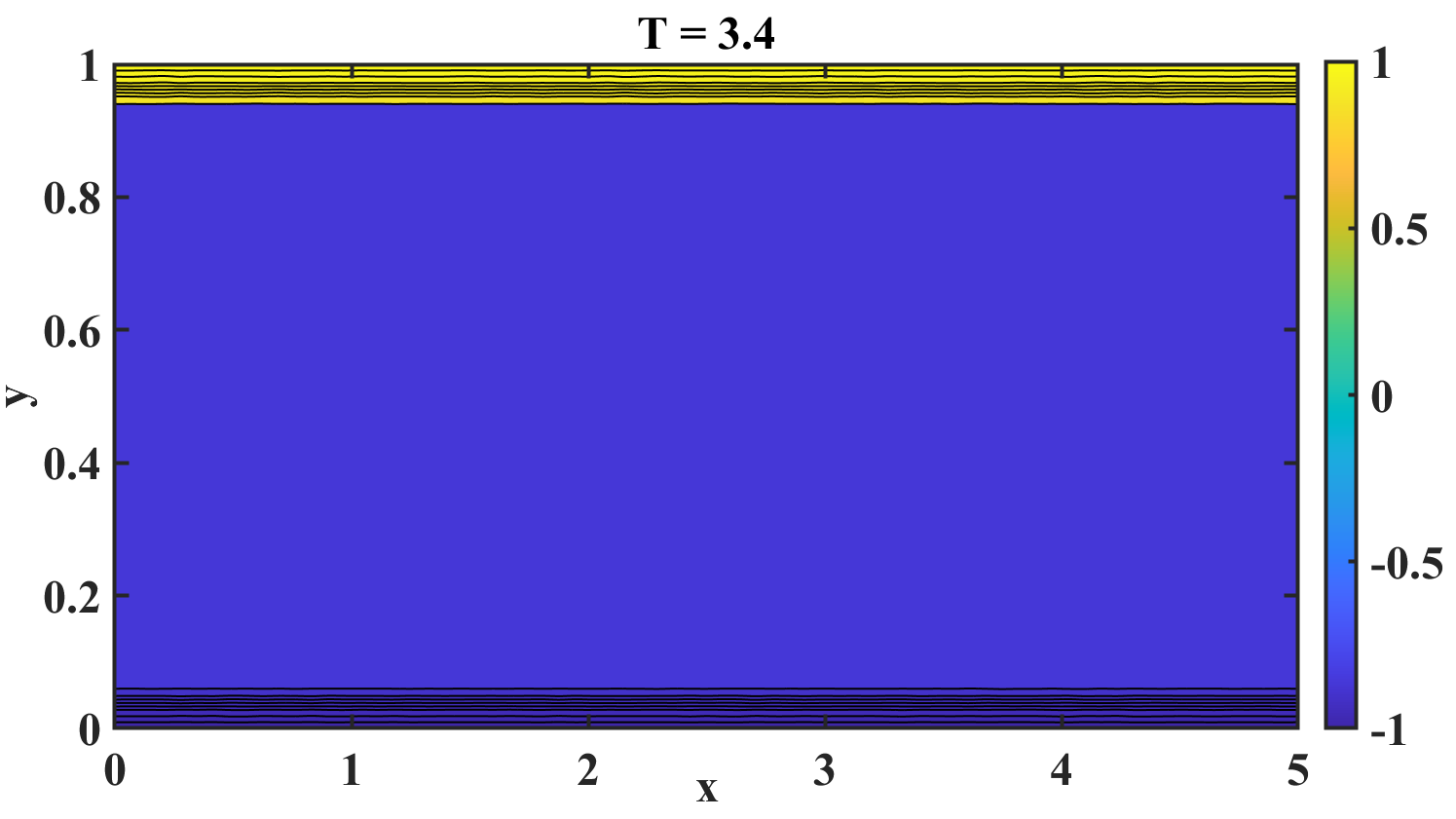}
\vskip 1pt
\includegraphics[width=0.49\linewidth, height=0.3\linewidth]{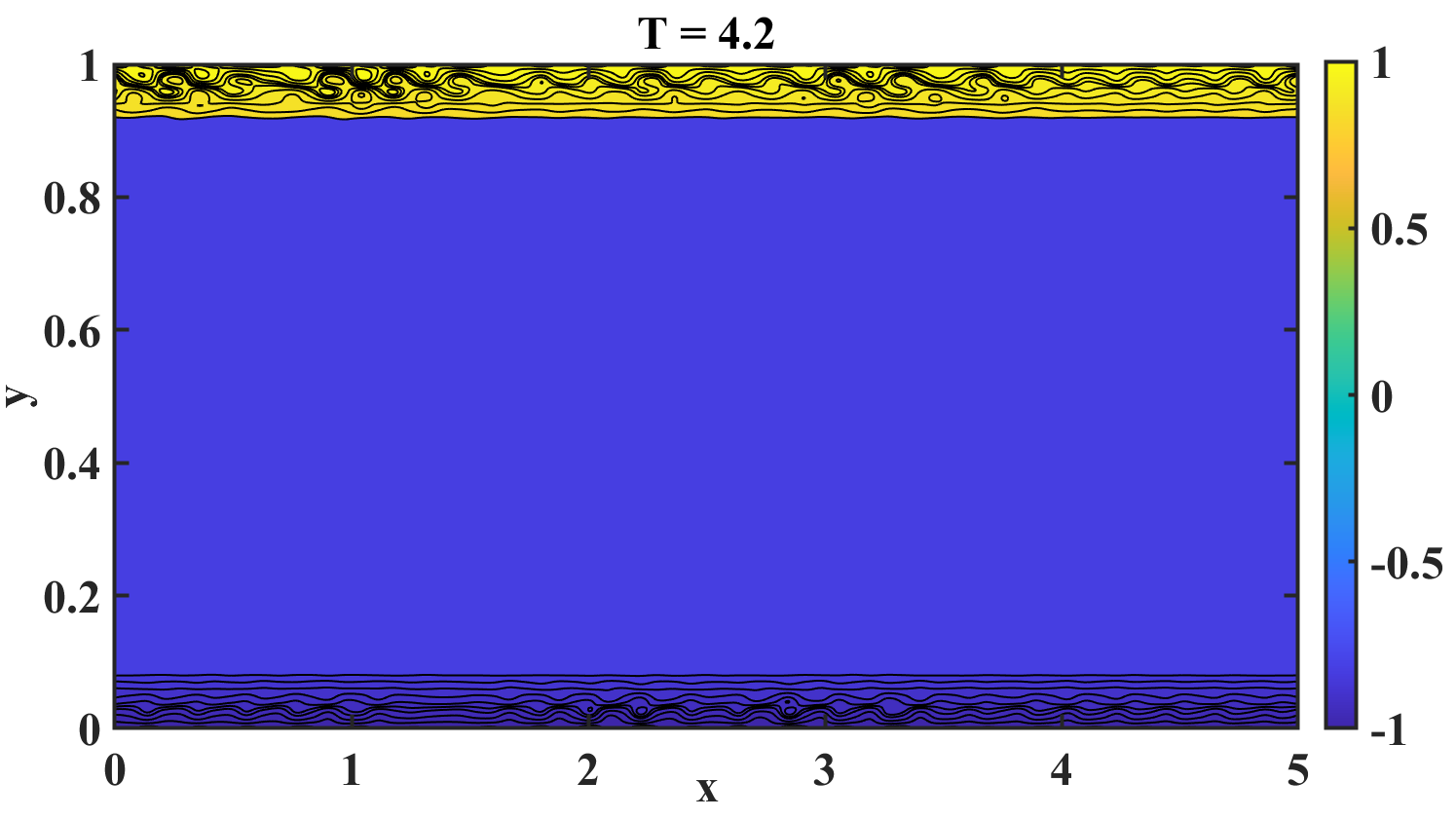}
\includegraphics[width=0.49\linewidth, height=0.3\linewidth]{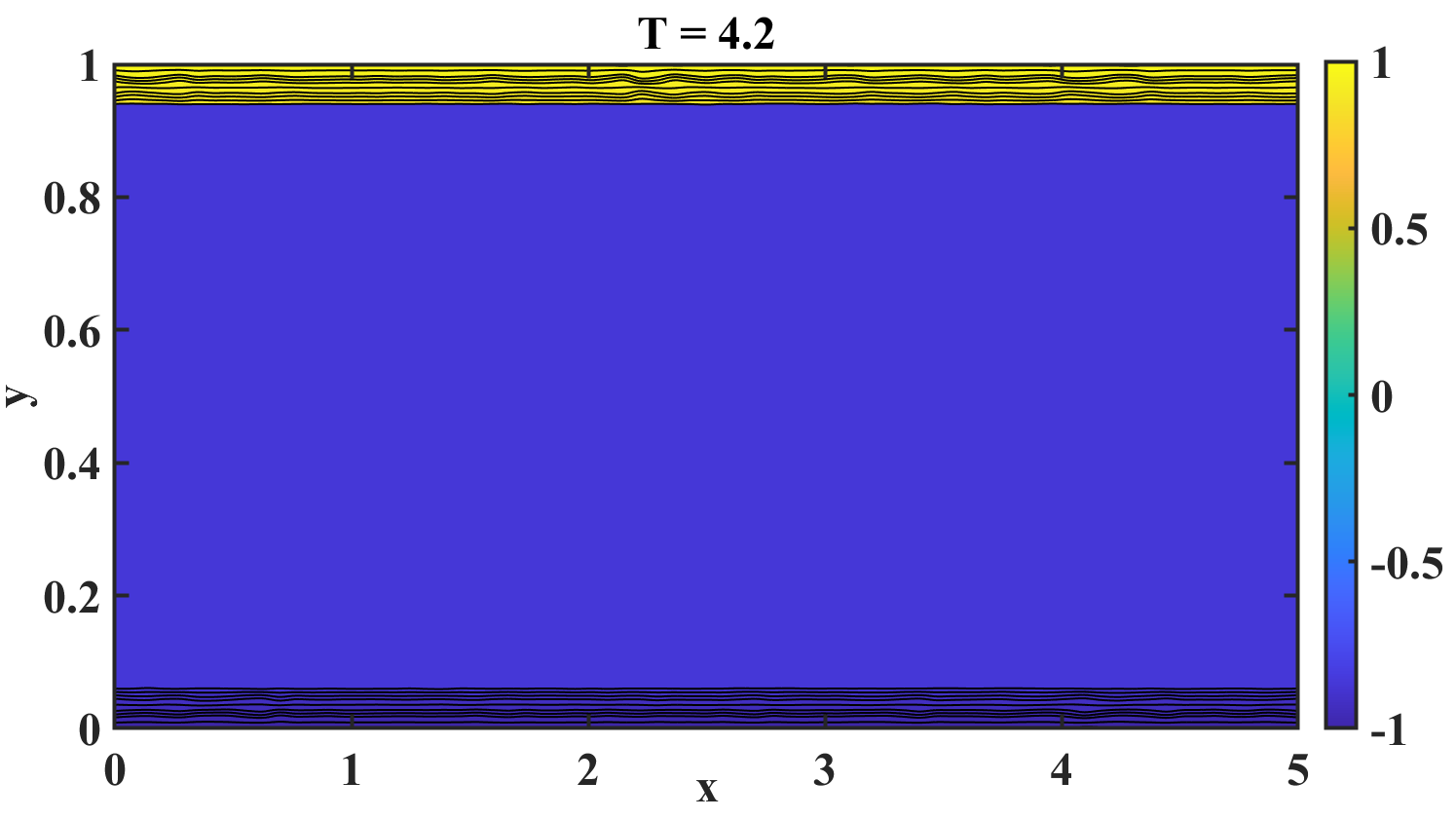}
\vskip 1pt
\includegraphics[width=0.49\linewidth, height=0.3\linewidth]{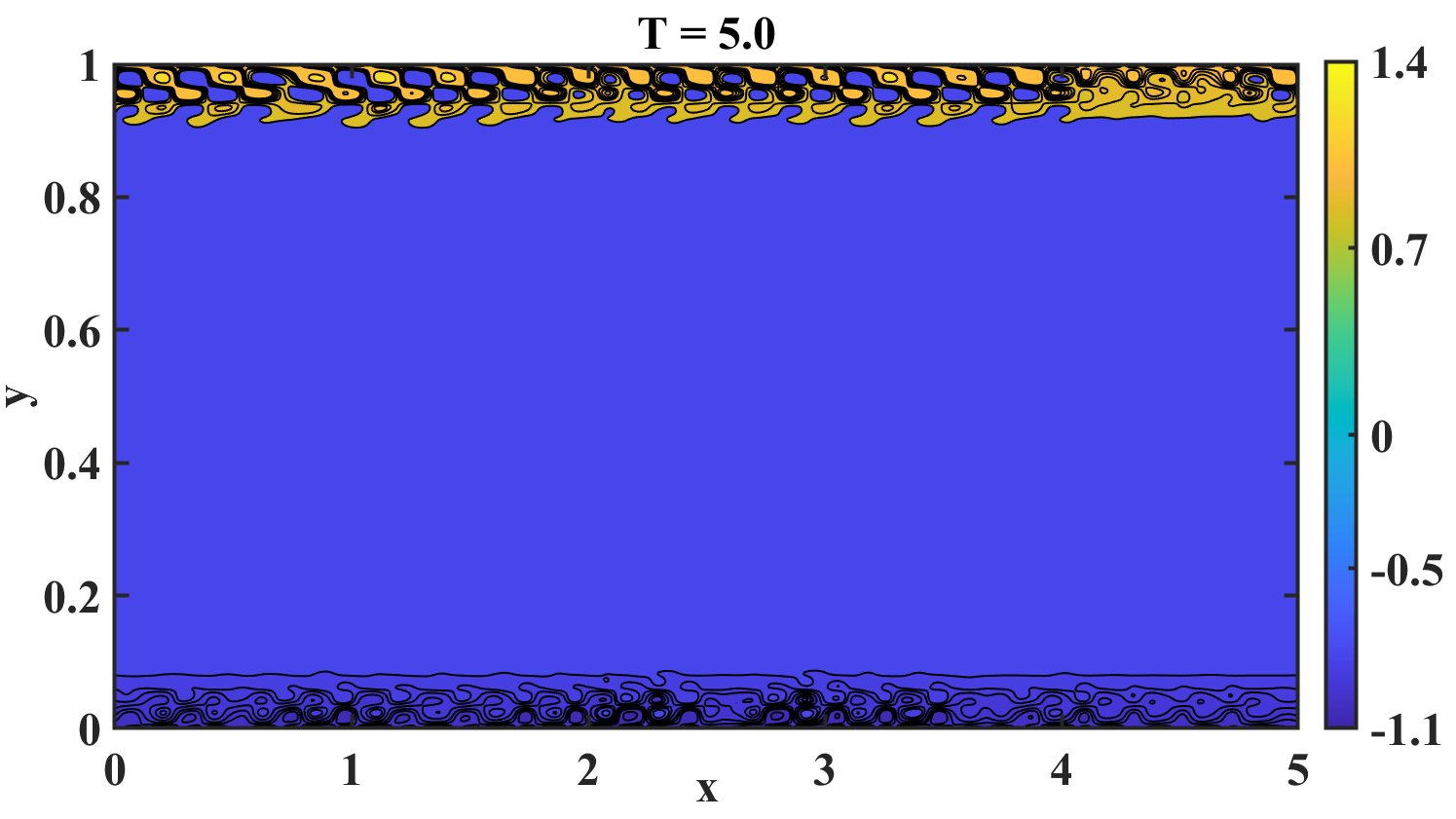}
\includegraphics[width=0.49\linewidth, height=0.3\linewidth]{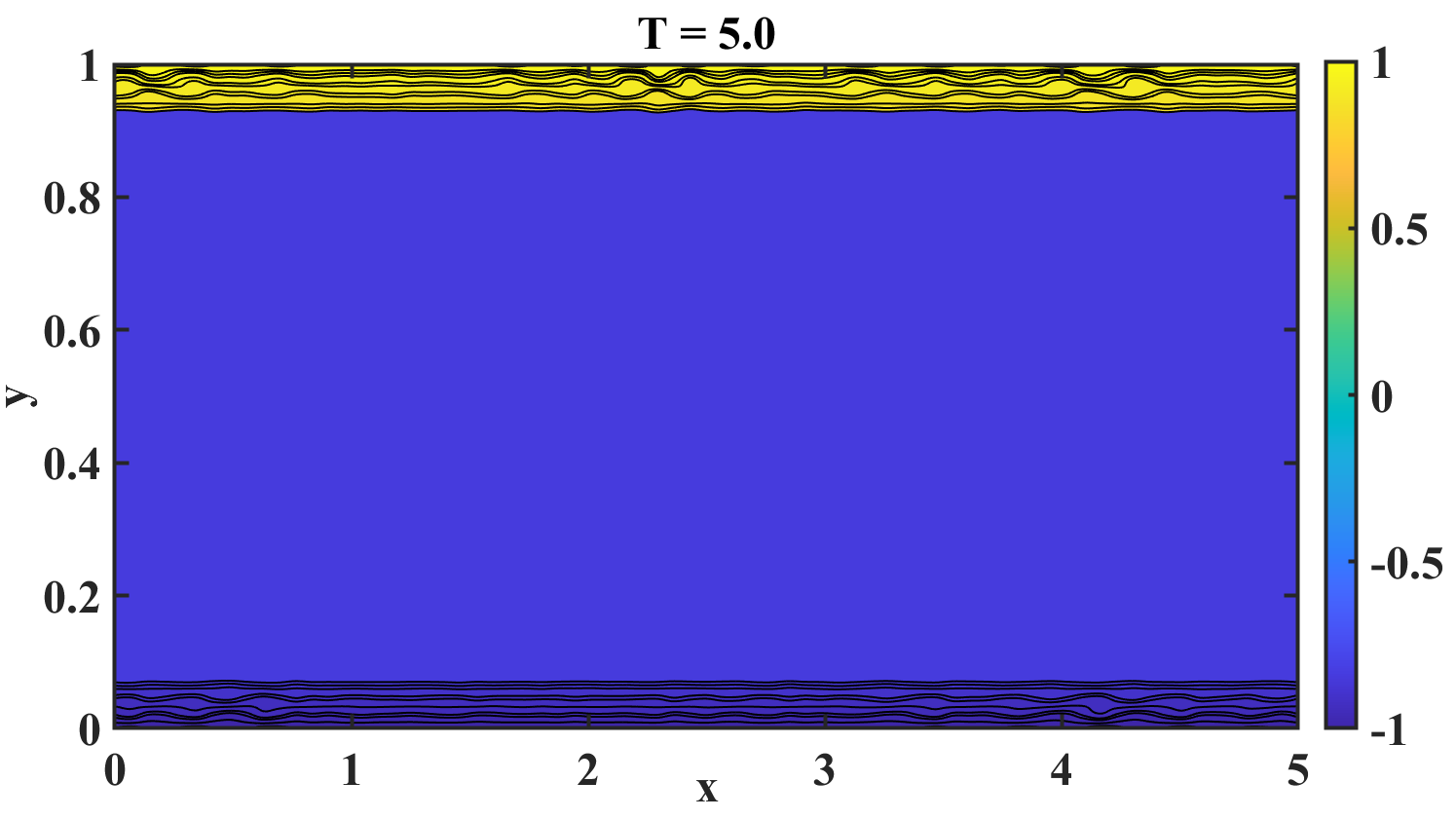}
\vskip 1pt
\includegraphics[width=0.49\linewidth, height=0.3\linewidth]{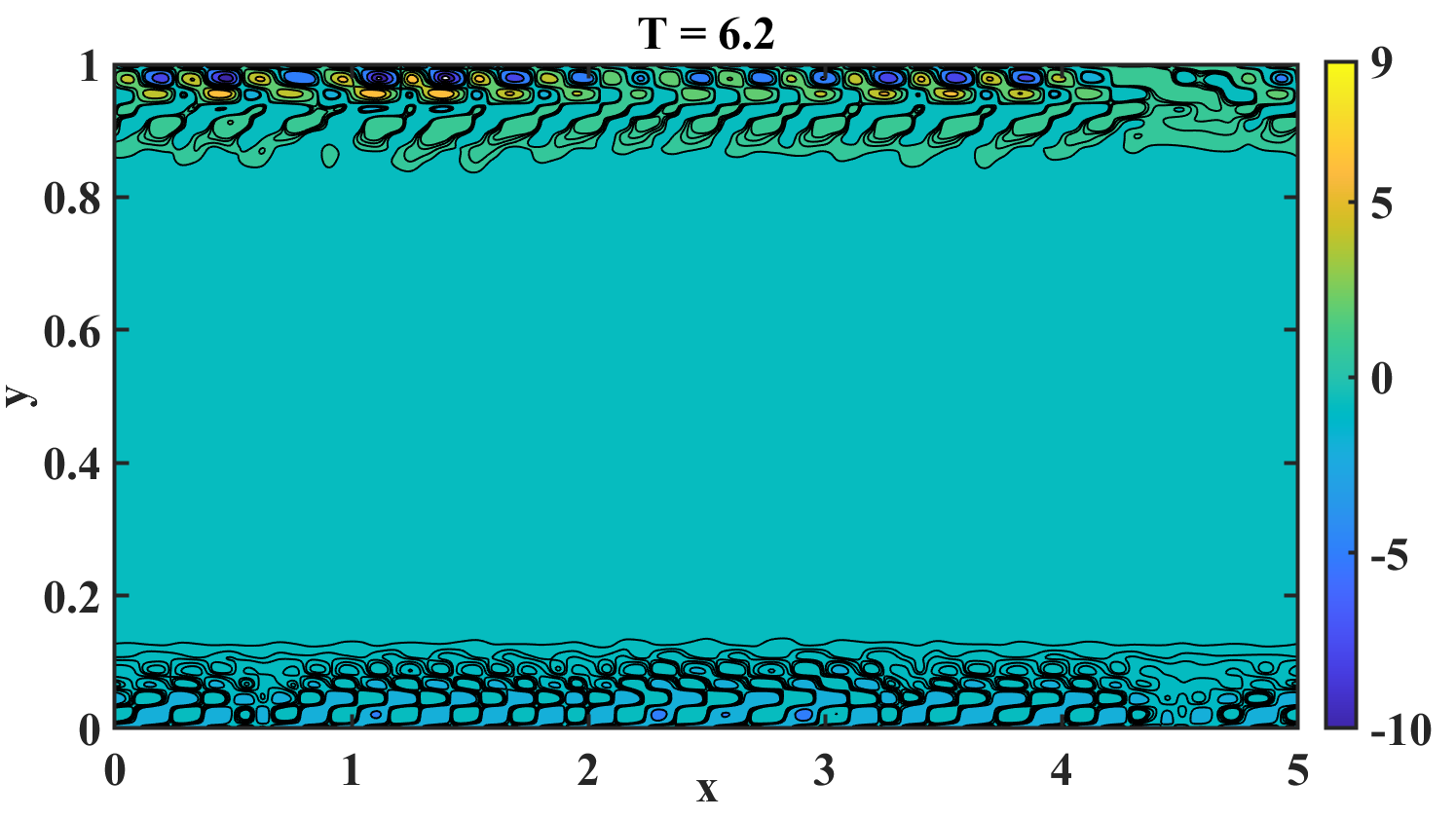}
\includegraphics[width=0.49\linewidth, height=0.3\linewidth]{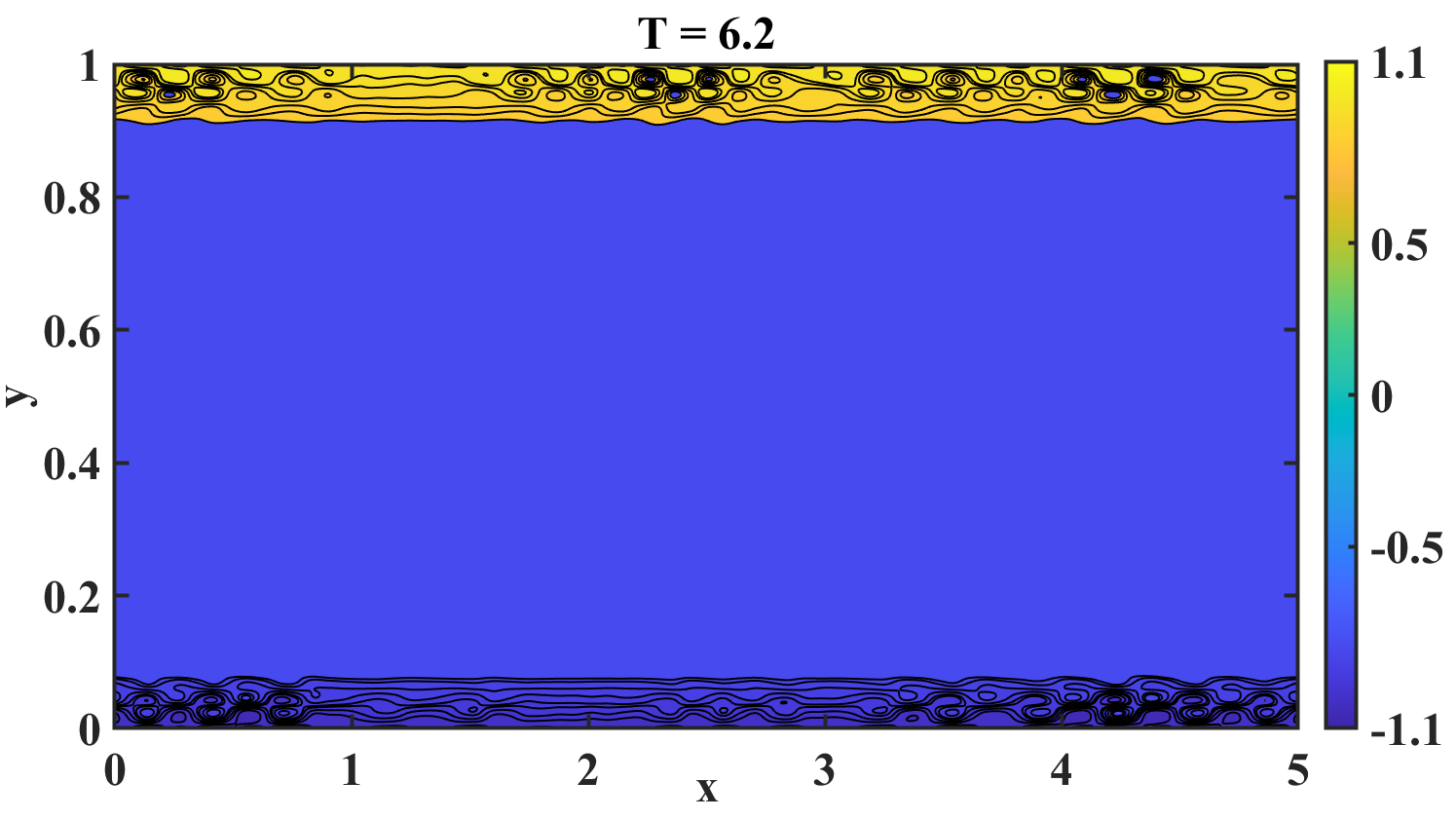}
\caption{Vorticity contours for the elastic stress-dominated ($\nu=0.3$) Zimm model ($\alpha=0.67$) case, shown at parameter values, $We=10.0$ and $Re=70$ (left column) versus $Re=1000$ (right column).}
\label{fig:Fig7}
\end{figure}

Further, notice (on the right hand side column in figure~\ref{fig:Fig7} and~\ref{fig:Fig9}) the disappearance of the structures at larger Reynolds number. This observation can be attributed to the fact that the macrostructures are `washed out' of the channel at higher flow velocities. In an earlier work, we had shown how the spatiotemporal phase diagram indicated regions of `temporal stability' at larger values of $Re$~\cite{Chauhan2021}. In this work, we associate these regions of temporal stability with regions devoid of flow structures.
\begin{figure}[htbp]
\includegraphics[width=0.49\linewidth, height=0.3\linewidth]{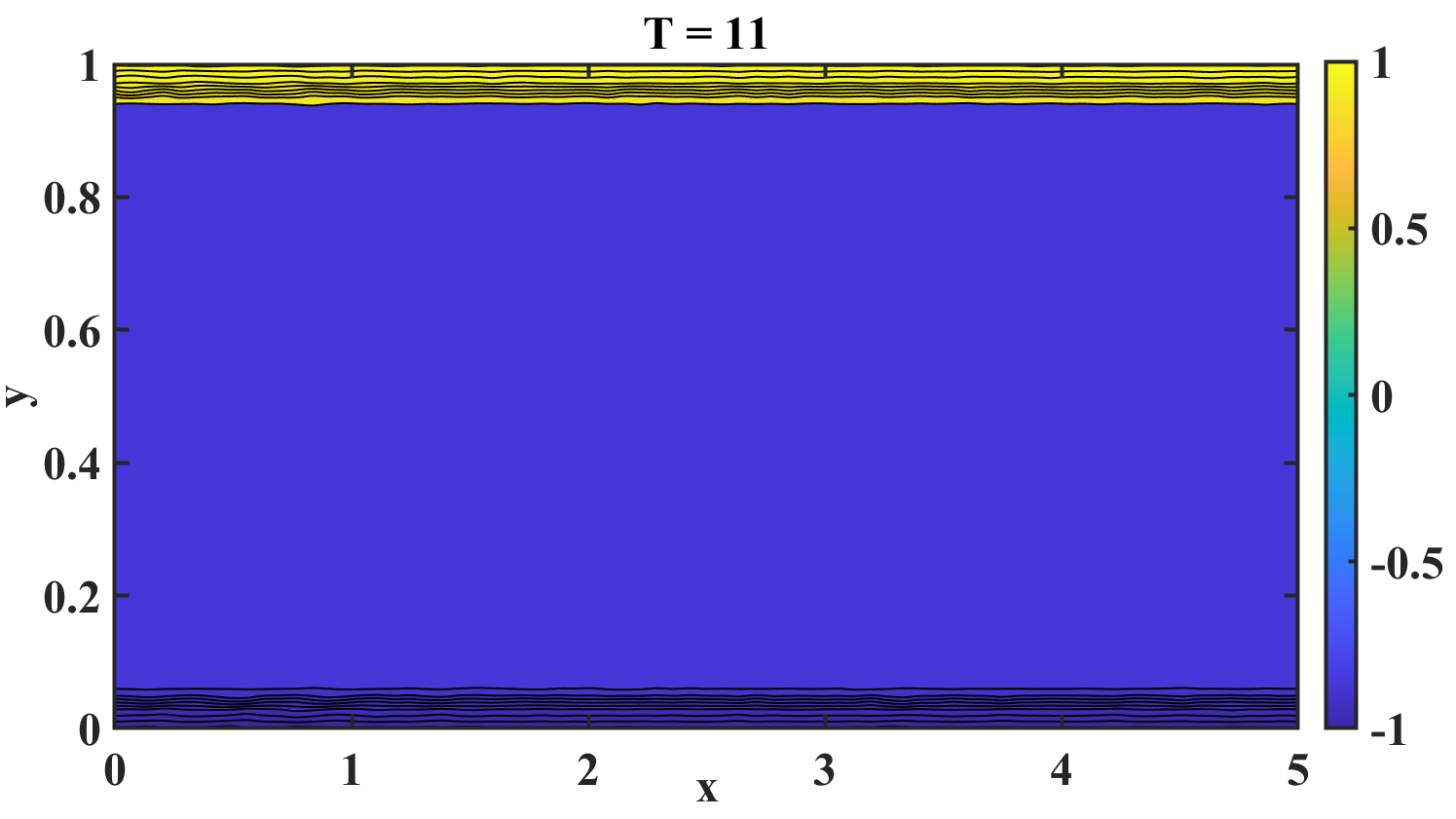}
\includegraphics[width=0.49\linewidth, height=0.3\linewidth]{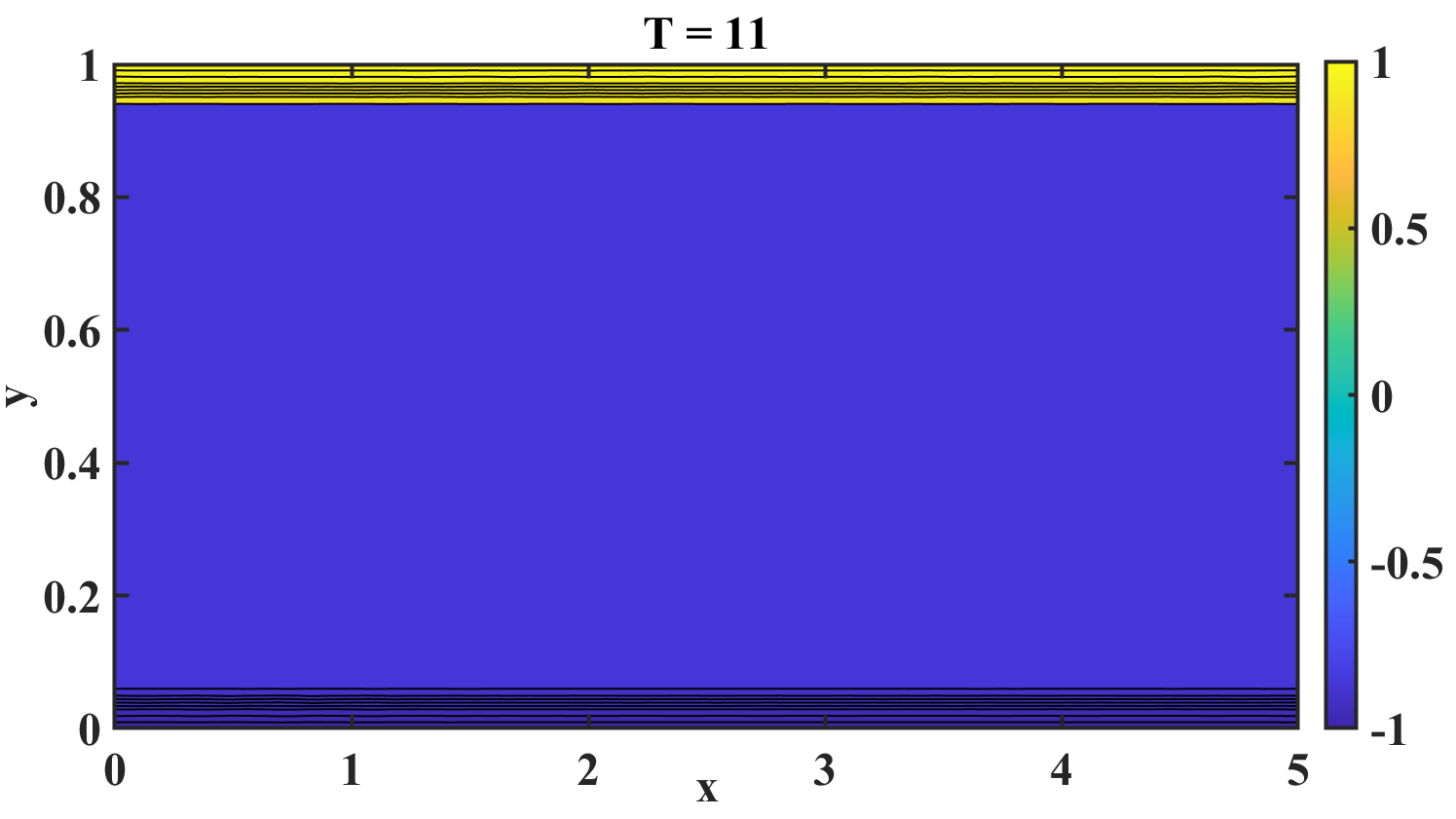}
\vskip 1pt
\includegraphics[width=0.49\linewidth, height=0.3\linewidth]{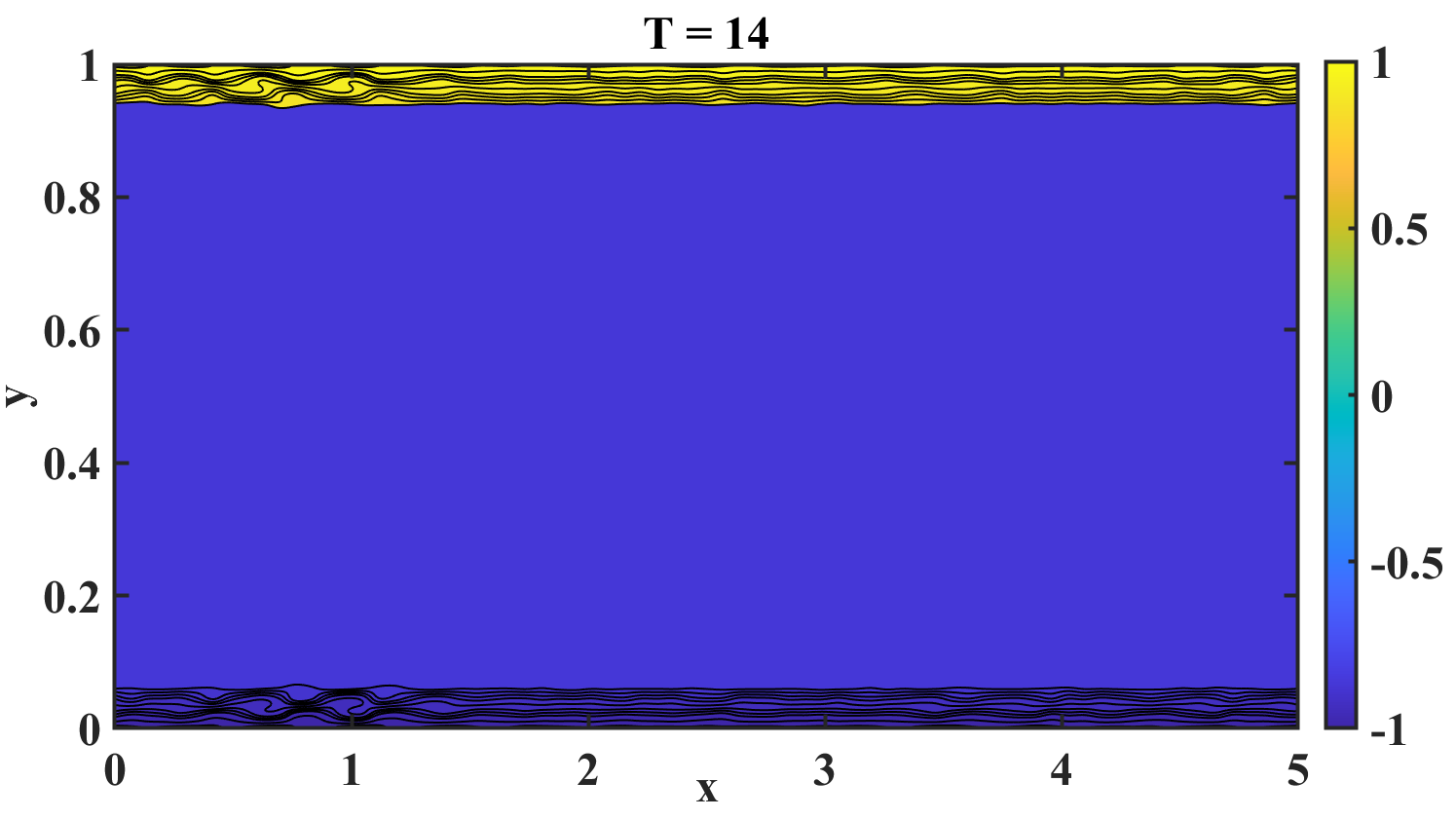}
\includegraphics[width=0.49\linewidth, height=0.3\linewidth]{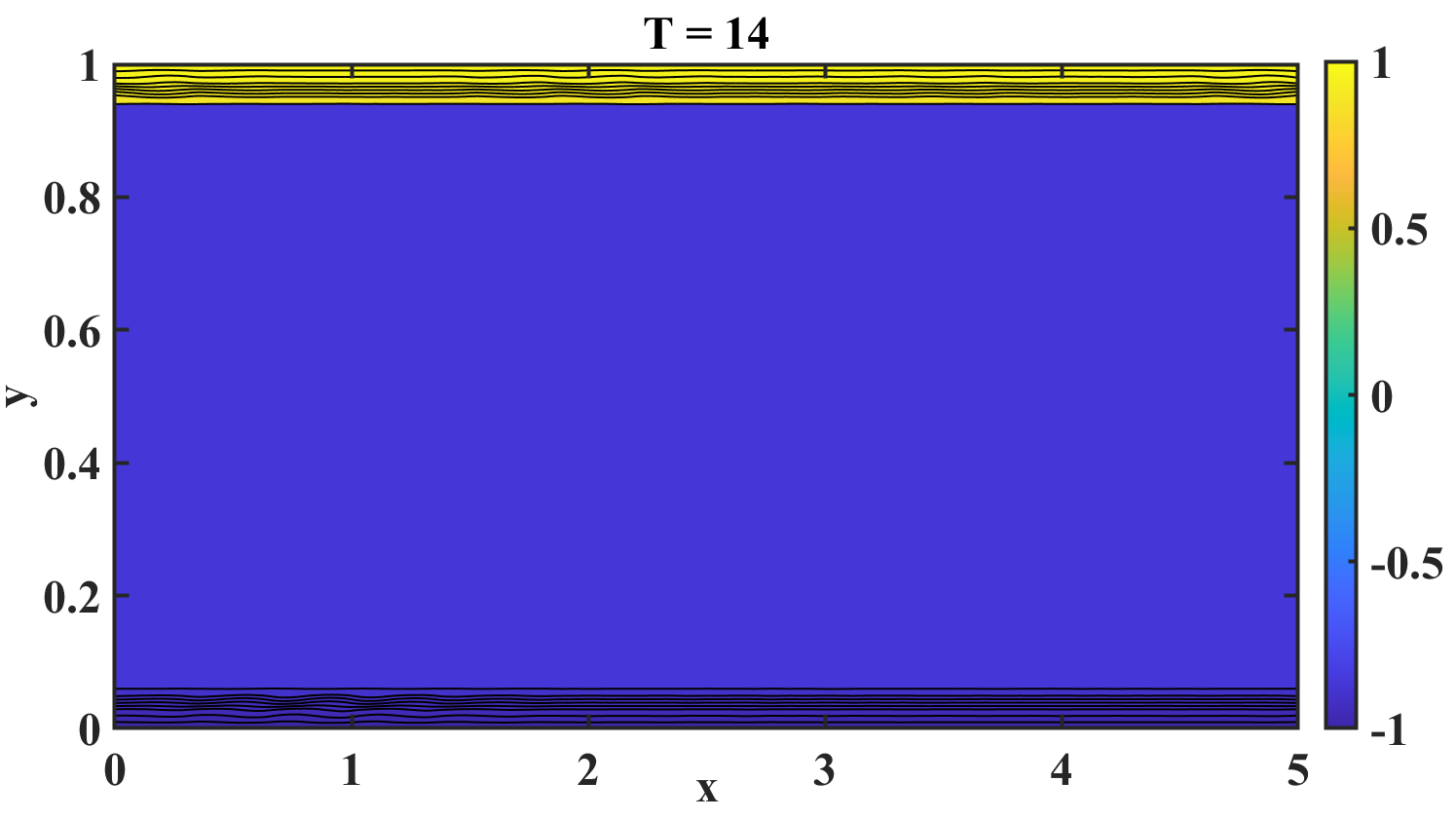}
\vskip 1pt
\includegraphics[width=0.49\linewidth, height=0.3\linewidth]{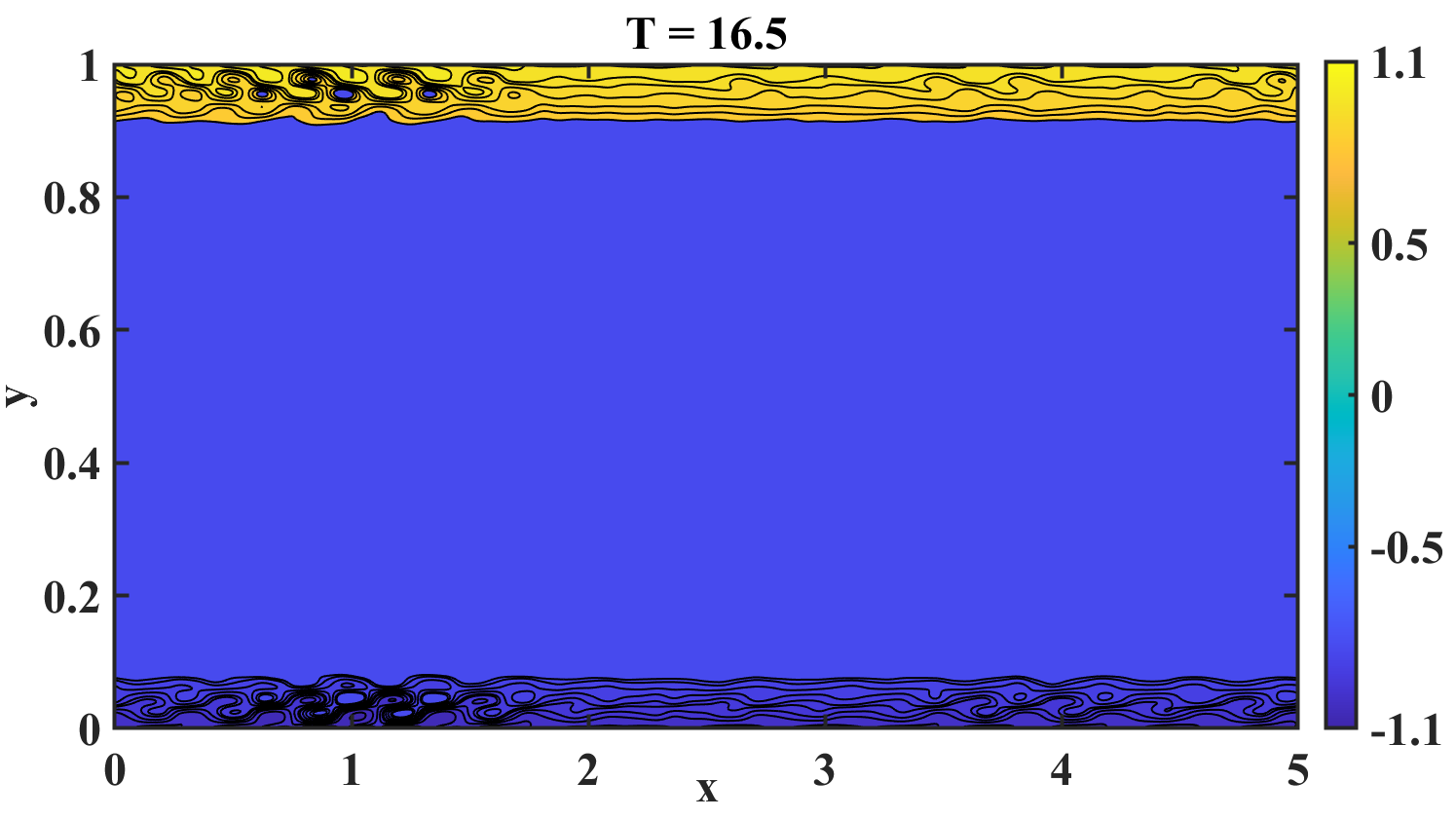}
\includegraphics[width=0.49\linewidth, height=0.3\linewidth]{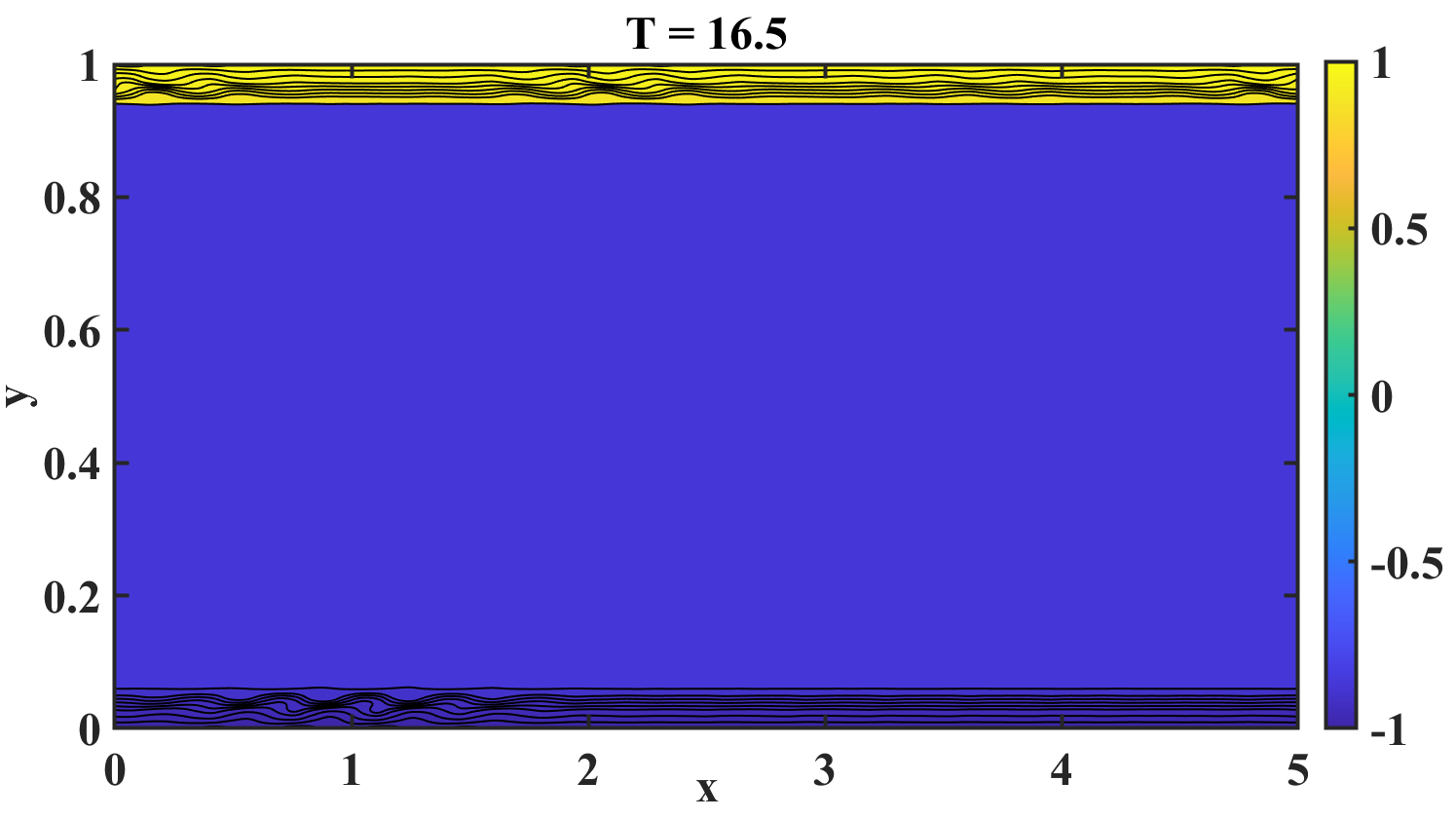}
\vskip 1pt
\includegraphics[width=0.49\linewidth, height=0.3\linewidth]{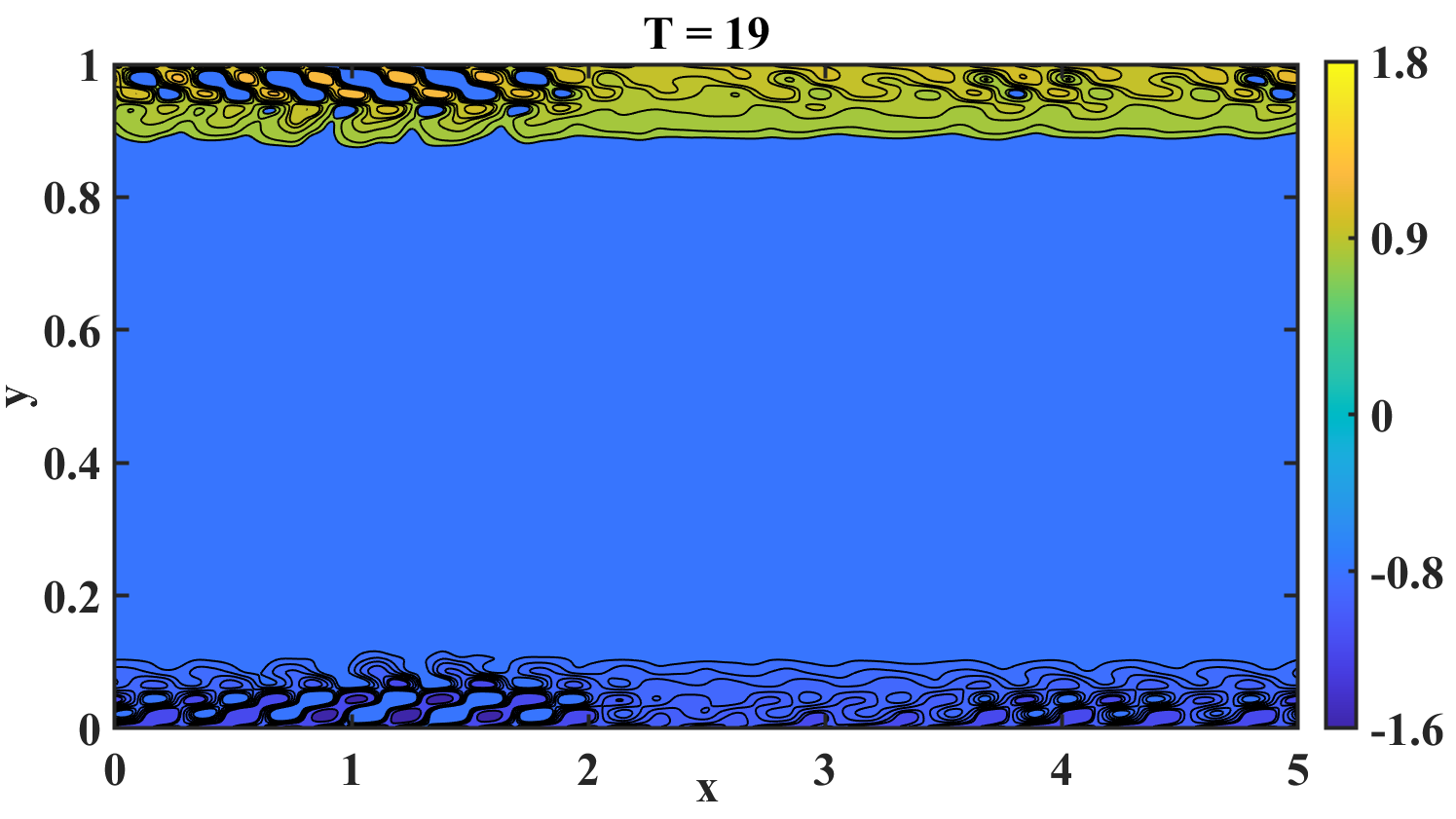}
\includegraphics[width=0.49\linewidth, height=0.3\linewidth]{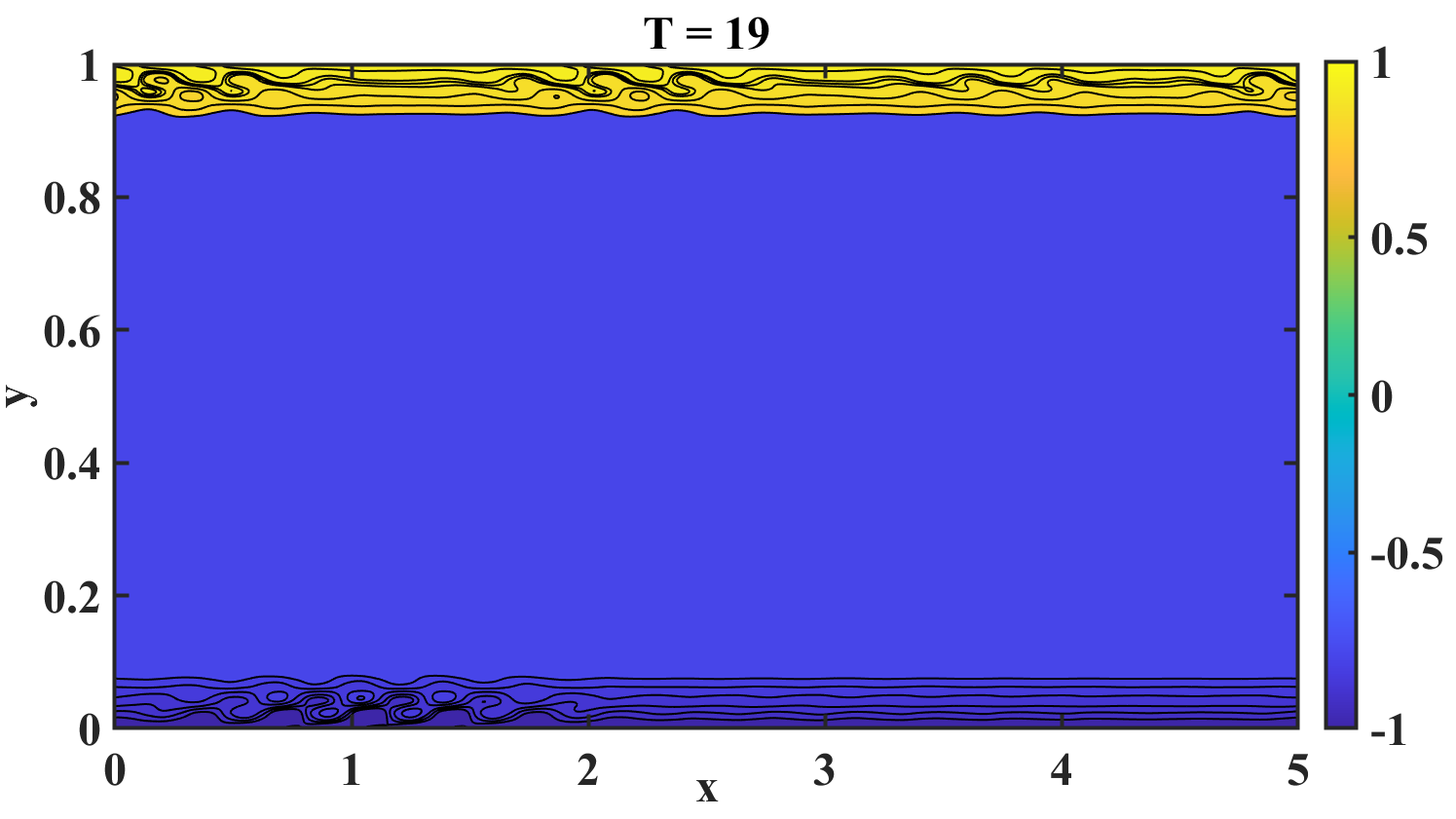}
\caption{Vorticity contours for the viscous stress-dominated ($\nu=0.6$) Zimm model ($\alpha=0.67$) case, shown at parameter values, $We=10.0$ and $Re=70$ (left column) versus $Re=1000$ (right column).}
\label{fig:Fig9}
\end{figure}

\paragraph{Rouse model:} The Rouse model represents `thicker' fluid, or fluids with slower diffusion (figures~\ref{fig:Fig6} and~\ref{fig:Fig8}). After comparing the respective range of vorticity contours at lower as well as higher values of $Re$, we find that the macrostructures are more prominent (both in size and magnitude) in the Rouse model (in comparison with the Zimm's model). Even within the Rouse model, we find that the elastic stress dominated case ($\nu=0.3$, figure~\ref{fig:Fig6}) exhibits formation of larger structures, than the viscous stress dominated case (figure~\ref{fig:Fig8}). Again, we note that the structure formation is conspicuously absent at larger values of $Re$ (i.~e., notice the plots on the right hand side columns in figures~\ref{fig:Fig6} and~\ref{fig:Fig8}).
\begin{figure}[htbp]
\includegraphics[width=0.49\linewidth, height=0.3\linewidth]{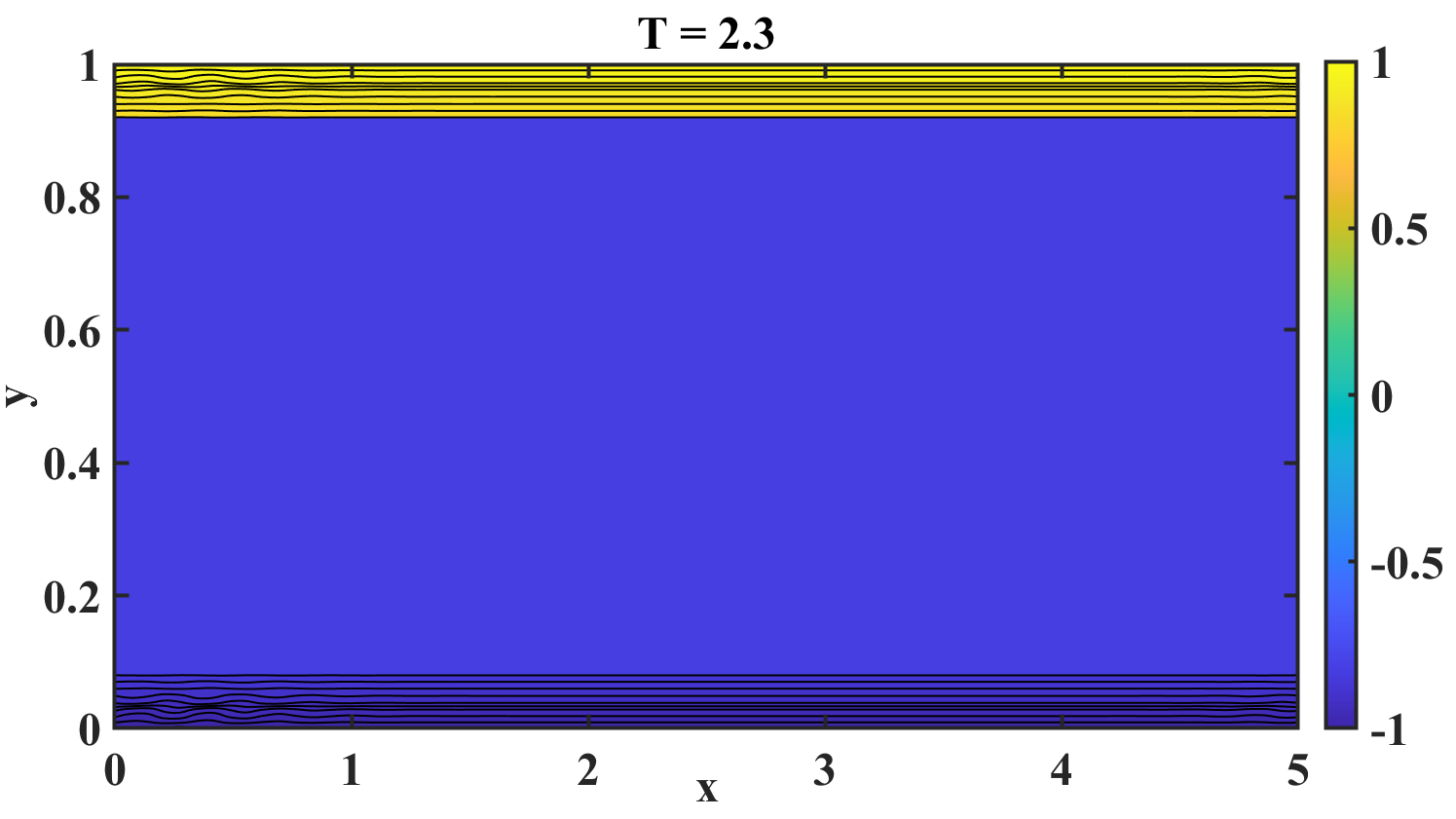}
\includegraphics[width=0.49\linewidth, height=0.3\linewidth]{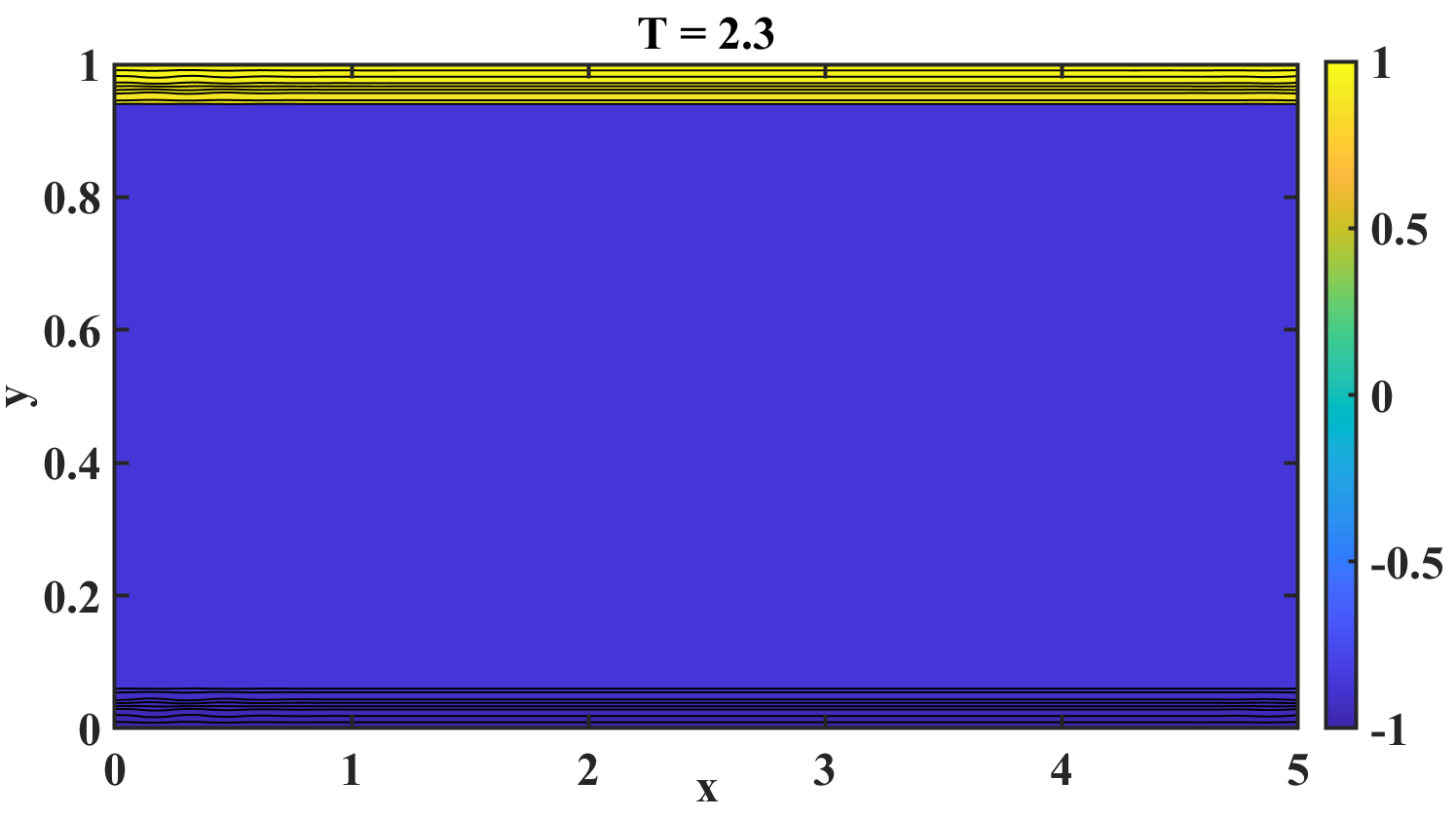}
\vskip 1pt
\includegraphics[width=0.49\linewidth, height=0.3\linewidth]{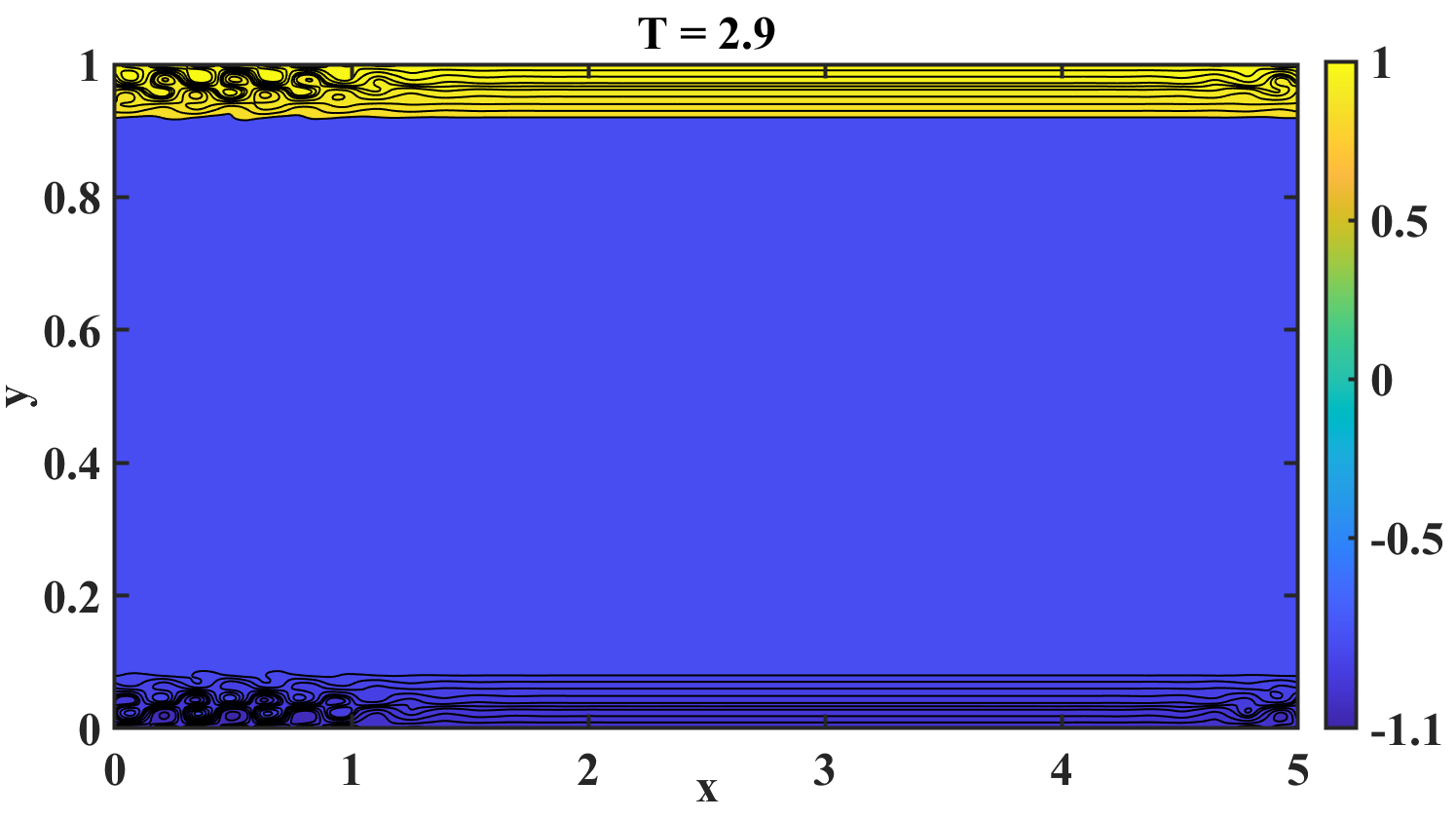}
\includegraphics[width=0.49\linewidth, height=0.3\linewidth]{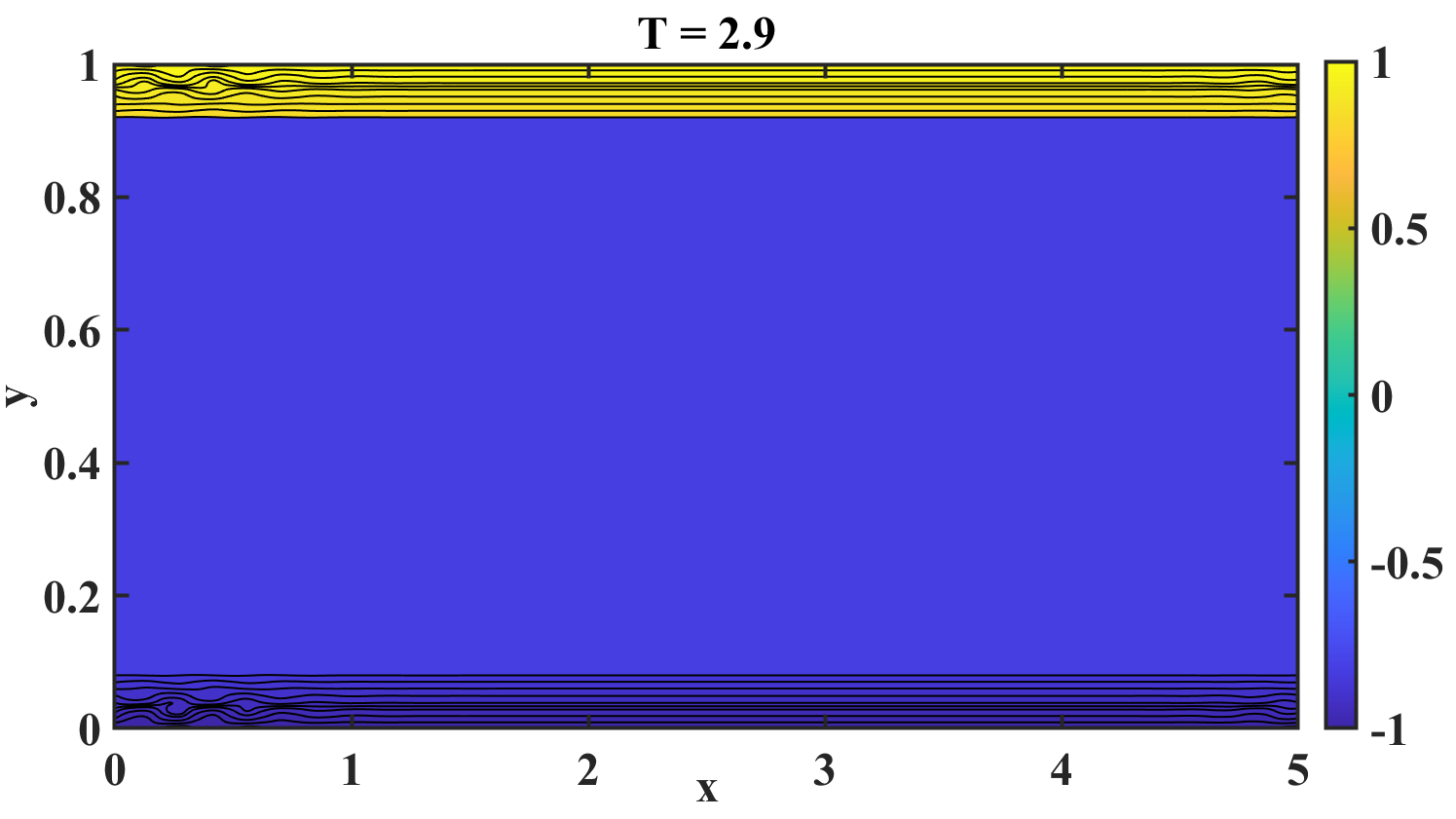}
\vskip 1pt
\includegraphics[width=0.49\linewidth, height=0.3\linewidth]{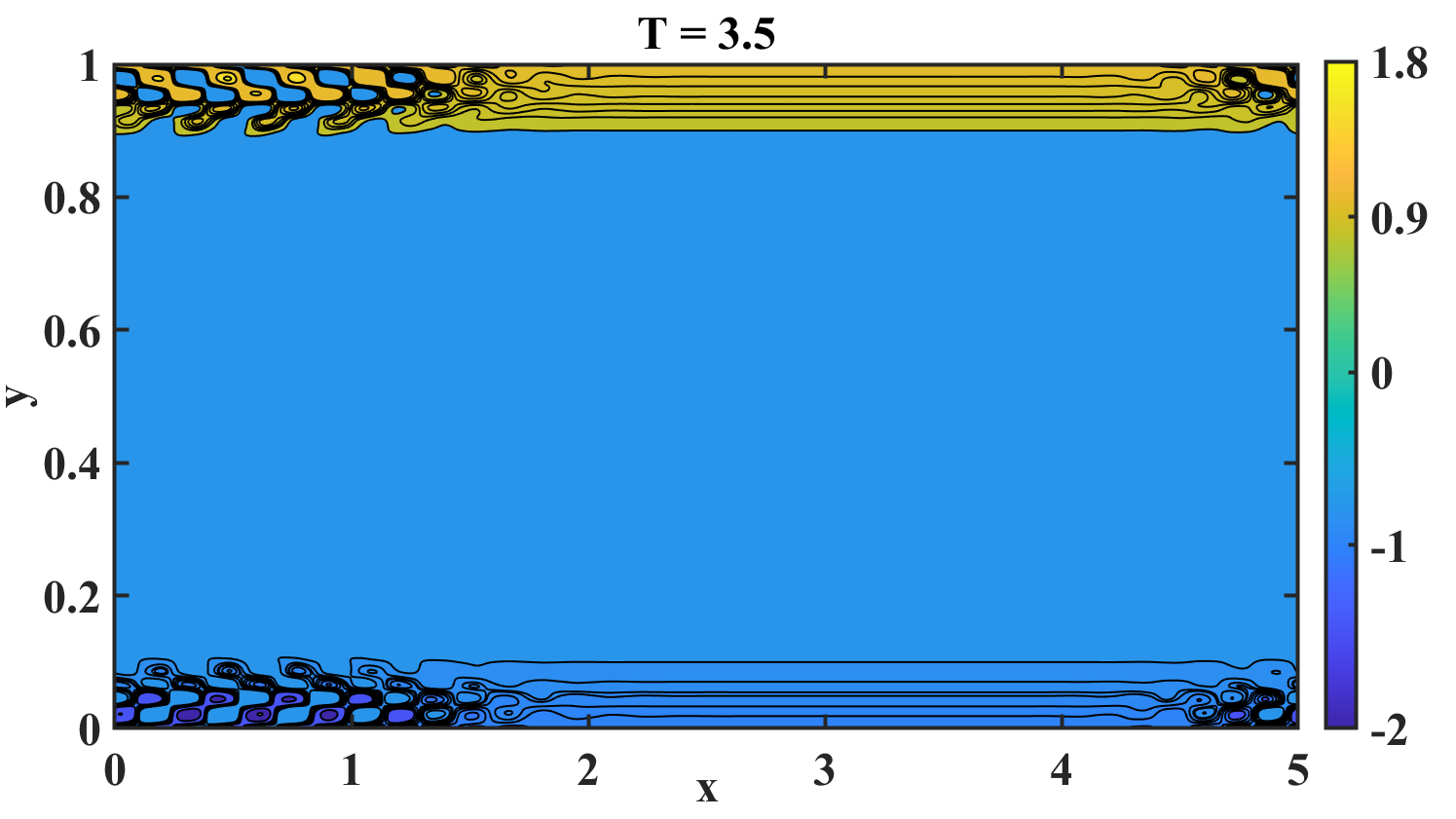}
\includegraphics[width=0.49\linewidth, height=0.3\linewidth]{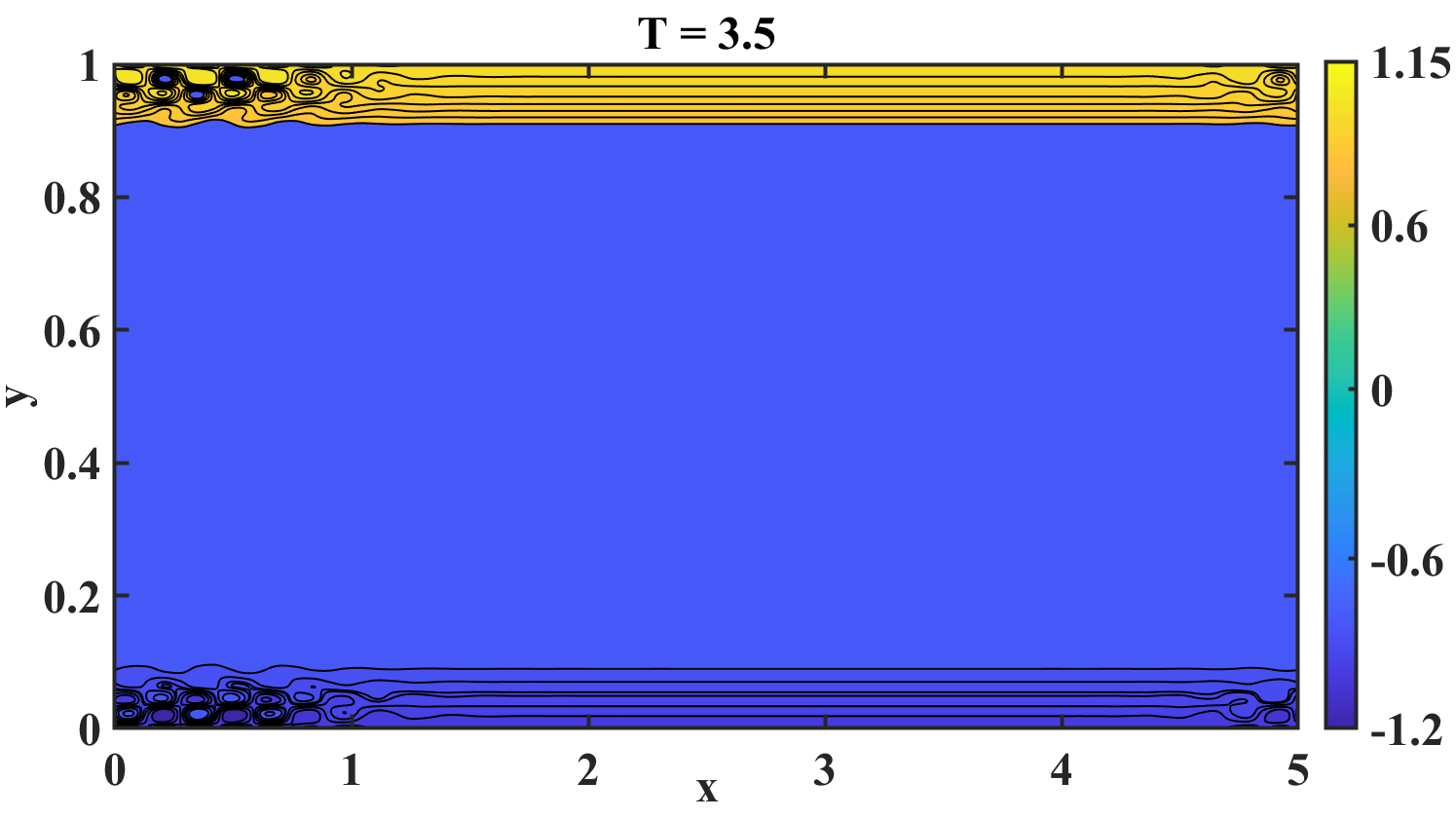}
\vskip 1pt
\includegraphics[width=0.49\linewidth, height=0.3\linewidth]{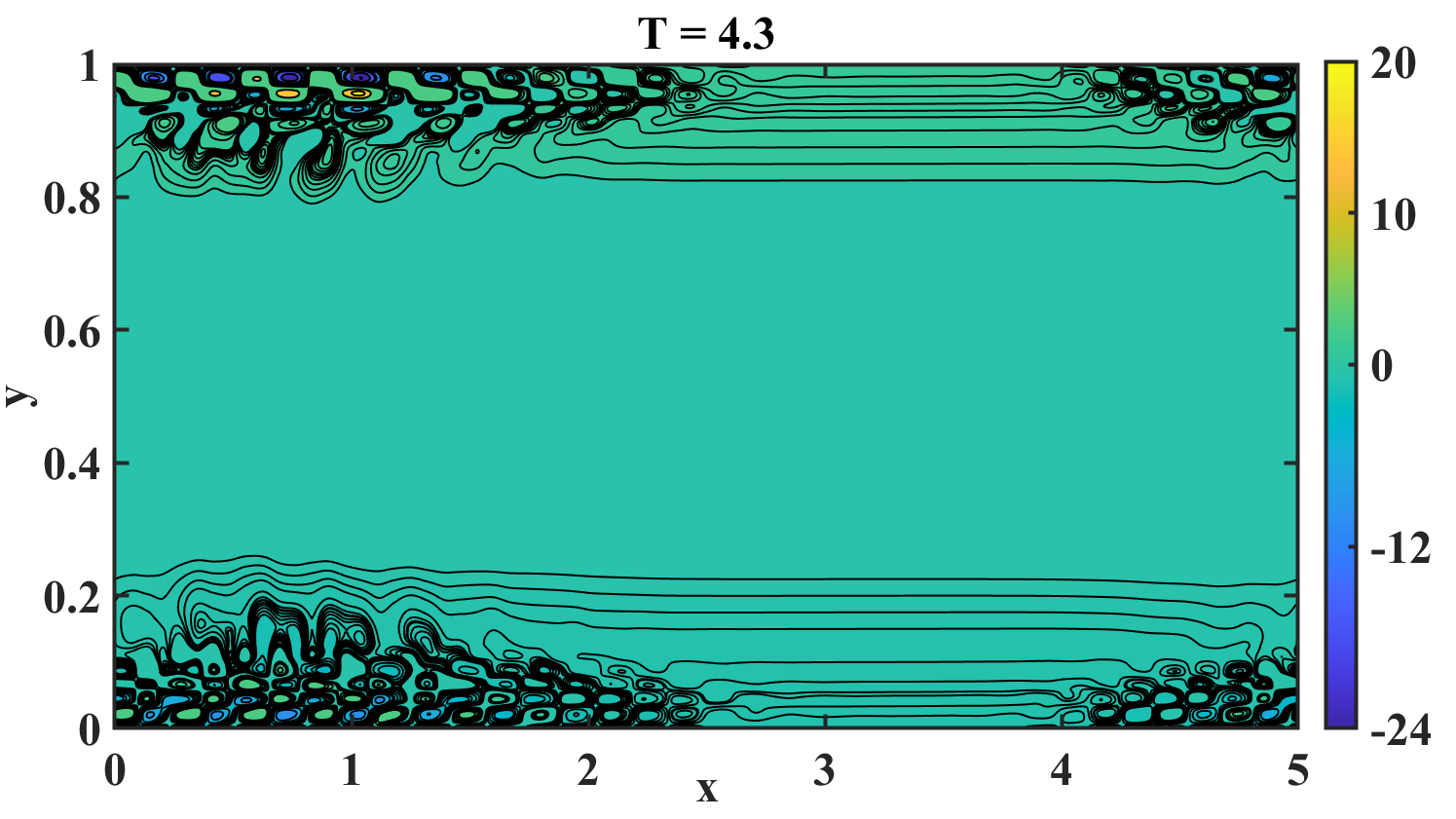}
\includegraphics[width=0.49\linewidth, height=0.3\linewidth]{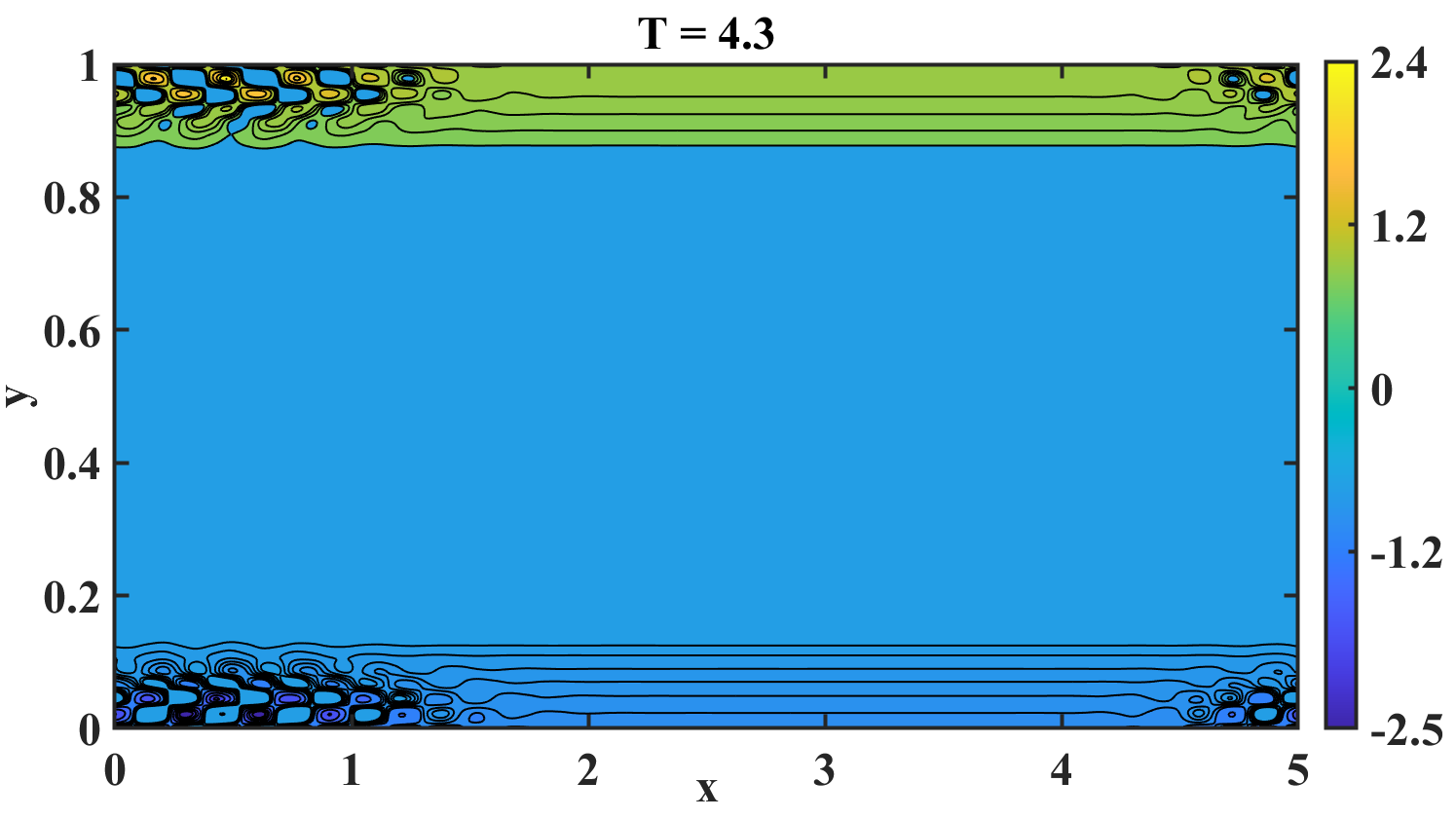}
\caption{Vorticity contours for the elastic stress-dominated ($\nu=0.3$) Rouse model ($\alpha=0.5$) case, shown at parameter values, $We=10.0$ and $Re=70$ (left column) versus $Re=1000$ (right column).}
\label{fig:Fig6}
\end{figure}

We summarize our discussion by noting that for the selected values of the parameters, $\nu, \alpha$ and $We$: (a) The Rouse model is comparatively more unstable than the Zimm's model at low fluid inertia (or low values of $Re$), (b) The elastic stress dominated case ($\nu<0.5$ case) is comparatively more unstable than the viscous stress dominated case, and (c) Temporal stability is achieved at higher values of $Re$, irrespective of the model or the polymer concentration.

A notably `abnormal' feature in our numerical simulations is the presence of temporal stability at high inertia. While the in silico studies of the classical Oldroyd-B channel flows indicate the appearance of temporal instability for Reynolds number as low as $Re \sim 50$~\cite{Khalid2021}, temporal stability at high fluid inertia for viscoelastic flows is only recognized in experimental realizations (until now). For example, Riley~\cite{Riley1988} reported an elasticity induced flow stabilization of viscoelastic fluids coated over complaint surfaces at {\color{black}a} fairly high Reynolds number ($Re \sim 4000$). In a separate study involving ethanol gel fuels, elastic stabilization at {\color{black}a} high shear rate was attributed due to an abnormally high second normal stress difference~\cite{Nandagopalan2018}. Viscoelastic flow stabilization at higher values of $Re$, in tapered microchannels, was explained due to the presence of wall effects~\cite{Zarabadi2019}. In another in vitro study, a biofilm deacidification created a non-homogeneous environment for molecular diffusion, leading to a `subdiffusive effect' with hindered flow rates~\cite{Zarabadi2018}. These in vitro studies not only corroborate our numerical outcome, especially establishing the emergence of temporally stable region at high inertia, but also highlight the potential of FPDE in effectively capturing the flow-instability transition in subdiffusive flows.
\begin{figure}[htbp]
\includegraphics[width=0.49\linewidth, height=0.3\linewidth]{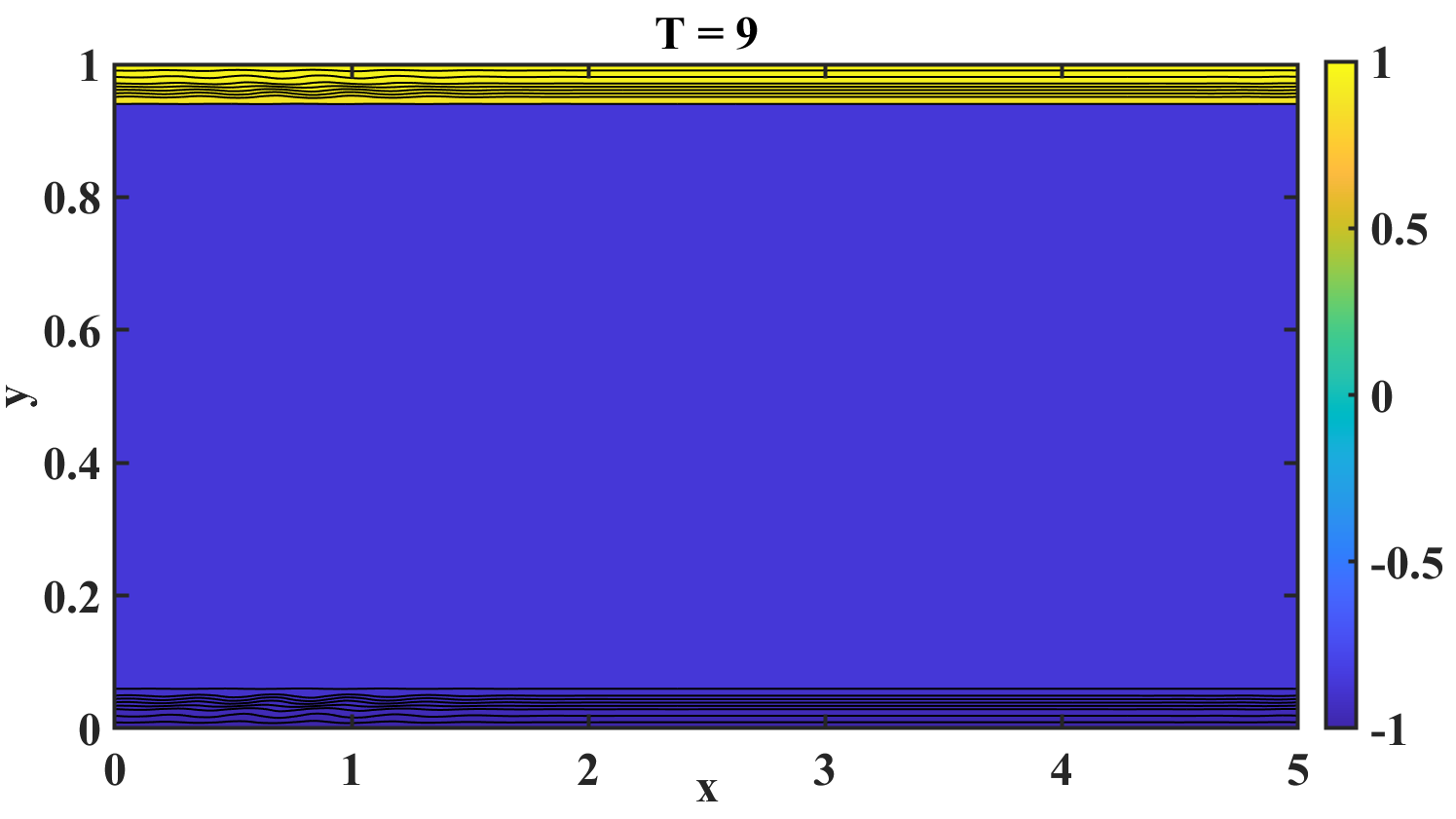}
\includegraphics[width=0.49\linewidth, height=0.3\linewidth]{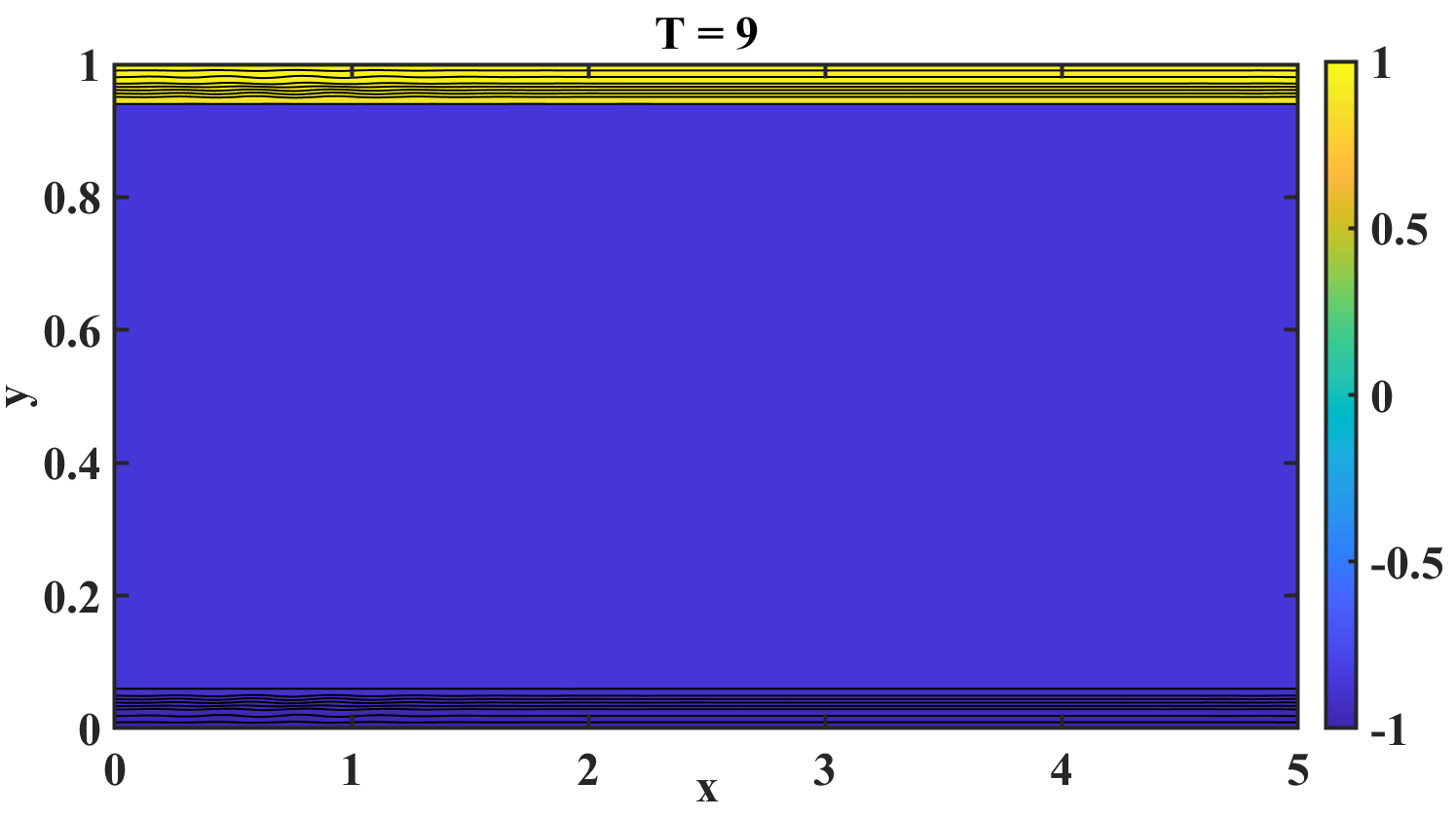}
\vskip 1pt
\includegraphics[width=0.49\linewidth, height=0.3\linewidth]{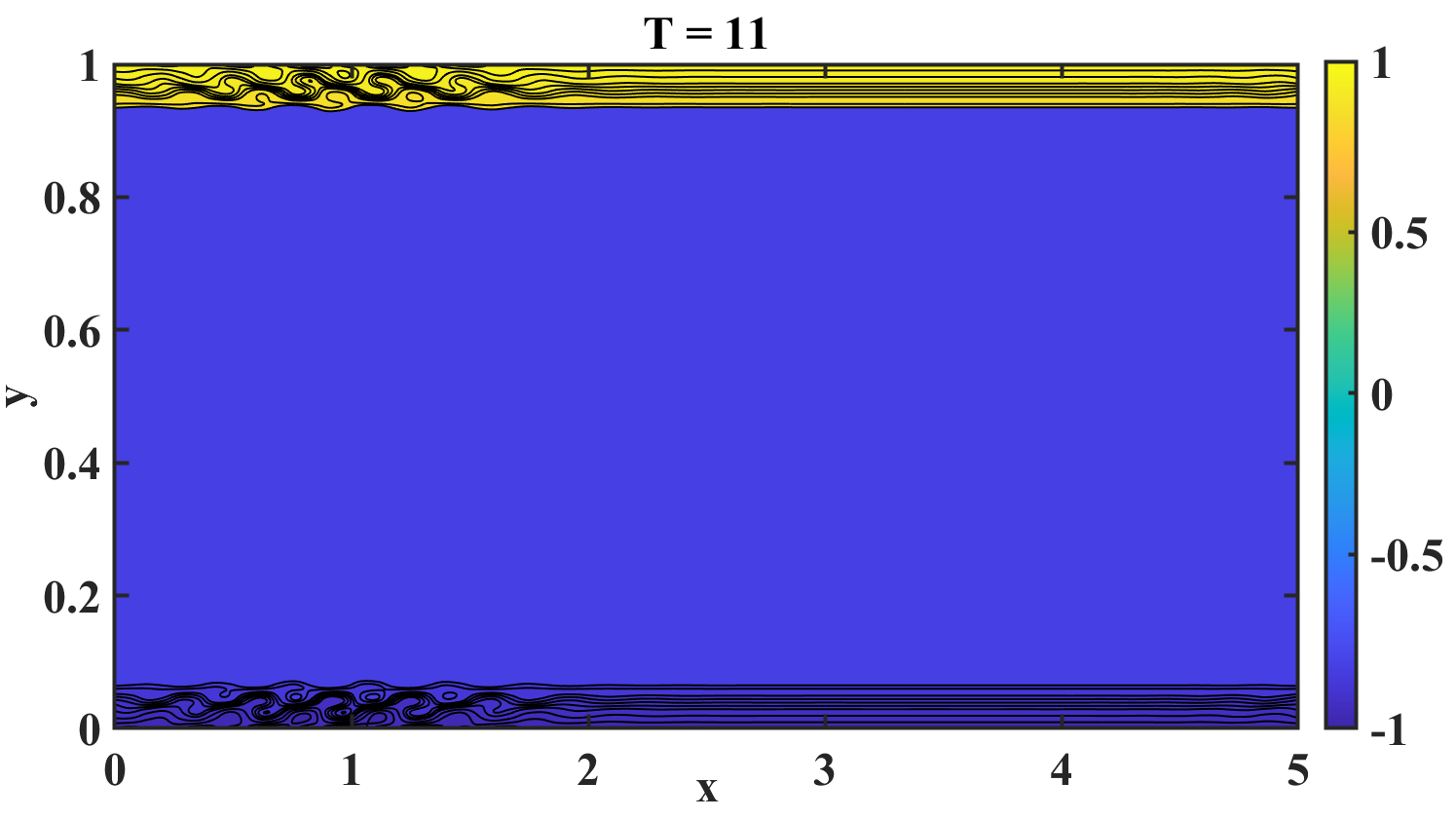}
\includegraphics[width=0.49\linewidth, height=0.3\linewidth]{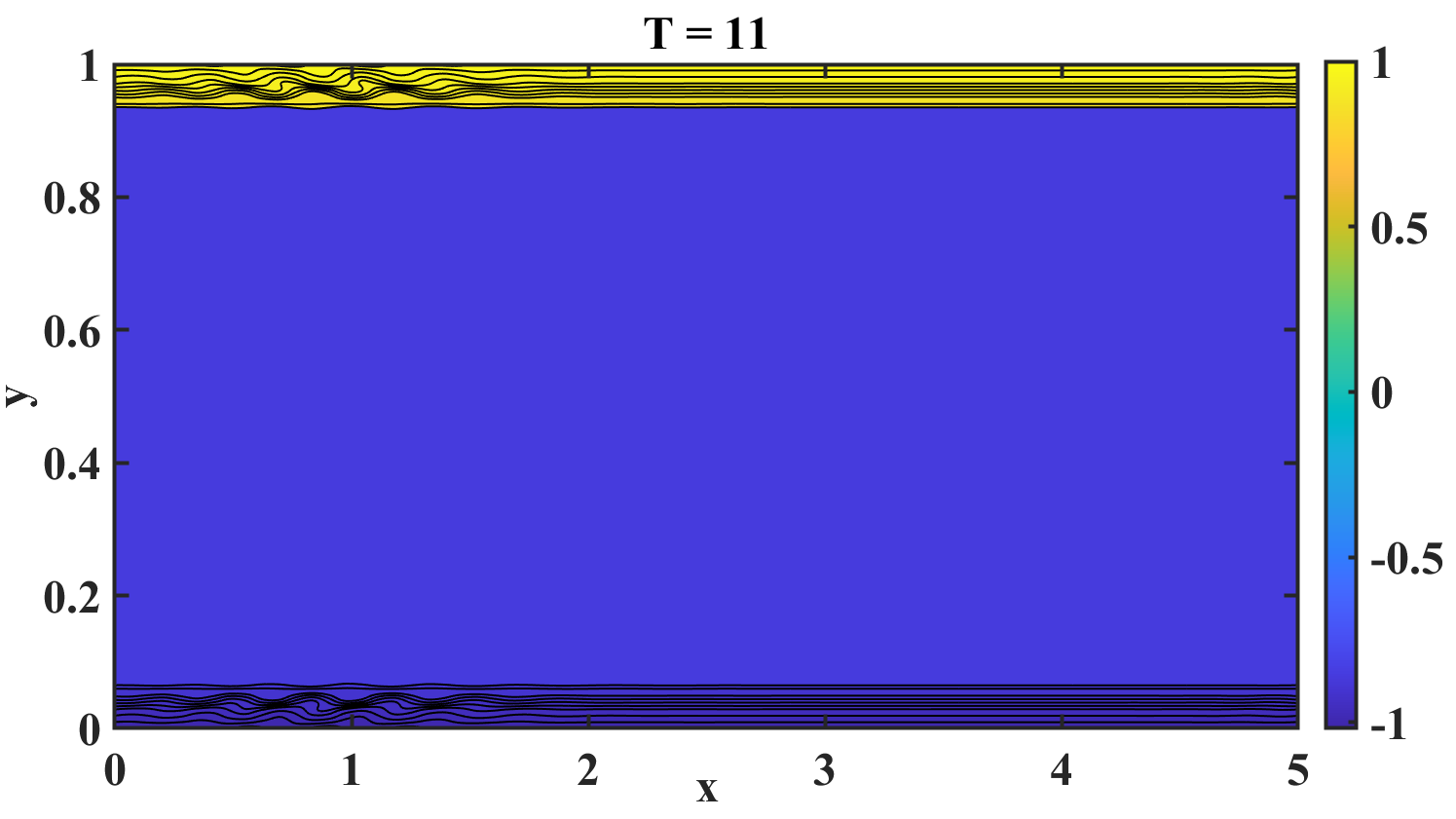}
\vskip 1pt
\includegraphics[width=0.49\linewidth, height=0.3\linewidth]{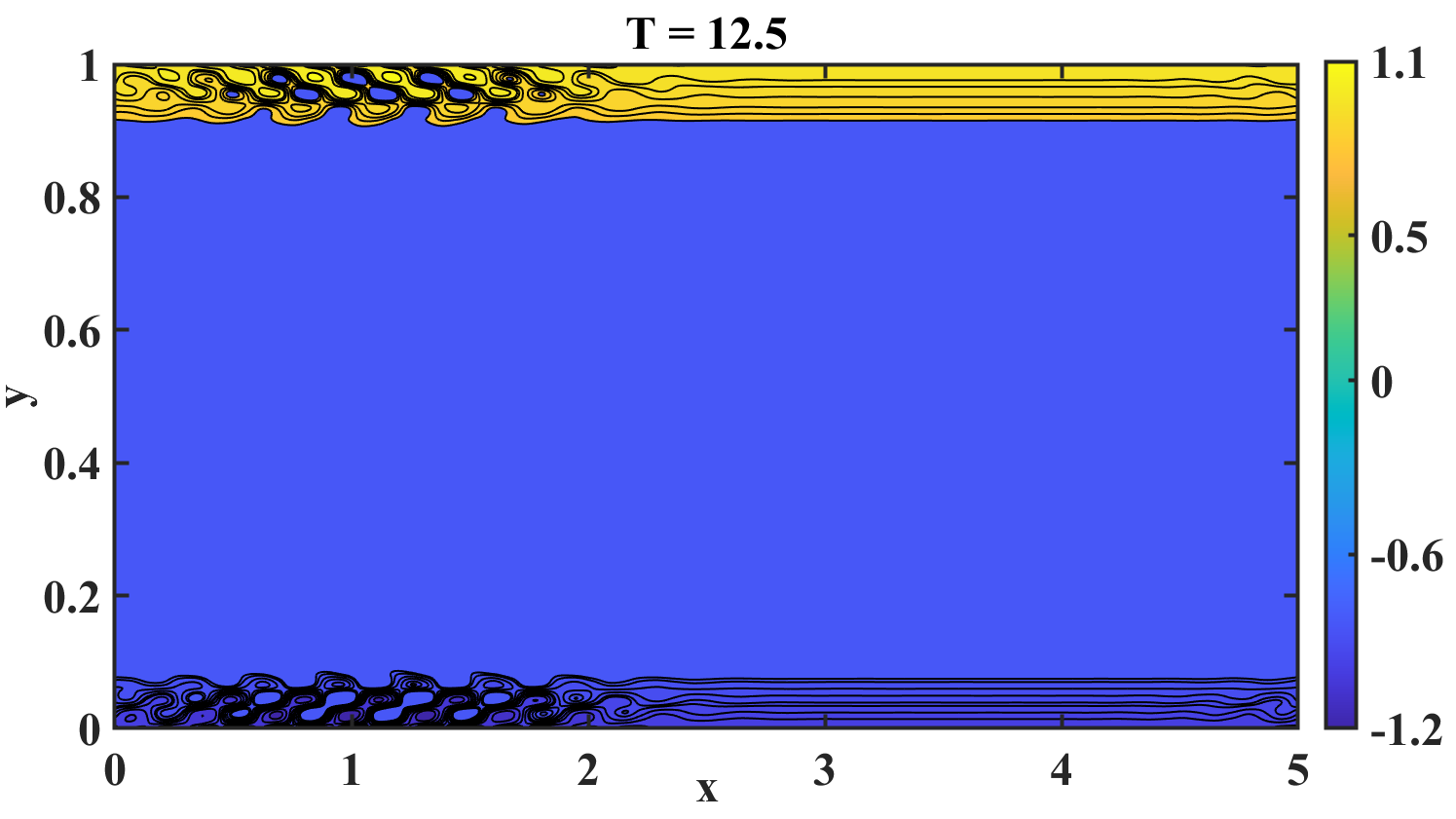}
\includegraphics[width=0.49\linewidth, height=0.3\linewidth]{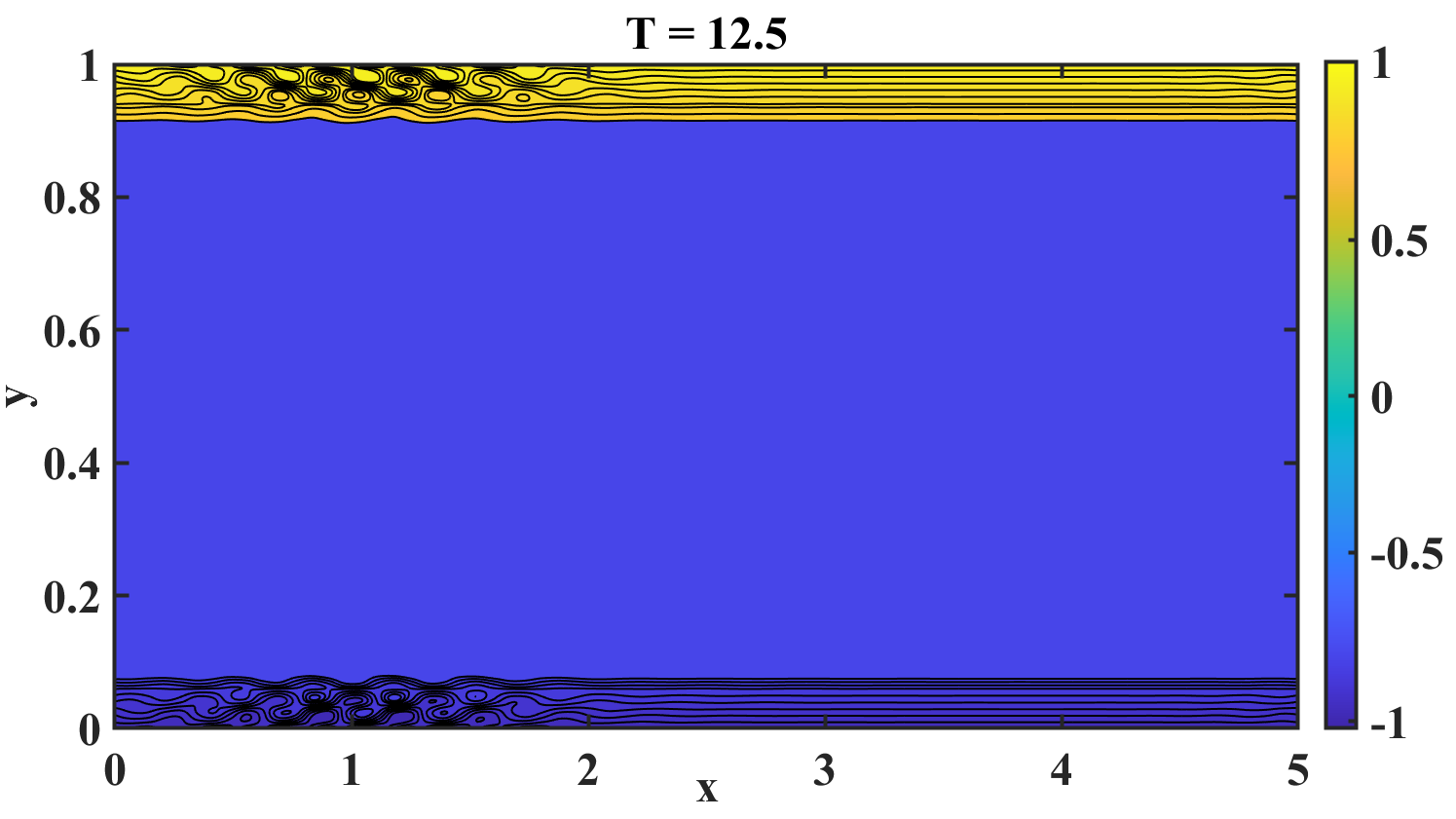}
\vskip 1pt
\includegraphics[width=0.49\linewidth, height=0.3\linewidth]{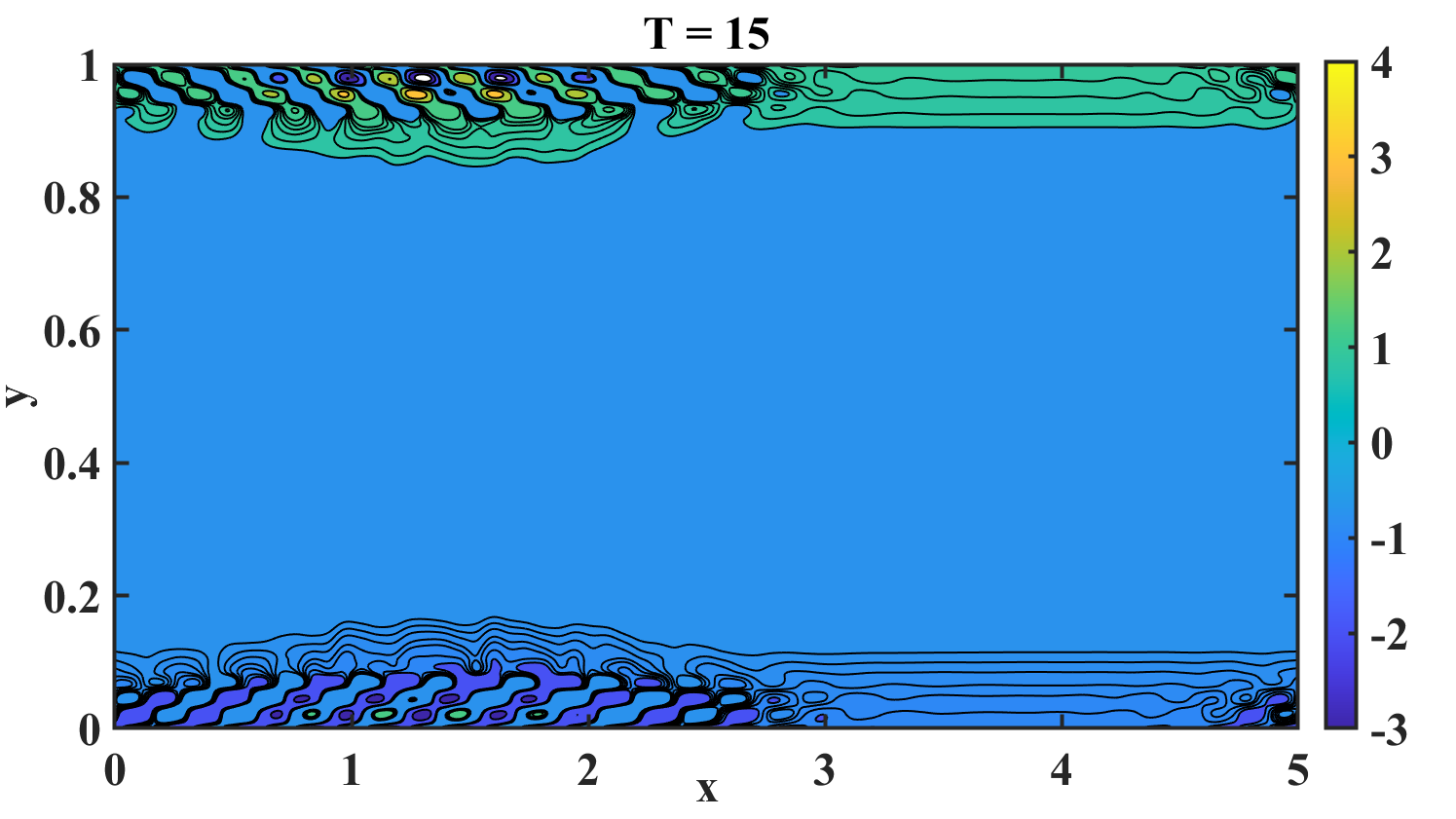}
\includegraphics[width=0.49\linewidth, height=0.3\linewidth]{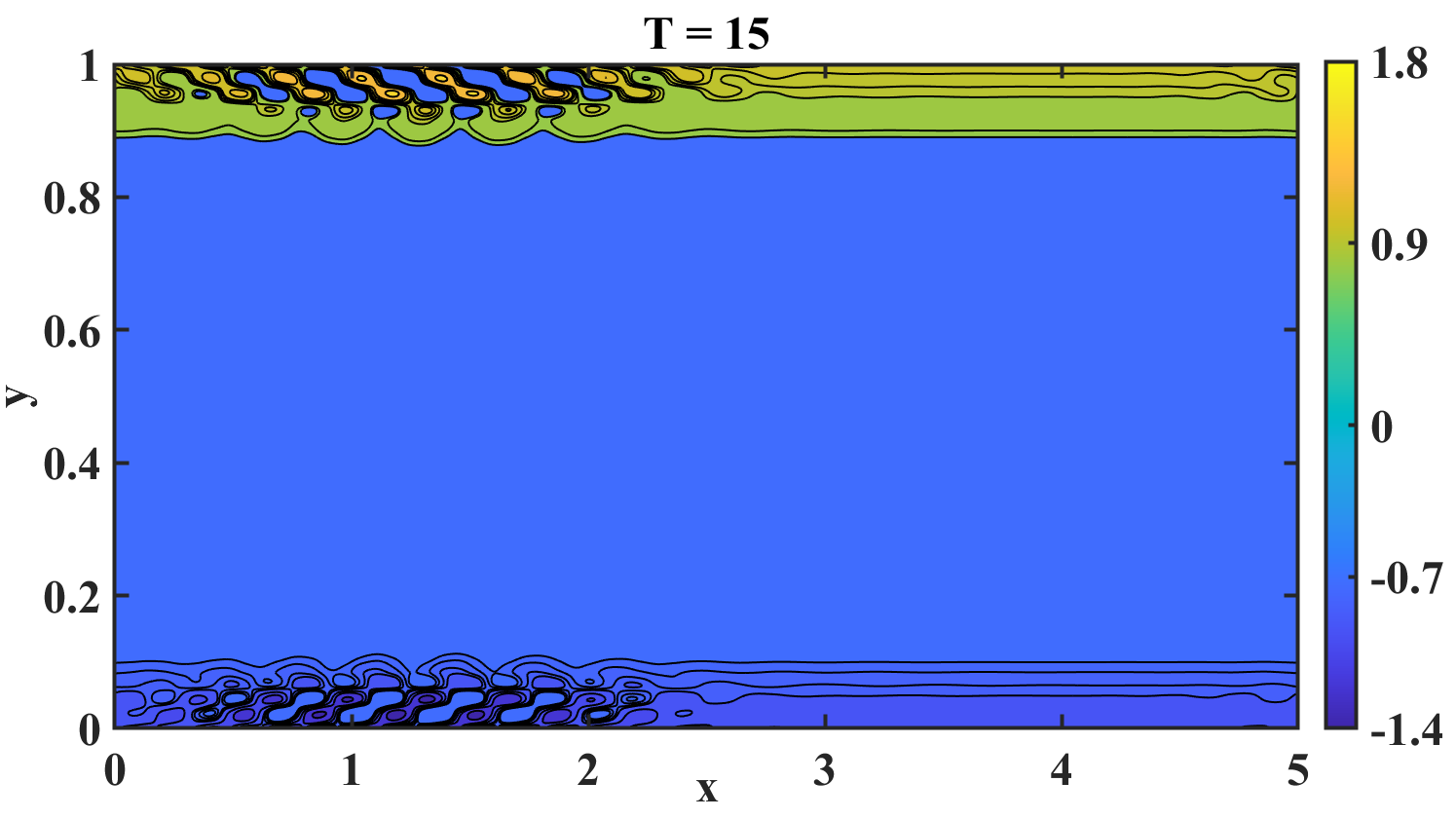}
\caption{Vorticity contours for the viscous stress-dominated ($\nu=0.6$) Rouse model ($\alpha=0.5$) case, shown at parameter values, $We=10.0$ and $Re=70$ (left column) versus $Re=1000$ (right column).}
\label{fig:Fig8}
\end{figure}

%%%%%%%%%%%%%%%%%%%%%%%%%%%%%%%%%%%%%%%%%%%%%%%%%%%%
\section{Conclusion} \label{sec:conclude}
This investigation addresses the development as well as the asymptotic and spectral stability of a novel class of numerical method for the spatiotemporal discretization of FPDE. Section~\ref{sec:method} presented the method, including the time integration (section~\ref{subsec:Time}) and the spatial discretization (section~\ref{subsec:Space}). Using 1D linear FADR equation, the asymptotic stability and spectral analysis was outlined in section~\ref{sec:anal}. The method was validated using the test cases for the 2D fraction diffusion equation, in section~\ref{sec:MV}. Section~\ref{sec:NS} described the numerical results of the subdiffusive dynamics of the viscoelastic channel flow. We conclude our discussion with a remark that the focus of this present work was on the development of the numerical method and not on the physical description of the subdiffusive flow dynamics. Hence, a comprehensive study on the mechanics of subdiffusive channel flow, using the numerical method developed in this work, is reported elsewhere. \vskip 5pt

%%%%%%%%%%%%%%%%%%%%%%%%%%%%%%%%%%%%%%%%%%%%%%%%%%%%
\noindent \textbf{Declaration of competing interest} The authors declare that they have no known competing financial interests or personal relationships that could have appeared to influence the work reported in this paper.

\bibliographystyle{elsarticle-harv} 
\bibliography{references.bib}

\begin{thebibliography}{36}
\expandafter\ifx\csname natexlab\endcsname\relax\def\natexlab#1{#1}\fi
\providecommand{\url}[1]{\texttt{#1}}
\providecommand{\href}[2]{#2}
\providecommand{\path}[1]{#1}
\providecommand{\DOIprefix}{doi:}
\providecommand{\ArXivprefix}{arXiv:}
\providecommand{\URLprefix}{URL: }
\providecommand{\Pubmedprefix}{pmid:}
\providecommand{\doi}[1]{\href{http://dx.doi.org/#1}{\path{#1}}}
\providecommand{\Pubmed}[1]{\href{pmid:#1}{\path{#1}}}
\providecommand{\bibinfo}[2]{#2}
\ifx\xfnm\relax \def\xfnm[#1]{\unskip,\space#1}\fi
%Type = Article
\bibitem[{Al-Khaled and Momani(2005)}]{Khaled2005}
\bibinfo{author}{Al-Khaled, K.}, \bibinfo{author}{Momani, S.},
  \bibinfo{year}{2005}.
\newblock \bibinfo{title}{An approximate solution for a fractional
  diffusion-wave equation using the decomposition method}.
\newblock \bibinfo{journal}{Applied Mathematics and Computation}
  \bibinfo{volume}{165}, \bibinfo{pages}{473--483}.
%Type = Article
\bibitem[{Ascher et~al.(1995)Ascher, Ruuth and Wetton}]{Ascher1995}
\bibinfo{author}{Ascher, U.M.}, \bibinfo{author}{Ruuth, S.J.},
  \bibinfo{author}{Wetton, B.T.R.}, \bibinfo{year}{1995}.
\newblock \bibinfo{title}{Implicit-explicit methods for time-dependent partial
  differential equations}.
\newblock \bibinfo{journal}{SIAM Journal on Numerical Analysis}
  \bibinfo{volume}{32}, \bibinfo{pages}{797--823}.
%Type = Article
\bibitem[{Brunner et~al.(2010)Brunner, Ling and Yamamoto}]{Brunner2010}
\bibinfo{author}{Brunner, H.}, \bibinfo{author}{Ling, L.},
  \bibinfo{author}{Yamamoto, M.}, \bibinfo{year}{2010}.
\newblock \bibinfo{title}{Numerical simulations of {2D} fractional subdiffusion
  problems}.
\newblock \bibinfo{journal}{Journal of Computational Physics}
  \bibinfo{volume}{229}, \bibinfo{pages}{6613--6622}.
%Type = Article
\bibitem[{Chauhan et~al.()Chauhan, Bansal and Sircar}]{Chauhan2021}
\bibinfo{author}{Chauhan, T.}, \bibinfo{author}{Bansal, D.},
  \bibinfo{author}{Sircar, S.}, .
\newblock \bibinfo{title}{Spatiotemporal linear stability of viscoelastic
  subdiffusive channel flows: a fractional calculus framework} \URLprefix
  \url{https://arxiv.org/abs/2301.02078}. \bibinfo{note}{arXive.org}.
%Type = Book
\bibitem[{Coffey et~al.(2004)Coffey, Kalmykov and Waldron}]{Coffey2004}
\bibinfo{author}{Coffey, W.T.}, \bibinfo{author}{Kalmykov, Y.P.},
  \bibinfo{author}{Waldron, J.T.}, \bibinfo{year}{2004}.
\newblock \bibinfo{title}{The {Langevin} {Equation}: {With} {Applications} to
  {Stochastic} {Problems} in {Physics}, {Chemistry} and {Electrical}
  {Engineering}}. volume~\bibinfo{volume}{14} of \textit{\bibinfo{series}{World
  {Scientific} {Series} in {Contemporary} {Chemical} {Physics}}}.
\newblock \bibinfo{edition}{2nd} ed., \bibinfo{publisher}{World Scientific
  Publishing Company}.
%Type = Article
\bibitem[{Crouzeix(1980)}]{Crouzeix1980}
\bibinfo{author}{Crouzeix, M.}, \bibinfo{year}{1980}.
\newblock \bibinfo{title}{Une m{\'e}thode multipas implicite-explicite pour
  l'approximation des {\'e}quations d'{\'e}volution paraboliques}.
\newblock \bibinfo{journal}{Numerische Mathematik} \bibinfo{volume}{35},
  \bibinfo{pages}{257--276}.
%Type = Article
\bibitem[{Diethelm and Freed(2006)}]{Diethelm2006}
\bibinfo{author}{Diethelm, K.}, \bibinfo{author}{Freed, A.},
  \bibinfo{year}{2006}.
\newblock \bibinfo{title}{An efficient algorithm for the evaluation of
  convolution integrals}.
\newblock \bibinfo{journal}{Computers \& Mathematics with Applications}
  \bibinfo{volume}{51}, \bibinfo{pages}{51--72}.
%Type = Book
\bibitem[{Doi(1996)}]{Doi1996}
\bibinfo{author}{Doi, M.}, \bibinfo{year}{1996}.
\newblock \bibinfo{title}{An {Introduction} to {Polymer} {Physics}}.
\newblock \bibinfo{publisher}{Clarendon Press}.
%Type = Article
\bibitem[{Fogelson and Neeves(2015)}]{Fogelson2015}
\bibinfo{author}{Fogelson, A.L.}, \bibinfo{author}{Neeves, K.B.},
  \bibinfo{year}{2015}.
\newblock \bibinfo{title}{Fluid {Mechanics} of {Blood} {Clot} {Formation}}.
\newblock \bibinfo{journal}{Annual Review of Fluid Mechanics}
  \bibinfo{volume}{47}, \bibinfo{pages}{377--403}.
%Type = Article
\bibitem[{Goychuk et~al.(2017)Goychuk, Kharchenko and Metzler}]{Goychuk2017}
\bibinfo{author}{Goychuk, I.}, \bibinfo{author}{Kharchenko, V.O.},
  \bibinfo{author}{Metzler, R.}, \bibinfo{year}{2017}.
\newblock \bibinfo{title}{Persistent {Sinai}-type diffusion in {Gaussian}
  random potentials with decaying spatial correlations}.
\newblock \bibinfo{journal}{Physical Review E} \bibinfo{volume}{96},
  \bibinfo{pages}{052134}.
%Type = Article
\bibitem[{Goychuk and Pöschel(2020)}]{Goychuk2020}
\bibinfo{author}{Goychuk, I.}, \bibinfo{author}{Pöschel, T.},
  \bibinfo{year}{2020}.
\newblock \bibinfo{title}{Hydrodynamic memory can boost enormously driven
  nonlinear diffusion and transport}.
\newblock \bibinfo{journal}{Physical Review E} \bibinfo{volume}{102},
  \bibinfo{pages}{012139}.
%Type = Article
\bibitem[{Goychuk and Pöschel(2021)}]{Goychuk2021}
\bibinfo{author}{Goychuk, I.}, \bibinfo{author}{Pöschel, T.},
  \bibinfo{year}{2021}.
\newblock \bibinfo{title}{Fingerprints of viscoelastic subdiffusion in random
  environments: {Revisiting} some experimental data and their interpretations}.
\newblock \bibinfo{journal}{Physical Review E} \bibinfo{volume}{104},
  \bibinfo{pages}{034125}.
%Type = Article
\bibitem[{Jannelli(2022)}]{Jannelli2022}
\bibinfo{author}{Jannelli, A.}, \bibinfo{year}{2022}.
\newblock \bibinfo{title}{Adaptive numerical solutions of time-fractional
  advection–diffusion–reaction equations}.
\newblock \bibinfo{journal}{Communications in Nonlinear Science and Numerical
  Simulation} \bibinfo{volume}{105}, \bibinfo{pages}{106073}.
%Type = Article
\bibitem[{Khalid et~al.(2021)Khalid, Chaudhary, Garg, Shankar and
  Subramanian}]{Khalid2021}
\bibinfo{author}{Khalid, M.}, \bibinfo{author}{Chaudhary, I.},
  \bibinfo{author}{Garg, P.}, \bibinfo{author}{Shankar, V.},
  \bibinfo{author}{Subramanian, G.}, \bibinfo{year}{2021}.
\newblock \bibinfo{title}{The centre-mode instability of viscoelastic plane
  {Poiseuille} flow}.
\newblock \bibinfo{journal}{Journal of Fluid Mechanics} \bibinfo{volume}{915},
  \bibinfo{pages}{A43}.
%Type = Article
\bibitem[{Kirkwood(1954)}]{Kirkwood1954}
\bibinfo{author}{Kirkwood, J.G.}, \bibinfo{year}{1954}.
\newblock \bibinfo{title}{The general theory of irreversible processes in
  solutions of macromolecules}.
\newblock \bibinfo{journal}{Journal of Polymer Science} \bibinfo{volume}{12},
  \bibinfo{pages}{1--14}.
%Type = Article
\bibitem[{Kou and Xie(2004)}]{Kou2004}
\bibinfo{author}{Kou, S.C.}, \bibinfo{author}{Xie, X.S.}, \bibinfo{year}{2004}.
\newblock \bibinfo{title}{Generalized {Langevin} {Equation} with {Fractional}
  {Gaussian} {Noise}: {Subdiffusion} within a {Single} {Protein} {Molecule}}.
\newblock \bibinfo{journal}{Physical Review Letters} \bibinfo{volume}{93},
  \bibinfo{pages}{180603}.
%Type = Article
\bibitem[{Kremer and Grest(1990)}]{Kremer1990}
\bibinfo{author}{Kremer, K.}, \bibinfo{author}{Grest, G.S.},
  \bibinfo{year}{1990}.
\newblock \bibinfo{title}{Dynamics of entangled linear polymer melts: {A}
  molecular‐dynamics simulation}.
\newblock \bibinfo{journal}{The Journal of Chemical Physics}
  \bibinfo{volume}{92}, \bibinfo{pages}{5057--5086}.
%Type = Article
\bibitem[{Lai et~al.(2009)Lai, Wang, Cone, Wirtz and Hanes}]{Lai2009}
\bibinfo{author}{Lai, S.K.}, \bibinfo{author}{Wang, Y.Y.},
  \bibinfo{author}{Cone, R.}, \bibinfo{author}{Wirtz, D.},
  \bibinfo{author}{Hanes, J.}, \bibinfo{year}{2009}.
\newblock \bibinfo{title}{Altering {Mucus} {Rheology} to “{Solidify}”
  {Human} {Mucus} at the {Nanoscale}}.
\newblock \bibinfo{journal}{PLoS ONE} \bibinfo{volume}{4},
  \bibinfo{pages}{e4294}.
%Type = Article
\bibitem[{Levine and Lubensky(2001)}]{Levine2001}
\bibinfo{author}{Levine, A.J.}, \bibinfo{author}{Lubensky, T.C.},
  \bibinfo{year}{2001}.
\newblock \bibinfo{title}{Response function of a sphere in a viscoelastic
  two-fluid medium}.
\newblock \bibinfo{journal}{Physical Review E} \bibinfo{volume}{63},
  \bibinfo{pages}{041510}.
%Type = Article
\bibitem[{Murio(2008)}]{Murio2008}
\bibinfo{author}{Murio, D.A.}, \bibinfo{year}{2008}.
\newblock \bibinfo{title}{Implicit finite difference approximation for time
  fractional diffusion equations}.
\newblock \bibinfo{journal}{Computers \& Mathematics with Applications}
  \bibinfo{volume}{56}, \bibinfo{pages}{1138--1145}.
%Type = Article
\bibitem[{Nandagopalan et~al.(2018)Nandagopalan, John, Baek, Miglani and
  Ardhianto}]{Nandagopalan2018}
\bibinfo{author}{Nandagopalan, P.}, \bibinfo{author}{John, J.},
  \bibinfo{author}{Baek, S.W.}, \bibinfo{author}{Miglani, A.},
  \bibinfo{author}{Ardhianto, K.}, \bibinfo{year}{2018}.
\newblock \bibinfo{title}{Shear-flow rheology and viscoelastic instabilities of
  ethanol gel fuels}.
\newblock \bibinfo{journal}{Experimental Thermal and Fluid Science}
  \bibinfo{volume}{99}, \bibinfo{pages}{181--189}.
%Type = Book
\bibitem[{Podlubny(1999)}]{Podlubny1999}
\bibinfo{author}{Podlubny, I.}, \bibinfo{year}{1999}.
\newblock \bibinfo{title}{Fractional differential equations: an introduction to
  fractional derivatives, fractional differential equations, to methods of
  their solution and some of their applications}.
\newblock Number \bibinfo{number}{v. 198} in \bibinfo{series}{Mathematics in
  science and engineering}, \bibinfo{publisher}{Academic Press},
  \bibinfo{address}{San Diego}.
%Type = Article
\bibitem[{Riley et~al.(1988)Riley, Gad-el Hak and Metcalfe}]{Riley1988}
\bibinfo{author}{Riley, J.J.}, \bibinfo{author}{Gad-el Hak, M.},
  \bibinfo{author}{Metcalfe, R.W.}, \bibinfo{year}{1988}.
\newblock \bibinfo{title}{Complaint {Coatings}}.
\newblock \bibinfo{journal}{Annual Review of Fluid Mechanics}
  \bibinfo{volume}{20}, \bibinfo{pages}{393--420}.
%Type = Article
\bibitem[{Rouse(1953)}]{Rouse1953}
\bibinfo{author}{Rouse, P.E.}, \bibinfo{year}{1953}.
\newblock \bibinfo{title}{A {Theory} of the {Linear} {Viscoelastic}
  {Properties} of {Dilute} {Solutions} of {Coiling} {Polymers}}.
\newblock \bibinfo{journal}{The Journal of Chemical Physics}
  \bibinfo{volume}{21}, \bibinfo{pages}{1272--1280}.
%Type = Book
\bibitem[{Rubenstein and Colby(2003)}]{Rubenstein2003}
\bibinfo{author}{Rubenstein, M.}, \bibinfo{author}{Colby, R.H.},
  \bibinfo{year}{2003}.
\newblock \bibinfo{title}{Polymer {Physics}}.
\newblock \bibinfo{publisher}{Oxford Univ. Press}, \bibinfo{address}{New York}.
%Type = Article
\bibitem[{Sengupta et~al.(2012)Sengupta, Bhumkar, Rajpoot, Suman and
  Saurabh}]{Sengupta2012}
\bibinfo{author}{Sengupta, T.K.}, \bibinfo{author}{Bhumkar, Y.G.},
  \bibinfo{author}{Rajpoot, M.K.}, \bibinfo{author}{Suman, V.},
  \bibinfo{author}{Saurabh, S.}, \bibinfo{year}{2012}.
\newblock \bibinfo{title}{Spurious waves in discrete computation of wave
  phenomena and flow problems}.
\newblock \bibinfo{journal}{Applied Mathematics and Computation}
  \bibinfo{volume}{218}, \bibinfo{pages}{9035--9065}.
%Type = Article
\bibitem[{Sengupta et~al.(2006)Sengupta, Sircar and Dipankar}]{Sircar2006}
\bibinfo{author}{Sengupta, T.K.}, \bibinfo{author}{Sircar, S.K.},
  \bibinfo{author}{Dipankar, A.}, \bibinfo{year}{2006}.
\newblock \bibinfo{title}{High {Accuracy} {Schemes} for {DNS} and {Acoustics}}.
\newblock \bibinfo{journal}{Journal of Scientific Computing}
  \bibinfo{volume}{26}, \bibinfo{pages}{151--193}.
%Type = Article
\bibitem[{Singh et~al.(2020)Singh, Bansal, Kaur and Sircar}]{Sircar2020}
\bibinfo{author}{Singh, S.}, \bibinfo{author}{Bansal, D.},
  \bibinfo{author}{Kaur, G.}, \bibinfo{author}{Sircar, S.},
  \bibinfo{year}{2020}.
\newblock \bibinfo{title}{Implicit-explicit-compact methods for advection
  diffusion reaction equations}.
\newblock \bibinfo{journal}{Computers \& Fluids} \bibinfo{volume}{212},
  \bibinfo{pages}{104709}.
%Type = Article
\bibitem[{Sircar et~al.(2015)Sircar, Aisenbrey, Bryant and Bortz}]{Sircar2015}
\bibinfo{author}{Sircar, S.}, \bibinfo{author}{Aisenbrey, E.},
  \bibinfo{author}{Bryant, S.}, \bibinfo{author}{Bortz, D.},
  \bibinfo{year}{2015}.
\newblock \bibinfo{title}{Determining equilibrium osmolarity in poly(ethylene
  glycol)/chondrotin sulfate gels mimicking articular cartilage}.
\newblock \bibinfo{journal}{Journal of Theoretical Biology}
  \bibinfo{volume}{364}, \bibinfo{pages}{397--406}.
%Type = Article
\bibitem[{Sircar and Bansal(2019)}]{Sircar2019}
\bibinfo{author}{Sircar, S.}, \bibinfo{author}{Bansal, D.},
  \bibinfo{year}{2019}.
\newblock \bibinfo{title}{Spatiotemporal linear stability of viscoelastic free
  shear flows: {Dilute} regime}.
\newblock \bibinfo{journal}{Physics of Fluids} \bibinfo{volume}{31},
  \bibinfo{pages}{084104}.
%Type = Article
\bibitem[{Sircar and Roberts(2016)}]{Sircar2016}
\bibinfo{author}{Sircar, S.}, \bibinfo{author}{Roberts, A.J.},
  \bibinfo{year}{2016}.
\newblock \bibinfo{title}{Surface deformation and shear flow in ligand mediated
  cell adhesion}.
\newblock \bibinfo{journal}{Journal of Mathematical Biology}
  \bibinfo{volume}{73}, \bibinfo{pages}{1035--1052}.
%Type = Article
\bibitem[{Sircar and Wang(2010)}]{Sircar2010}
\bibinfo{author}{Sircar, S.}, \bibinfo{author}{Wang, Q.}, \bibinfo{year}{2010}.
\newblock \bibinfo{title}{Transient rheological responses in sheared biaxial
  liquid crystals}.
\newblock \bibinfo{journal}{Rheologica Acta} \bibinfo{volume}{49},
  \bibinfo{pages}{699--717}.
%Type = Article
\bibitem[{Visbal and Gaitonde(2002)}]{Visbal2002}
\bibinfo{author}{Visbal, M.R.}, \bibinfo{author}{Gaitonde, D.V.},
  \bibinfo{year}{2002}.
\newblock \bibinfo{title}{On the {Use} of {Higher}-{Order}
  {Finite}-{Difference} {Schemes} on {Curvilinear} and {Deforming} {Meshes}}.
\newblock \bibinfo{journal}{Journal of Computational Physics}
  \bibinfo{volume}{181}, \bibinfo{pages}{155--185}.
%Type = Phdthesis
\bibitem[{Zarabadi(2019)}]{Zarabadi2019}
\bibinfo{author}{Zarabadi, M.}, \bibinfo{year}{2019}.
\newblock \bibinfo{title}{Development of a {Robust} {Microfluidic}
  {Electrochemical} {Cell} for {Biofilm} {Study} in {Controlled} {Hydrodynamic}
  {Conditions}}.
\newblock Ph.D. thesis. Univ. Laval,.
%Type = Article
\bibitem[{Zarabadi et~al.(2018)Zarabadi, Charette and Greener}]{Zarabadi2018}
\bibinfo{author}{Zarabadi, M.P.}, \bibinfo{author}{Charette, S.J.},
  \bibinfo{author}{Greener, J.}, \bibinfo{year}{2018}.
\newblock \bibinfo{title}{Flow-{Based} {Deacidification} of \textit{{Geobacter}
  sulfurreducens} {Biofilms} {Depends} on {Nutrient} {Conditions}: a
  {Microfluidic} {Bioelectrochemical} {Study}}.
\newblock \bibinfo{journal}{ChemElectroChem} \bibinfo{volume}{5},
  \bibinfo{pages}{3645--3653}.
%Type = Article
\bibitem[{Zimm(1956)}]{Zimm1956}
\bibinfo{author}{Zimm, B.H.}, \bibinfo{year}{1956}.
\newblock \bibinfo{title}{Dynamics of {Polymer} {Molecules} in {Dilute}
  {Solution}: {Viscoelasticity}, {Flow} {Birefringence} and {Dielectric}
  {Loss}}.
\newblock \bibinfo{journal}{The Journal of Chemical Physics}
  \bibinfo{volume}{24}, \bibinfo{pages}{269--278}.

\end{thebibliography}





% \end{thebibliography}
\end{document}